\documentclass[12pt,letterpaper]{article}

\usepackage{etoolbox}
\newtoggle{SUPPLEMENTAL}\toggletrue{SUPPLEMENTAL}
\togglefalse{SUPPLEMENTAL} 
\newcommand{\Supplemental}[2]{\iftoggle{SUPPLEMENTAL}{#1}{#2}}
\newtoggle{BLINDED}\toggletrue{BLINDED}
\newcommand{\Blinded}[2]{\iftoggle{BLINDED}{#1}{#2}}

\usepackage[margin=1in,letterpaper]{geometry}
\usepackage[T1]{fontenc}
\usepackage[group-separator={\,}]{siunitx} 
\usepackage{graphicx,grffile}
\usepackage{amsmath,amssymb,amsthm}
\usepackage{mathtools,dsfont,thmtools}
\usepackage{caption} 
\usepackage{booktabs}
\usepackage[flushleft]{threeparttable}
\makeatletter 
   \g@addto@macro\TPT@defaults{\linespread{1}\footnotesize} 
\makeatother
\usepackage{listings,textcomp,verbatim}
\usepackage[longnamesfirst]{natbib}
\usepackage{alltt}
\usepackage{hypernat}
\usepackage[nodisplayskipstretch]{setspace}
\usepackage{enumerate}

\usepackage[bottom]{footmisc}
\usepackage{xr} 
\usepackage{hyperref}
\usepackage{bookmark}
\usepackage[capitalize,noabbrev]{cleveref}
\usepackage{multirow} 

  \renewenvironment{thebibliography}[1]%
  {\begin{oldthebibliography}{#1}\setlength{\parskip}{0ex}\setlength{\itemsep}{0ex}}%
  {\end{oldthebibliography}}

\DeclareGraphicsExtensions{.pdf,.png}
\graphicspath{ {./Figures/} }

\Crefname{assumption}{Assumption}{Assumptions}

\crefname{equation}{}{} 
\Crefname{equation}{Equation}{Equations} 
\Crefname{method}{Method}{Methods}
\Crefname{conjecture}{Conjecture}{Conjectures}
\crefname{conjecture}{Conjecture}{Conjectures}
\Crefname{fact}{Fact}{Facts}
\crefname{fact}{Fact}{Facts}
\Crefname{enumi}{Task}{Tasks}
\crefname{enumi}{Task}{Tasks}

\newcommand{\independenT}[2]{\mathrel{\rlap{$#1#2$}\mkern2mu{#1#2}}}
\newcommand\independent{\protect\mathpalette{\protect\independenT}{\perp}} 

\theoremstyle{plain}
\newtheorem{theorem}{Theorem}
\newtheorem{lemma}[theorem]{Lemma}
\newtheorem{fact}[theorem]{Fact}
\newtheorem{proposition}[theorem]{Proposition}

\theoremstyle{definition}
\newtheorem{assumption}{Assumption}
\newtheorem{definition}{Definition}
\declaretheorem[style=definition,qed=$\qedsymbol$]{method} 

  {\begin{itemize}%
    \setlength{\itemsep}{0pt}%
    \setlength{\parskip}{0pt}}%
  {\end{itemize}}
  
\newenvironment{enumeratecomp}[1][i.]%
  {\begin{enumerate}[#1]%
    \setlength{\itemsep}{0pt}%
    \setlength{\parskip}{0pt}}%
  {\end{enumerate}}

\newcommand{\matf}[1]{\underline{\boldsymbol{\mathbf{#1}}}} 
\newcommand{\vecf}[1]{\boldsymbol{\mathbf{#1}}} 
\DeclareMathOperator{\Var}{Var}

\newcommand{\R}{\mathbb{R}}

\newcommand{\FWER}{\mathrm{FWER}}
\newcommand{\BetaDist}{\mathrm{Beta}}
\newcommand{\UnifDist}{\mathrm{Unif}}

\let\Pr\relax \DeclareMathOperator{\Pr}{P} 

\DeclareMathOperator{\1}{\mathds{1}}
\newcommand{\Ind}[1]{\1\left\{#1\right\}}
\newcommand{\Normal}{\mathrm{N}}
\newcommand{\Normalp}[2]{\Normal\left(#1,#2\right)}

\newcommand{\dconv}{\stackrel{d}{\to}}

\providecommand{\abs}[1]{\lvert#1\rvert}
\providecommand{\norm}[1]{\lVert#1\rVert}

\let\originalleft\left
\let\originalright\right
\renewcommand{\left}{\mathopen{}\mathclose\bgroup\originalleft}
\renewcommand{\right}{\aftergroup\egroup\originalright}

\newcommand{\mockalph}[1]{}  

\newcounter{allenumi}

\def\calibdecerrfiveten{0.02} 
\def\calibdecerrone{0.05} 
\def\NDRAWSfiveten{2\times10^5}
\def\NDRAWSone{10^6}
\def\caliberrfive{0.001} 

\allowdisplaybreaks[3]

\begin{document}

\title{Comparing distributions by multiple testing \\ across quantiles or {CDF} values}

\author{\Blinded{Matt Goldman\thanks{\Supplemental{}{Microsoft Research, }mattgold@microsoft.com.} \and David M.\ Kaplan\thanks{\Supplemental{}{Corresponding author.  Department of Economics, University of Missouri, }kaplandm@missouri.edu. 
\Supplemental{}{%
We thank the referees and editors for helpful revision ideas that improved the content and presentation. 
For other helpful feedback, we thank Zack Miller and participants at the 2016 Midwest Econometrics Group and 2017 NY Camp Econometrics. 
We thank Boaz Nadler and Amit Moscovich for alerting us to important related works from statistics and computer science.%
}%
}}{[BLINDED]\Supplemental{}{\thanks{%
}}}}
\date{\today}

\maketitle


\doublespacing
\singlespacing

\begin{abstract}
\Supplemental{This supplement includes 
XXX.}{%
When comparing two distributions, it is often helpful to learn at which quantiles or values there is a statistically significant difference.  
This provides more information than the binary ``reject'' or ``do not reject'' 
decision of a global goodness-of-fit test. 
Framing our question as multiple testing across the continuum of quantiles $\tau\in(0,1)$ or values $r\in\R$, we show that the Kolmogorov--Smirnov test (interpreted as a multiple testing procedure) achieves strong control of the familywise error rate. 
However, its well-known flaw of low sensitivity in the tails remains. 
We provide an alternative method that retains such strong control of familywise error rate while also having even sensitivity, i.e., equal pointwise type I error rates at each of $n\to\infty$ order statistics across the distribution. 
Our one-sample method computes instantly, using our new formula that also instantly computes goodness-of-fit $p$-values and uniform confidence bands. 
To improve power, we also propose stepdown and pre-test procedures that maintain control of the asymptotic familywise error rate. 
One-sample 
and two-sample 
cases are considered, as well as extensions to regression discontinuity designs and conditional distributions. 
Simulations, empirical examples, and code are provided.

\vspace{\baselineskip}

\vspace{\baselineskip}
}
\end{abstract}

\Blinded{\newpage}{}

\doublespacing
\onehalfspacing

\Supplemental{\vfill\pagebreak}{}

\section{Introduction}\label{sec:intro}

Increasingly, economists compare not only means, but entire distributions. 
This includes comparing income distributions (over two time periods, geographic areas, or demographic subpopulations) and a variety of economic outcomes in experiments and program evaluation, either comparing unknown ``treated'' and ``untreated'' distributions (i.e., two-sample inference), or comparing an unknown ``treated'' distribution to a known population distribution (i.e., one-sample inference). 
For example, in a paper garnering 308 citations to date,\footnote{Source: Google Scholar, July 23, 2017.} \citet{BitlerEtAl2006} study a Connecticut welfare reform program using quantile treatment effects on earnings and other outcomes (Section V), arguing that mean effects miss the substantial amount of economically important heterogeneity. 
In other recent empirical work using quantile treatment effects, 
\citet{DjebbariSmith2008} study the PROGRESA conditional cash transfer program in Mexico, 
\citet{JacksonPage2013} study the Tennessee STAR class size reduction data, and 
\citet{BitlerEtAl2008} study a Canadian welfare reform experiment. 
Our new methodology for comparing distributions complements the quantile treatment effect methods used in these papers. 

To compare distributions, the most common statistical tests answer one of two questions: 
1) Are the distributions identical or different? 
2) Do the distributions differ at the median (or another pre-specified quantile)?  
Often, the question with the most economic and policy relevance is instead: 
3) Across the entire distribution, at which quantiles or values do the distributions differ? 
For example, one may care not just \emph{whether} two (sub)populations have different income distributions, but \emph{where} (at which quantiles or values) they differ.  
In an experimental setting, the question is at which values the treatment effect is statistically significant; see \cref{sec:emp} for an empirical example. 

We contribute to answering question (3). 
First, we formalize the question as multiple testing of a continuum either of CDF hypotheses at different values ($r\in\R$) or quantile hypotheses at different quantile indices ($\tau\in(0,1)$), which we call ``quantile multiple testing.'' 
This framework appears to be new to the literature on distributional inference. 
Second, we show that the Kolmogorov--Smirnov (KS) test can be interpreted as answering question (3) and that it appropriately controls the probability of having at least one false rejection, i.e., controls the familywise error rate (FWER).%
\footnote{There are other reasonable ways to quantify control of type I errors for multiple testing, such as the false discovery rate (FDR) of \citet{BenjaminiHochberg1995} and the $k$-FWER and false discovery proportion (FDP) of \citet{LehmannRomano2005fwer}.} 
Third, we propose a new approach to answer question (3). 
Like the KS, our approach is nonparametric and appropriately controls FWER, without being conservative (like the Bonferroni method). 
Unlike the KS, our approach maintains ``even sensitivity'' (as quantified by pointwise size) across the continuum of quantile hypotheses. 
This addresses the long noted problem of the KS test's poor power against deviations in the tails \citep[e.g.,][p.\ 117]{Eicker1979}. 
Fourth, we provide a new formula to instantly compute our method as well as related goodness-of-fit $p$-values and uniform confidence bands for an unknown CDF. 
Fifth, we refine our basic method with stepdown and pre-test procedures to improve power without sacrificing strong control of FWER. 

Question (3) cannot necessarily be answered by methods addressing questions (1) or (2). 
The answer to question (1) is only ``yes'' or ``no.'' 
Using a method that answers question (2), if separate hypothesis tests are run at many different quantiles, each with size $\alpha$, then the well-known multiple testing problem is that the overall probability of making any false rejection (a.k.a.\ ``false discovery'') is above $\alpha$; e.g., see \citet{RomanoEtAl2010}. 
Even with two identical distributions (so all quantiles are identical), 
such a naive procedure may falsely reject equality for at least one quantile $30\%$ of the time even if $\alpha=5\%$. 
This $30\%$ overall false rejection probability (formalized later) is the FWER. 
Even a more sophisticated ``multiple quantile treatment effect'' approach like in the aforementioned \citet{BitlerEtAl2006} has drawbacks: 
a) the choice of quantiles is arbitrary, 
b) standard asymptotic approximations break down in the tails (where extremal quantile methods are required), and 
c) no finite-sample results are available. 
Our approach addresses all three of these problems. 

The KS test \citep{Kolmogorov1933,Smirnov1939,Smirnov1948} 
is a goodness-of-fit (GOF) test that is only intended to answer question (1) above, 
but we show that it readily identifies a set of values at which the population and null distribution functions differ 
in a way that controls FWER. 
We call this the ``KS-based'' multiple testing procedure to distinguish it from the ``KS test'' for GOF. 
Using other GOF approaches like Cram\'{e}r--von Mises, Anderson--Darling, and permutation tests, identifying \emph{where} two distributions differ is not possible. 

The KS test, however, is known to suffer from low sensitivity (i.e., low power) in the tails.  
More specifically, the allocation of sensitivity is uneven: it is concentrated in the middle of the distribution, as has been discussed formally in the literature since (at least) \citet{Jaeschke1979,Eicker1979}. 
The KS test also has an uneven distribution of sensitivity to deviations above versus below the null distribution at any point $x$ away from the median. 

The KS test's uneven sensitivity can lead to obviously incorrect inferences. 
For example, let sample size $n=20$, with null hypothesis distribution $\UnifDist(0,1)$.  
Even if five of the $20$ observations exceed one million, any of which alone clearly contradicts the null hypothesis, the KS test still fails to reject at a $10\%$ level.%
\footnote{R code: 
\texttt{ks.test(c(1:15/21,10\string^6+1:5),punif)} results in \texttt{D = 0.25, p-value = 0.1376}.} 
This is not a small-sample issue: with $500$ of $n=\num{200000}$ observations exceeding one million, the KS still fails to reject at a $10\%$ level.\footnote{R code: 
\texttt{n=200000;k=500;ks.test(c(1:(n-k)/(n+1),10\string^6+1:k),punif)} results in \texttt{D = 0.0025, p-value = 0.1641}.}

For GOF testing, i.e., testing the global null $H_0 \colon F(\cdot)=F_0(\cdot)$, as well as uniform confidence bands, 
\citet{BujaRolke2006} appear to be the first to achieve ``even sensitivity'' by using the probability integral transform. 
The probability integral transform reduces the problem to order statistics from a standard uniform distribution, whose finite-sample distribution is known. 
Their one-sample uniform confidence band (from Section 5.1)\ was eventually detailed and published by \citet{AldorNoimanEtAl2013}. 
\citet{BujaRolke2006} also construct a two-sample permutation test for equality. 
Although computation of our one-sample procedure is equivalent to computation of their uniform band, they do not propose any multiple testing procedure or discuss FWER, and our two-sample procedure is entirely new. 
Other papers have explored implications of the same probability integral transform, such as \citet{MoscovichEtAl2016} and \citet{MoscovichNadler2017} in the computer science literature, but only for GOF testing or uniform confidence bands, never quantile multiple testing. 

The same probability integral transform underlies our methods.  
It provides finite-sample distributions while being distribution-free, and it facilitates finite-sample control of both overall FWER and pointwise type I error rates. 
The tradeoff is that iid sampling is required. 
However, the finite-sample sampling distribution (of the true CDF evaluated jointly at all order statistics) turns out to be equivalent to the finite-sample posterior distribution from the continuity-corrected Bayesian bootstrap in \citet{Banks1988}: in the iid case, the uniform confidence band of \citet{AldorNoimanEtAl2013} is also a Bayesian uniform credible band.%
\footnote{This is shown formally in a prior version of this paper; see Proposition 10 in \url{https://faculty.missouri.edu/~kaplandm/pdfs/GK2016_dist_inf_KStype_longer.pdf}.} 
We are hopeful that further work will show that our iid assumption may be ``relaxed'' by using the Bayesian bootstrap to allow sampling weights \citep[as in][]{Lo1993} and clustering \citep[as in][]{CameronEtAl2008} while maintaining exact finite-sample properties in the iid case. 

We provide a closed-form calibration formula that allows not only our multiple testing procedure but also the uniform confidence band and GOF $p$-values of \citet{BujaRolke2006} to be computed instantly. 
This formula replaces just-in-time simulations that can last (depending on sample size) minutes or even hours. 

Many of our new multiple testing results rely on viewing the problem from the quantile perspective rather than the probability perspective of \citet{BujaRolke2006} and related papers. 
Seeing the problem as testing multiple quantiles helps us establish FWER properties and is critical for our procedures to improve power. 
Our strategy is to test $n$ different quantile hypotheses using the $n$ order statistics (i.e., ordered sample values). 
In contrast, papers like \citet{BujaRolke2006} apply the null CDF $F_0(\cdot)$ to the order statistics $X_{n:1}<\cdots<X_{n:k}<\cdots<X_{n:n}$, comparing $F_0(X_{n:k})$ to certain critical values for $k=1,\ldots,n$. 
When the true $F(\cdot)$ equals $F_0(\cdot)$, it is easy to analyze such a test's properties, and such is sufficient for GOF testing. 
However, if $F(x)=F_0(x)$ only over a proper subset of $\R$, then it is difficult to compute the false rejection probability: $X_{n:k}$ is random, so $F_0(X_{n:k})=F(X_{n:k})$ is true in some samples but not others. 
The quantile perspective avoids this difficulty: each pointwise null hypothesis concerns only one fixed population quantile value, $F^{-1}(\tau)$, which is tested with one order statistic. 

This quantile multiple testing perspective facilitates procedures to improve power. 
It leads naturally to a stepdown procedure in the spirit of \citet{Holm1979}, where if any quantile hypotheses are rejected by the initial test, then the remaining ones may be tested with a smaller critical value. 
Further, for one-sided testing, we propose a pre-test to determine at which quantiles the null hypothesis inequality constraint may be binding, and the pointwise test levels are recalibrated with attention restricted to this subset, similar to 
\citet{LintonEtAl2010}, among others. 

In the literature, using the probability integral transform for GOF testing dates back to \citet{Fisher1932}, \citet{Pearson1933}, and \citet{Neyman1937}. 
An extension especially relevant to us is that the joint distribution of $F(X_{n:1}),\ldots,F(X_{n:n})$ for order statistics $X_{n:1}<\cdots<X_{n:n}$ is the same as that of the order statistics from a $\UnifDist(0,1)$ distribution; \citet{ScheffeTukey1945} seem to be the first to note this \citep[e.g., as cited in][]{DavidNagaraja2003}. 
%
Using a closely related sampling distribution, nonparametric (empirical) likelihood-based GOF testing and uniform confidence bands are respectively proposed by \citet{BerkJones1979} and \citet{Owen1995}.  
However, they do not discuss multiple testing or two-sample inference, and our methods compute faster and spread sensitivity more evenly. 

For multiple testing concepts like FWER and stepdown procedures, see Chapter 9 in \citet{LehmannRomano2005text}, \citet{RomanoEtAl2010}, and references therein.

\Cref{sec:KS} contains results for the KS-based multiple testing procedure. 
\Cref{sec:1s,sec:2s} describe our new methods and their properties. 
\Cref{sec:emp,sec:sim} contain empirical examples and simulation results, respectively. 
\Cref{sec:app-meth} contains additional methods, 
\cref{sec:app-pfs} contains proofs, 
\cref{sec:app-comp} contains computational details, and 
\cref{sec:app-sim} contains additional simulations. 

Notationally, we use $\alpha$ for FWER and $\tilde\alpha$ for pointwise type I error rate. 
Acronyms and abbreviations used include those for confidence interval (CI), data generating process (DGP), empirical distribution function (EDF), familywise error rate (FWER), goodness-of-fit (GOF), Kolmogorov--Smirnov (KS), multiple testing procedure (MTP), and rejection probability (RP). 
%
Random and non-random vectors are respectively typeset as, e.g., $\vecf{X}$ and $\vecf{x}$, 
while random and non-random scalars are typeset as $X$ and $x$, 
and random and non-random matrices as $\matf{X}$ and $\matf{x}$; 
$\Ind{\cdot}$ is the indicator function. 
The Dirichlet distribution with parameters $a_1,\ldots,a_K$ is written $\textrm{Dir}(a_1,\ldots,a_K)$, 
the beta distribution $\BetaDist(a,b)$, and the uniform distribution $\UnifDist(a,b)$; in some cases these stand for random variables following such distributions. 
The $\alpha$-quantile of the $\BetaDist(k,n+1-k)$ distribution is denoted by $B_{k,n}^\alpha$.

\section{Setup}\label{sec:setup}

First, we define multiple testing terms following \citet[\S9.1]{LehmannRomano2005text}. 
\begin{definition}\label{def:FWER}
For a family of null hypotheses $H_{0h}$ indexed by $h$, let $I\equiv\{h:H_{0h}\textrm{ is true}\}$.  
The ``familywise error rate'' is 
\[ \FWER \equiv \Pr\left(\textrm{reject any $H_{0h}$ with $h\in I$}\right) . \] 
\end{definition}
\begin{definition}\label{def:FWER-control}
Given the notation in \cref{def:FWER}, ``weak control'' of FWER at level $\alpha$ requires $\FWER \le \alpha$ if each $H_{0h}$ is true.  
``Strong control'' of FWER requires $\FWER \le \alpha$ for any $I$.
\end{definition}
Given \cref{def:FWER-control}, when we establish strong control of FWER, then weak control is directly implied. 
In our results, we will establish strong control of FWER over a set of distributions, similar to establishing ``size control'' by showing that type I error rates are controlled over a set of distributions.\footnote{We do not allow the population distribution to drift asymptotically to consider ``uniformity,'' but we conjecture that at least our basic methods do not suffer such issues.} 

Second, we maintain the following assumptions throughout. 
\begin{assumption}\label{a:iid}
One-sample: scalar observations $X_i\stackrel{iid}{\sim}F$, and the sample size is $n$. 
Two-sample: scalar observations $X_i\stackrel{iid}{\sim}F_X$, $Y_i\stackrel{iid}{\sim}F_Y$, with respective sample sizes $n_X$ and $n_Y$, and the samples are independent of each other: $\{X_i\}_{i=1}^{n_X}\independent\{Y_k\}_{k=1}^{n_Y}$. 
\end{assumption}
\begin{assumption}\label{a:F}
One-sample: $F(\cdot)$ is continuous and strictly increasing over its support, taken to be $\R$. 
Two-sample: $F_X(\cdot)$ and $F_Y(\cdot)$ are continuous and strictly increasing over their common support, taken to be $\R$. 
\end{assumption}

\Cref{a:iid} is applicable in many cases (such as our empirical example), but excludes settings with sampling weights or dependence. 
As noted in \cref{sec:intro}, explorations of weakening this assumption through the connection with \citet{Banks1988} are in progress. 

\Cref{a:F} excludes discrete distributions; in such cases, our methods are conservative (like the KS test). 
The support is taken to be $\R$ for simplicity; any subset of $\R$ is fine. 
The continuity in \cref{a:F} allows the probability integral transform to be used, $F(X_i)\stackrel{iid}{\sim}\UnifDist(0,1)$. 
The strict monotonicity implies the CDF is invertible (without having to define the generalized inverse), so $F^{-1}(\cdot)$ is the quantile function, and $F^{-1}\bigl(F(r)\bigr)=r$ as well as $F\bigl(F^{-1}(\tau)\bigr)=\tau$. 

Third, we address the following tasks. 
Unlike with a GOF test, which has a single global hypothesis and corresponding single decision (reject or not), \cref{task:1s-test-FWER-2s,task:1s-test-FWER-1s,task:2s-test-FWER-CDF-2s,task:2s-test-FWER-CDF-1s} each involve a continuum of pointwise hypotheses that each require a decision. 
\begin{enumeratecomp}[\bfseries T{a}sk 1]
\item\label{task:1s-test-FWER-2s} One-sample, two-sided testing of 
$H_{0\tau} \colon F^{-1}(\tau) = F_0^{-1}(\tau)$ for $\tau\in(0,1)$ and fixed $F_0^{-1}(\cdot)$, with strong control of FWER.  
\item\label{task:1s-test-FWER-1s} Same as \cref{task:1s-test-FWER-2s} but with one-sided 
$H_{0\tau} \colon F^{-1}(\tau) \ge F_0^{-1}(\tau)$ or 
$H_{0\tau} \colon F^{-1}(\tau)\le F_0^{-1}(\tau)$. 
\item\label{task:2s-test-FWER-CDF-2s} Two-sample, two-sided testing of 
$H_{0r} \colon F_X(r)=F_Y(r)$ for $r\in\R$, 
with strong control of FWER. 
\item\label{task:2s-test-FWER-CDF-1s} Same as \cref{task:2s-test-FWER-CDF-2s} but with one-sided 
$H_{0r} \colon F_X(r)\le F_Y(r)$ or 
$H_{0r} \colon F_X(r)\ge F_Y(r)$. 
\setcounter{allenumi}{\theenumi}
\end{enumeratecomp}

Quantile and distribution (CDF) tests are equivalent in the one-sample setting, but not two-sample. 
\Cref{task:1s-test-FWER-2s,task:1s-test-FWER-1s} are equivalent to CDF tests: 
 given \cref{a:F}, if $F(r)>F_0(r)=\tau$, then $F^{-1}(\tau)<F_0^{-1}(\tau)$. 
In contrast, for \cref{task:2s-test-FWER-CDF-2s,task:2s-test-FWER-CDF-1s}, the null hypothesis does not determine the value of $F_X(r)$ or $F_Y(r)$. 
Consequently, a method with strong control of FWER for \cref{task:2s-test-FWER-CDF-2s,task:2s-test-FWER-CDF-1s} does not necessarily have strong control of FWER for the corresponding quantile hypotheses, although it does have weak control of FWER; \cref{sec:2s} contains details. 

Although different than the economic interpretation of quantile hypotheses, the interpretation of CDF hypotheses is still simple and meaningful. 
For example, imagine $F_Y(\cdot)$ is the distribution of hourly wage for individuals who did a job training program, and $F_X(\cdot)$ is the hourly wage distribution without the program, both in dollars per hour. 
If $H_{0r} \colon F_Y(r) \ge F_X(r)$ is rejected in favor of $F_Y(r)<F_X(r)$ for $r=15$, then the data suggest that the program increases the probability of an individual making at least \$15/hr. 

\section{KS-based multiple testing procedures}
\label{sec:KS}

The one-sample and two-sample KS GOF tests are well known, including the simulation of finite-sample exact $p$-values. 
We present the corresponding ``KS-based'' multiple testing procedures (MTPs) and establish their strong control of FWER. 
Although seemingly intuitive, we are unaware of such a presentation in the literature. 
Last in this section, we discuss the problem of uneven sensitivity. 

For the one-sample, two-sided KS-based MTP, given notation in \cref{a:iid,a:F}, let 
\begin{equation}\label{eqn:def-EDF-Dx0}
\hat{F}(x) \equiv \sum_{i=1}^{n} (1/n)\Ind{X_i\le x} , \quad
D_n^{x,0} \equiv \sqrt{n}\bigl\lvert\hat{F}(x)-F_0(x)\bigr\rvert , 
\end{equation}
for all $x\in\R$. 
Let $c_n(\alpha)$ denote the exact critical value with sample size $n$, so 
\begin{equation}\label{eqn:def-c-1s}
\Pr\bigl( D_n > c_n(\alpha) \bigr) = \alpha , \quad
D_n \equiv \sup_{x\in\R} D_n^x , \quad
D_n^x \equiv \sqrt{n}\bigl\lvert \hat{F}(x)-F(x) \bigr\rvert . 
\end{equation}
It is well known that $c_n(\alpha)$ is distribution-free, depending only on $n$ and $\alpha$. 
Alternatively, the asymptotic $c_\infty(\alpha)$ can be used, such that 
\[ \Pr\biggl(\sup_{t\in[0,1]}\lvert B(t)\rvert > c_\infty(\alpha) \biggr) = \alpha \]
for standard Brownian bridge $B(\cdot)$. 

The KS test proper is a GOF test that rejects 
$H_0 \colon F(\cdot)=F_0(\cdot)$ 
when 
$\sup_{x\in\R}D_n^{x,0} > c_n(\alpha)$. 
Under $H_0$, this occurs with probability $\alpha$. 

The corresponding MTP addressing \cref{task:1s-test-FWER-2s} is intuitive: reject $H_{0x} \colon F(x) = F_0(x)$ for any $x\in\R$ such that $D_n^{x,0}>c_n(\alpha)$. 
(To directly address \cref{task:1s-test-FWER-2s}: if $H_{0x}$ is rejected, then $H_{0\tau}$ is rejected for $\tau=F_0(x)$.) 
Weak control of FWER is immediate from the GOF test's size control: when $F(\cdot)=F_0(\cdot)$, the probability of at least one pointwise rejection is equivalent to the probability of $\sup_{x\in\R}D_n^{x,0}>c_n(\alpha)$, which is exactly $\alpha$. 
Strong control of FWER is also straightforward to establish. 

\begin{proposition}\label{prop:KS-1s-FWER}
Let \cref{a:iid,a:F} hold, as well as the definitions in \cref{eqn:def-EDF-Dx0,eqn:def-c-1s}. 
The two-sided exact (or asymptotic) KS-based MTP that rejects $H_{0x} \colon F(x) = F_0(x)$ for any $x\in\R$ where $D_n^{x,0}$ exceeds the critical value has strong control of exact (or asymptotic) FWER. 
The corresponding one-sided KS-based MTPs of $H_{0x} \colon F(x)\le F_0(x)$ or $H_{0x} \colon F(x) \ge F_0(x)$ also have strong control of FWER. 
\end{proposition}
\begin{proof}
As in \cref{def:FWER}, let $I\equiv\{x : H_{0x}\textrm{ is true}\}\subseteq\R$. 
For the two-sided case, using \cref{eqn:def-EDF-Dx0,eqn:def-c-1s},
\[ \FWER
  \equiv \Pr\left(\sup_{x\in I}D_n^{x,0} > c_n(\alpha) \right) 
  =   \Pr\left(\sup_{x\in I}D_n^{x} > c_n(\alpha) \right) 
  \le \overbrace{\Pr\left( \sup_{x\in\R}D_n^{x} > c_n(\alpha) \right)
  = \alpha}^{\textrm{by \cref{eqn:def-c-1s}}} . \]
The one-sided case is similar and shown in the appendix. 
\end{proof}

In the two-sample case, let 
\begin{gather}\label{eqn:def-EDF-2s}
\hat{F}_X(r) \equiv \sum_{i=1}^{n_X} (1/n_X)\Ind{X_i\le r} , \quad
\hat{F}_Y(r) \equiv \sum_{i=1}^{n_Y} (1/n_Y)\Ind{Y_i\le r} , \\
\label{eqn:def-Dr-2s}
D_{n_X,n_Y}^r \equiv \sqrt{\frac{n_X n_Y}{n_X+n_Y}}\bigl\lvert\hat{F}_X(r)-\hat{F}_Y(r)\bigr\rvert , 
\end{gather}
for all $r\in\R$. 
Under $H_0 \colon F_X(\cdot) = F_Y(\cdot)$, the critical value $c_{n_X,n_Y}(\alpha)$ is again distribution-free, and it converges to the same $c_\infty(\alpha)$ as for the one-sample test as $n_X,n_Y\to\infty$ at the same rate. 
However, in finite samples, 
we only have an inequality: under $H_0$, 
\begin{equation}\label{eqn:def-c-2s}
\Pr\bigl( D_{n_X,n_Y} > c_n(\alpha) \bigr) \le \alpha , \quad
D_{n_X,n_Y} \equiv \sup_{r\in\R} D_{n_X,n_Y}^r . 
\end{equation}
Equality is impossible for most $\alpha$ because the distribution of $D_{n_X,n_Y}$ is discrete: it depends only on the ordering (i.e., permutation) of the $X_i$ and $Y_i$, and with finite $n_X$ and $n_Y$ the number of such orderings is finite. 

The corresponding two-sample MTP addressing \cref{task:2s-test-FWER-CDF-2s} is intuitive: reject $H_{0r} \colon F_X(r) = F_Y(r)$ for any $r\in\R$ such that $D_{n_X,n_Y}^r>c_{n_X,n_Y}(\alpha)$. 
Weak control of FWER is again immediate from the GOF test's size control. 
Strong control of FWER can also be established, as in \cref{prop:KS-2s-FWER}.%
\footnote{Asymptotically, and usually not framed in terms of FWER, stronger results in more complex models exist, such as the nonparametric, uniform (over $\tau$) confidence band for the difference of two conditional quantile processes in \citet[\S6.2]{QuYoon2015}, or the ``uniform inference'' on the quantile treatment effect process in \citet[\S4]{FirpoGalvao2015}.} 
The key is that, given $\alpha$, $n_X$, and $n_Y$, rejection of $H_{0r}$ depends only on $\hat{F}_X(r)$ and $\hat{F}_Y(r)$, whose distributions are independent (by \cref{a:iid,eqn:def-EDF-2s}) and depend only on $F_X(r)$ and $F_Y(r)$.  
Such a property extends over multiple $r$ values jointly, too.  
This allows us to link the FWER with a probability under $F_X(\cdot)=F_Y(\cdot)$, which is bounded by the size of the global GOF test. 
(Implicitly, this was actually the one-sample argument for \cref{prop:KS-1s-FWER}, too.) 
Since this general proof structure is used later, part of the argument is given in the following lemma. 

\begin{lemma}\label{lem:weak-to-strong}
Let \Cref{a:iid,a:F} hold. 
Consider any MTP for \cref{task:2s-test-FWER-CDF-2s} or \cref{task:2s-test-FWER-CDF-1s}. 
Assume it has weak control of FWER at level $\alpha$. 
Assume that, given $\alpha$, $n_X$, and $n_Y$, rejection of $H_{0r}$ depends only on $\hat{F}_X(r)$ and $\hat{F}_Y(r)$, for any $r\in\R$. 
Then, the MTP has strong control of FWER at level $\alpha$. 
\end{lemma}
\begin{proposition}\label{prop:KS-2s-FWER}
Let \cref{a:iid,a:F} hold, as well as the definitions in \cref{eqn:def-EDF-2s,eqn:def-Dr-2s,eqn:def-c-2s}. 
The two-sided exact (or asymptotic) KS-based MTP that rejects $H_{0r} \colon F_X(r) = F_Y(r)$ for any $r\in\R$ where $D_{n_X,n_Y}^r$ exceeds the critical value has strong control of exact (or asymptotic) FWER. 
The corresponding one-sided KS-based MTPs of $H_{0r} \colon F_X(r) \le F_Y(r)$ or $H_{0r} \colon F_X(r) \ge F_Y(r)$ also have strong control of FWER. 
\end{proposition}

Although strong control of FWER is helpful, the KS-based MTPs suffer from uneven sensitivity to deviations from the null. 
One symptom of this was seen in the example in \cref{sec:intro}, where the one-sample KS test could not reject that the population was $\UnifDist(0,1)$ even with five out of $n=20$ observations exceeding one million. 
More generally, the one-sample, two-sided KS test does not reject at a $10\%$ level even if $F_0(X_{20:16})=1$ or if $F_0(X_{20:5})=0$, which any reasonable test should 
and which our test does. 
This uneven sensitivity results in ``low power in the tails'' and is well documented in the literature. 
For example, \citet{Eicker1979} says that the KS is ``sensitive asymptotically only in the central range given by $\{\tau:(\log\log n)^{-1} < F(\tau) < 1-(\log\log n)^{-1}\}$'' (p.\ 117). 

\begin{figure}[htbp]
\centering
\includegraphics[width=0.6\textwidth,clip=true,trim=20 10 20 0]{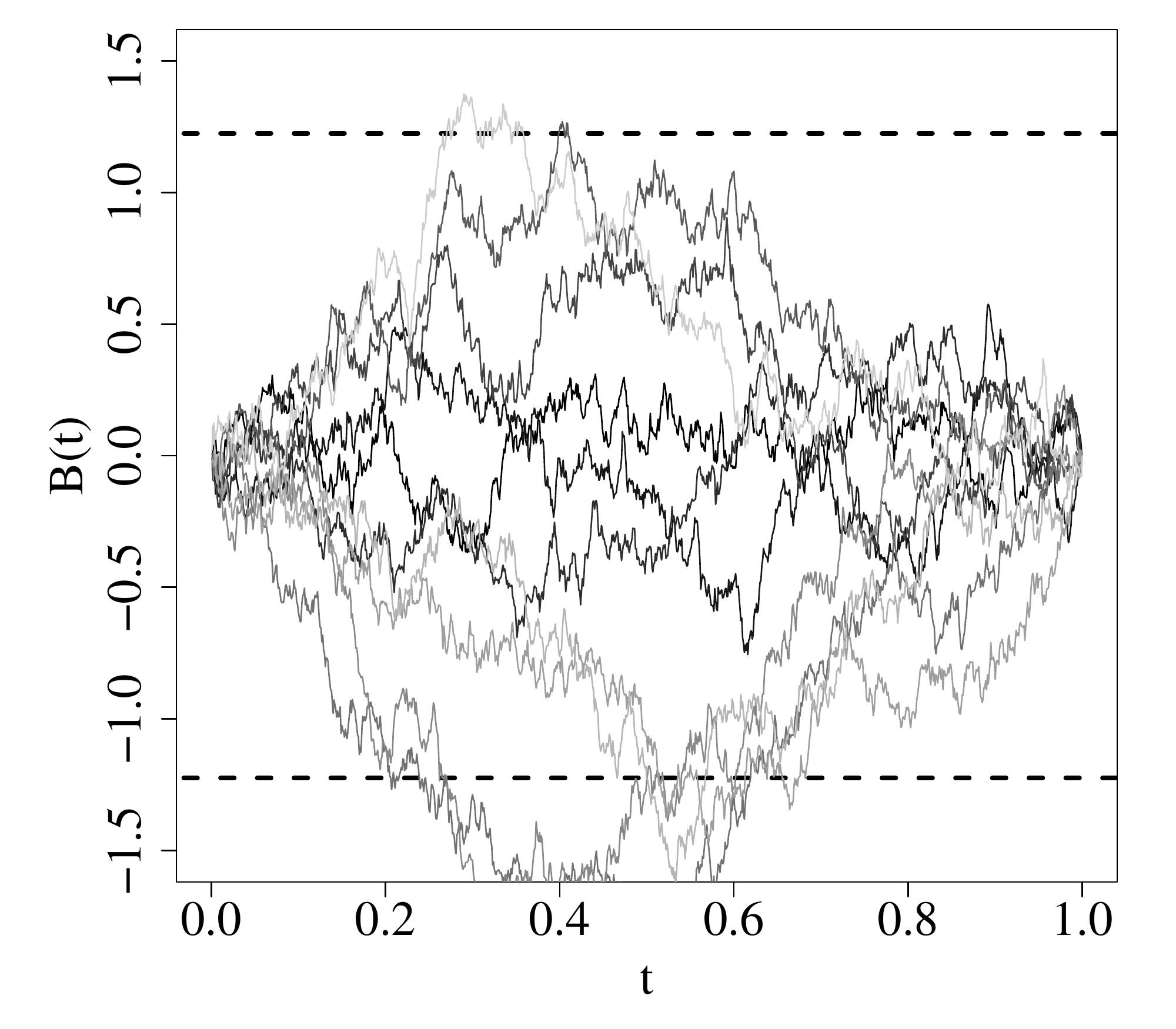}
\caption{\label{fig:Bt-paths}Sample paths of standard Brownian bridge $B(\cdot)$, with asymptotic $10\%$ two-sided KS critical value $\pm1.225$.  Six paths are (randomly) chosen among those exceeding the threshold; four are chosen that do not.}
\end{figure}

\Cref{fig:Bt-paths} visualizes the intuition for the KS test's low sensitivity in the tails. 
The figure shows sample paths (realizations) of a standard Brownian bridge along with the two-sided asymptotic $10\%$ KS critical value, $1.225$. 
That is, $\Pr\bigl(\sup_{t\in[0,1]}\lvert B(t)\rvert > 1.225 \bigr)=0.1$, so only $10\%$ of sample paths wander above $1.225$ or below $-1.225$, leading to (asymptotic) weak control of FWER. 
(The figure oversamples paths exceeding the critical value to avoid a visual mess.) 
However, as can be seen, such excesses are much more likely to occur near the median ($t=0.5$) than in the tails. 
In the figure, the six paths deviating beyond the critical value do so only in $t\in[0.25,0.75]$, and the clumping of sample paths near $B(t)=0$ for $t$ near zero or one shows the difficulty of having a large deviation in the tails. 
This is due to the pointwise variance being highest at $t=0.5$, with $\Var\bigl(B(0.5)\bigr)=t(1-t)=0.25$, and lowest as $t\to0$ or $t\to1$, where the variance approaches zero. 

\begin{figure}[htbp]
\centering
\hfill
\includegraphics[width=0.49\textwidth,clip=true,trim=0 10 0 80]{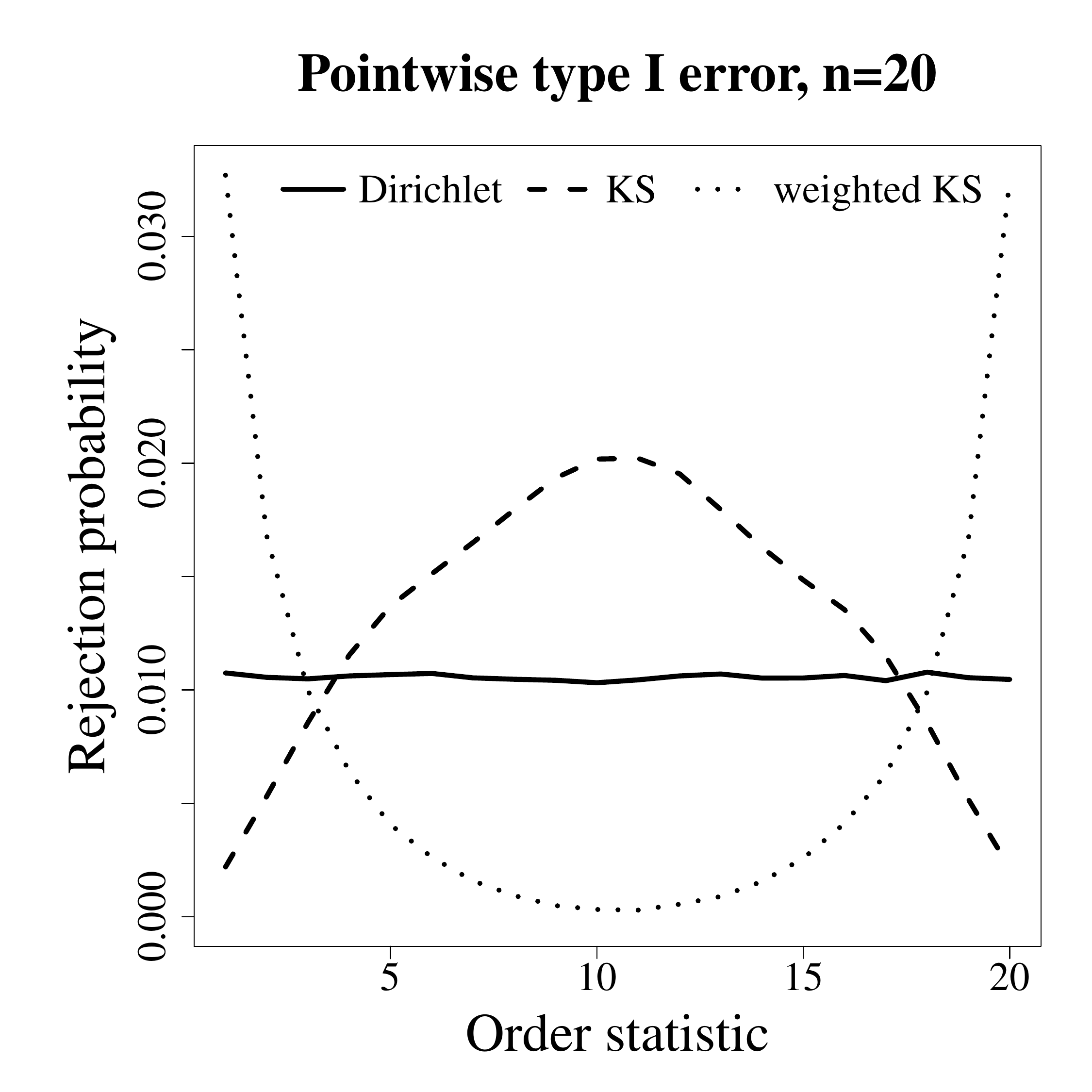}
\hfill
\includegraphics[width=0.49\textwidth,clip=true,trim=0 10 0 80]{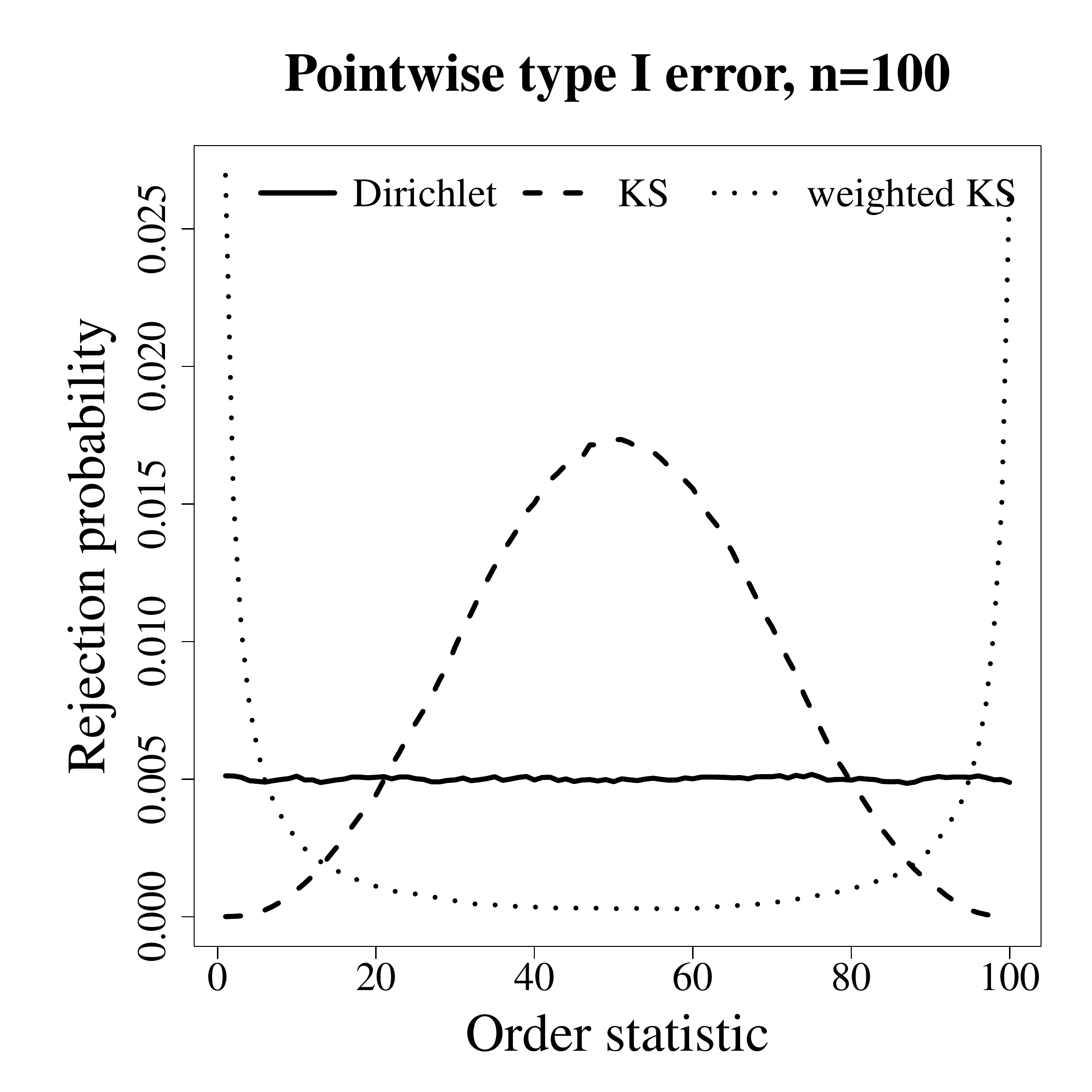}
\hfill\null
\caption{\label{fig:sim-1s-pt}Simulated ``pointwise'' RP ($\tilde\alpha$) at each order statistic, $H_0 \colon F(\cdot) = F_0(\cdot)$, $n=20$ (left) and $n=100$ (right), $10^6$ replications. Overall FWER for all MTPs is exactly $\alpha=0.1$.}
\end{figure}

\Cref{fig:sim-1s-pt} visualizes (un)even sensitivity from a different perspective. 
The KS-based MTP is labeled ``KS'' in the legend; the ``weighted KS'' is a weighted version described below. 
``Dirichlet'' is our new MTP, detailed in \cref{sec:1s}. 
FWER is exactly $10\%$ for all methods shown, but sensitivity is allocated differently across the distribution. 
The probability of $X_{n:k}$ causing some $H_{0r} \colon F(r) = F_0(r)$ to be rejected is simulated\footnote{The simulation uses a standard uniform distribution for $F(\cdot)$, but the results are distribution-free: one could simply transform the data by $F_0(\cdot)$ and test against a standard uniform.} 
for $k=1,\ldots,n$. 
The resulting pattern is a more systematic account of the intuition in \cref{fig:Bt-paths}: the pointwise rejection probability (RP) due to central order statistics is much higher than the RP due to extreme order statistics, and RP goes to almost zero at the sample minimum and maximum. 
This pattern is already clear with $n=20$ (left) and becomes more exaggerated with $n=100$ (right). 
Practically, this shows that although one is technically testing across the entire distribution, the KS-based MTP (implicitly) weights the middle of the distribution much more than the tails, which may not be desired. 
The corresponding uneven allocation of KS pointwise power is illustrated in \cref{sec:sim-pt-pwr}. 

\begin{figure}[htbp]
\centering
\hfill
\includegraphics[width=0.49\textwidth,clip=true,trim=0 10 0 80]{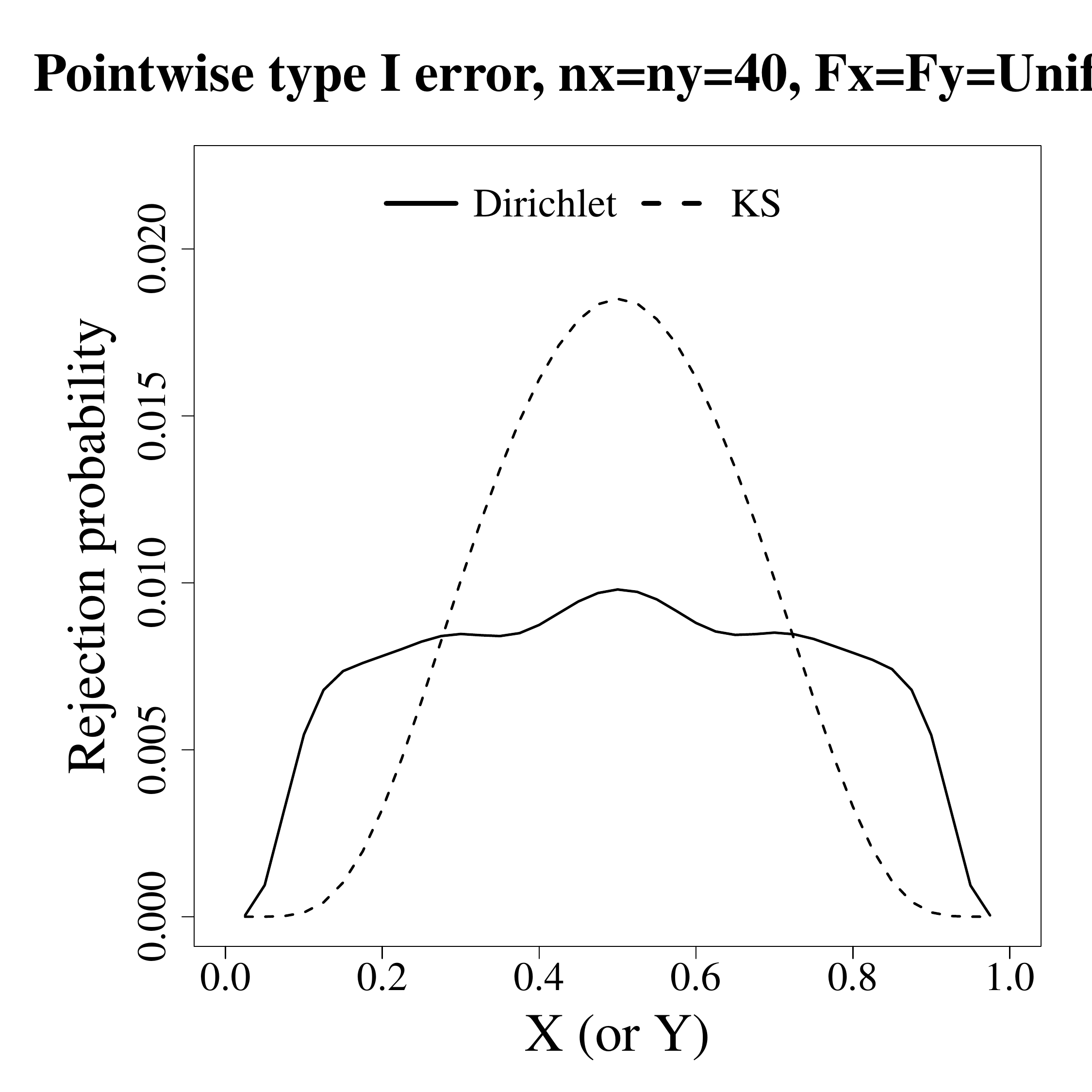}
\hfill
\includegraphics[width=0.49\textwidth,clip=true,trim=0 10 0 80]{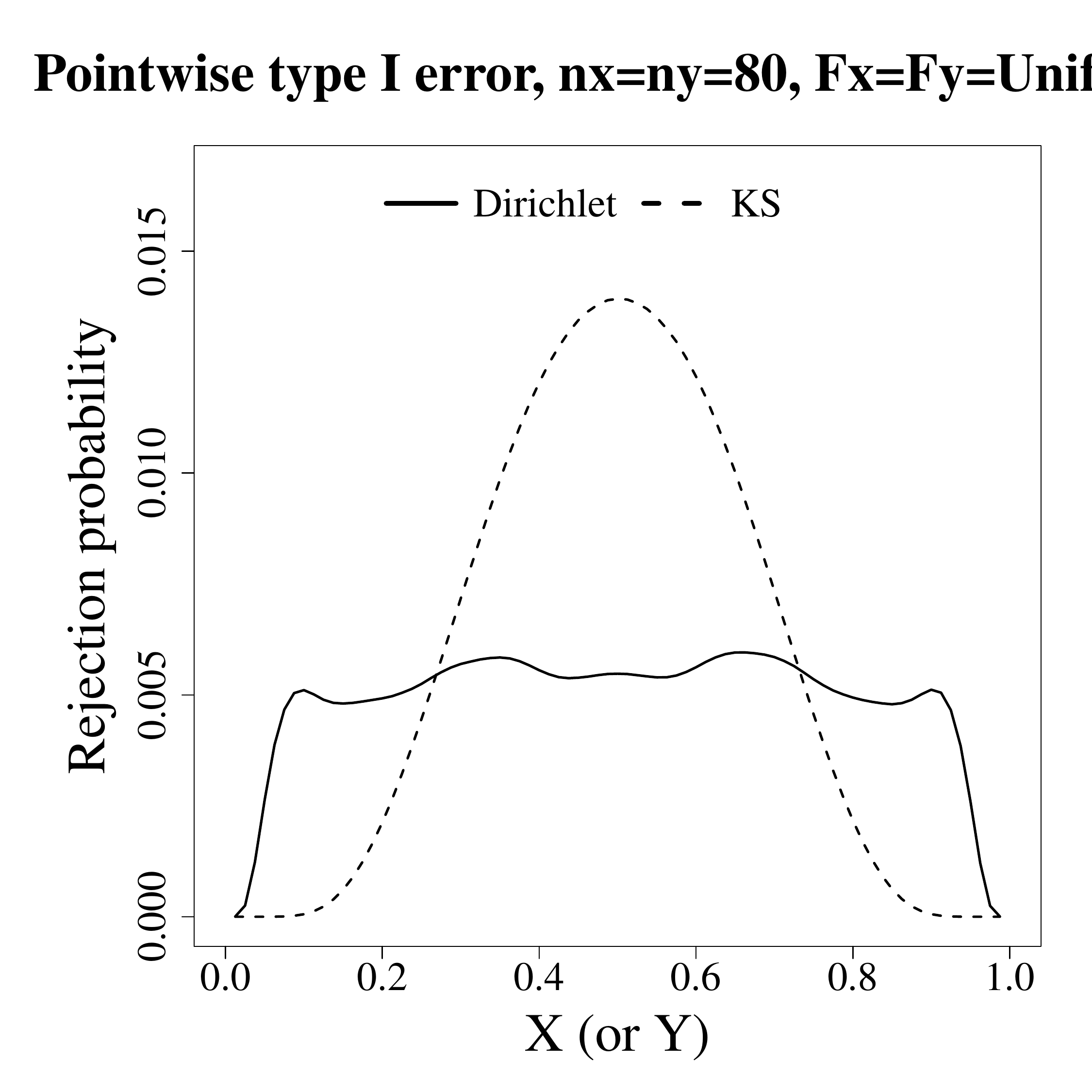}
\hfill\null
\caption{\label{fig:sim-2s-pt}Simulated pointwise RP ($\tilde\alpha$), $n_X=n_Y=40$ (left) and $n_X=n_Y=80$ (right), FWER for both MTPs is exactly $\alpha=0.1$, $10^6$ replications, $F_X=F_Y=\UnifDist(0,1)$.}
\end{figure}

\Cref{fig:sim-2s-pt} shows the same qualitative pattern for the two-sample KS-based MTP. 
The horizontal axis now just shows the value $r\in[0,1]$ for which $H_{0r} \colon F_X(r) = F_Y(r)$ is rejected. 
As in \cref{fig:sim-1s-pt}, the pointwise RP peaks near the median and goes to zero in the tails. 

If this uneven sensitivity stems from \cref{fig:Bt-paths} having uneven pointwise variance $t(1-t)$, then the simple solution is to divide by the standard deviation $\sqrt{t(1-t)}$ to achieve equal, unit variance at each $t$. 
In the one-sample KS context, this means dividing by $\sqrt{F_0(t)[1-F_0(t)]}$. 
\citet{AndersonDarling1952} gave exactly this solution in their Example 2 (p.\ 202), where they note it was first suggested by L.\ J.\ Savage (Footnote 2); they say, ``In a certain sense, this function assigns to each point of the distribution $F(x)$ equal weights'' (pp.\ 202--203).\footnote{The same paper includes a weighted Cram\'{e}r--von Mises test that is most commonly called the Anderson--Darling test.} 
However, they note that their results require the tails to have zero weight (pp.\ 210--211), which undermines the goal of even sensitivity. 
If nonetheless the weight is applied even in the tails, then the tails become overly sensitive. 
This is characterized by \citet[p.\ 117]{Eicker1979} as the weighted KS being ``sensitive only in the moderate tails given by, e.g., $\bigl\{\tau : n^{-1}\log n < F(\tau) < \bigl((\log\log n)\log\log\log n\bigr)^{-1}\bigr\}$.'' 
Other discussions of the unintended (bad) consequences of this weighting scheme are found in \citet{Jaeschke1979} and \citet{Lockhart1991}, among others. 
For the corresponding MTP, the ``weighted KS'' line in \cref{fig:sim-1s-pt} shows that, indeed, the pointwise RP is much higher in the tails than near the median. 

\Cref{fig:sim-1s-pt,fig:sim-2s-pt} each show a line labeled ``Dirichlet'' that achieves a great degree of even sensitivity. 
These are the basic MTPs we propose in \cref{sec:1s,sec:2s}.

\section{One-sample Dirichlet approach}
\label{sec:1s}

We propose methods for multiple testing of quantiles based on the probability integral transform and Dirichlet distribution, including stepdown and pre-test procedures to improve power. 
The Dirichlet distribution is used for GOF testing and uniform confidence bands in \citet{BujaRolke2006}, but our methods, their properties, and the quantile multiple testing framework itself are novel. 

The Dirichlet approach uses the same pointwise type I error rate, $\tilde\alpha$, for multiple quantile tests across the distribution, while choosing the value of $\tilde\alpha$ to ensure strong control of finite-sample FWER at level $\alpha$.

\subsection{Basic method, FWER, and computation}
\label{sec:1s-basic}

All of our methods use the probability integral transform. 
The following results are from \citet[pp.\ 236--238]{Wilks1962}, with some notational changes. 
\begin{theorem}[Wilks  \textbf{8.7.1}, \textbf{8.7.2}, \textbf{8.7.4}]
\label{thm:Wilks}
The following are true under \cref{a:iid,a:F}.  
Denote the order statistics by $X_{n:1}<\cdots<X_{n:n}$.  
Then $F(X_{n:1})$, $F(X_{n:2})-F(X_{n:1})$, \ldots, $F(X_{n:n})-F(X_{n:n-1})$, $1-F(X_{n:n})$ are random variables jointly following the $(n+1)$-variate Dirichlet distribution $\textrm{Dir}(1,\ldots,1)$.  
That is, the random variables $F(X_{n:1})$, \ldots, $F(X_{n:n})$ have the ordered $n$-variate Dirichlet distribution $\textrm{Dir}^*(1,\ldots,1;1)$, with marginals $F(X_{n:k})\sim\BetaDist(k,n+1-k)$.
\end{theorem}

\Cref{thm:Wilks} determines the finite-sample size of a single quantile test based on an order statistic. 
Specifically, consider the test of $H_{0\tau} \colon F^{-1}(\tau) \ge F_0^{-1}(\tau)$ that rejects when $X_{n:k}<F_0^{-1}(\tau)$ for some $k$. 
Under $H_0$, the type I error rate is bounded (tightly) by 
\begin{equation}\label{eqn:single-CI}
\Pr\bigl( X_{n:k} < F_0^{-1}(\tau) \bigr) 
\le \Pr\bigl( X_{n:k} < F^{-1}(\tau) \bigr) 
= \Pr\bigl( F(X_{n:k}) < \tau \bigr) 
= \Pr\bigl( \BetaDist(k,n+1-k) < \tau \bigr) , 
\end{equation}
i.e., the $\BetaDist(k,n+1-k)$ CDF evaluated at $\tau$. 
This CDF can be computed immediately by any modern statistical software. 
The only difficulty is if a specific $\alpha$ is desired for a specific $\tau$, in which case one cannot find an exact, non-randomized test \citep[but for solutions using interpolation, see][]{BeranHall1993,GoldmanKaplan2017a}. 

Here, instead of testing a specific $\tau$, we presume that the econometrician desires to test a wide range of quantiles. 
By choosing $\tau$ values that allow exact testing using order statistics $X_{n:k}$, 
we can get finite-sample results for a growing number ($n$) of quantiles. 
For comparison, \citet{GoldmanKaplan2017b} provide a confidence set for a fixed number of exactly pre-specified $\tau$ (and $\alpha$), with $O(n^{-1})$ coverage probability error. 

Our strategy is to use each of the $n$ order statistics to test a different $\tau$-quantile null hypothesis. 
Pointwise, let $B_{k,n}^{\tilde\alpha}$ denote the $\tilde\alpha$-quantile of the $\BetaDist(k,n+1-k)$ distribution. 
For any $\tilde\alpha\in(0,0.5)$, let the tested quantiles be 
\begin{equation}\label{eqn:def-ell-k-u-k}
\ell_k \equiv B_{k,n}^{\tilde\alpha} , \quad
   u_k \equiv B_{k,n}^{1-\tilde\alpha} , 
\end{equation}
so $\Pr\bigl(X_{n:k}<F^{-1}(\ell_k)\bigr)=\tilde\alpha$ and $\Pr\bigl(X_{n:k}>F^{-1}(u_k)\bigr)=\tilde\alpha$, using \cref{eqn:single-CI}. 
A one-sided test of $H_{0\ell_k} \colon F^{-1}(\ell_k) \ge F_0^{-1}(\ell_k)$ that rejects when $X_{n:k}<F_0^{-1}(\ell_k)$ thus has exact size $\tilde\alpha$, and similarly for the test of $H_{0u_k} \colon F^{-1}(u_k) \le F_0^{-1}(u_k)$ that rejects when $X_{n:k}>F_0^{-1}(u_k)$. 
By using the same $\tilde\alpha$ across the entire distribution, we achieve even sensitivity. 
More precisely, we achieve the same finite-sample size $\tilde\alpha$ for the pointwise tests at the $n$ quantile indices $\ell_k$ or $u_k$, $k=1,\ldots,n$. 

Moreover, using the Dirichlet distribution in \cref{thm:Wilks}, one can solve for the $\tilde\alpha$ value such that 
\begin{equation}\label{eqn:Dir-1s-1s-tilde}
\alpha 
= 1 - \Pr\Bigl( \bigcap_{k=1}^{n} X_{n:k}\ge F^{-1}(\ell_k) \Bigr) 
= 1 - \Pr\Bigl( \bigcap_{k=1}^{n} F(X_{n:k})\ge\ell_k \Bigr) . 
\end{equation}
This choice of $\tilde\alpha$ achieves strong control of finite-sample FWER at level $\alpha$. 
To see this, let $K\equiv\{k:F^{-1}(\ell_k)\ge F_0^{-1}(\ell_k)\}$, the set of true hypotheses. 
Then, 
\begin{align*}
\FWER
  &=   \overbrace{1 - \Pr(\textrm{no rejections among }k\in K)}^{\textrm{by definition of FWER}}
   =   \overbrace{1 - \Pr\Bigl( \bigcap_{k\in K}  X_{n:k}\ge F_0^{-1}(\ell_k) \Bigr) }^{\textrm{by definition of $H_{0\ell_k}$}}
\\&\le \overbrace{1 - \Pr\Bigl( \bigcap_{k\in K}  X_{n:k}\ge F  ^{-1}(\ell_k) \Bigr) }^{\textrm{because $F^{-1}(\ell_k)\ge F_0^{-1}(\ell_k)$ for all $k\in K$, by definition of $K$}}
   \le \overbrace{1 - \Pr\Bigl( \bigcap_{k=1}^{n} X_{n:k}\ge F  ^{-1}(\ell_k) \Bigr) }^{\textrm{because }K\subseteq\{1,2,\ldots,n\}}
\\&= \overbrace{\alpha }^{\textrm{from \cref{eqn:Dir-1s-1s-tilde}}} 
. 
\end{align*}
A parallel argument applies to the other one-sided case with $u_k$. 

To extend testing of the $n$ quantiles to the continuum of $H_{0\tau}$ for all $\tau\in(0,1)$, without affecting FWER, monotonicity of the quantile function is sufficient. 
If $X_{n:k}<F_0^{-1}(\ell_k)$, then $H_{0\ell_k} \colon F^{-1}(\ell_k) \ge F_0^{-1}(\ell_k)$ is rejected. 
If $X_{n:k}<F_0^{-1}(\tau)$ for another $\tau<\ell_k$, then $H_{0\tau} \colon F^{-1}(\tau)\ge F_0^{-1}(\tau)$ is also rejected. 
If we add this event, or rather its complement, into \cref{eqn:Dir-1s-1s-tilde}, however, it disappears because 
\[ \{F(X_{n:k})\ge \ell_k\} \cap \{F(X_{n:k})\ge \tau\} = \{F(X_{n:k})\ge \ell_k\} \]
since the event $\{F(X_{n:k})\ge \tau\}\supset \{F(X_{n:k})\ge \ell_k\}$ for any $\tau<\ell_k$. 

The full one-sided and two-sided MTPs are now described, followed by their strong control of exact FWER, and finally a computational improvement. 

\begin{method}\label{meth:Dir-1s-1s}
For \cref{task:1s-test-FWER-1s}, consider $H_{0\tau} \colon F^{-1}(\tau) \ge F_0^{-1}(\tau)$.  
Let $\ell_k\equiv B_{k,n}^{\tilde\alpha}$. 
Solve for $\tilde\alpha$ from \cref{eqn:Dir-1s-1s-tilde}, using \cref{thm:Wilks}.\footnote{Alternatively, use the $\tilde\alpha$ approximation in \cref{fact:alpha-tilde-rate} after adjusting the one-sided $\alpha$ to two-sided $\alpha_2=2\alpha-\alpha^2$ per Theorem 5.1 of \citet{MoscovichEtAl2016}; details below.} 
For every $\tau\in(0,1)$, reject $H_{0\tau}$ if and only if $F_0^{-1}(\tau)>\min\{X_{n:k}:\ell_k\ge\tau\}$. 
For $H_{0\tau} \colon F^{-1}(\tau) \le F_0^{-1}(\tau)$, replace $\ell_k$ with $u_k\equiv B_{k,n}^{1-\tilde\alpha}$ and reverse all inequalities. 
\end{method}

\begin{method}\label{meth:Dir-1s-2s}
For \cref{task:1s-test-FWER-2s}, let $\ell_k\equiv B_{k,n}^{\tilde\alpha}$ and $u_k\equiv B_{k,n}^{1-\tilde\alpha}$. 
Using \cref{thm:Wilks}, solve for $\tilde\alpha$ from 
\begin{equation}\label{eqn:Dir-1s-2s-tilde}
\alpha 
   = 1 - \Pr\Bigl( \bigcap_{k=1}^{n}\{F^{-1}(\ell_k)\le X_{n:k}\le F^{-1}(u_k)\} \Bigr)
   = 1 - \Pr\Bigl( \bigcap_{k=1}^{n}\{\ell_k \le F(X_{n:k}) \le u_k\} \Bigr) , 
\end{equation}
or use the $\tilde\alpha$ approximation in \cref{fact:alpha-tilde-rate}. 
For every $\tau\in(0,1)$, reject $H_{0\tau}$ if and only if $F_0^{-1}(\tau)>\min\{X_{n:k}:\ell_k\ge\tau\}$ or $F_0^{-1}(\tau)<\max\{X_{n:k}:u_k\le\tau\}$. 
\end{method}

\begin{theorem}\label{thm:Dir-1s-FWER}
Under \Cref{a:iid,a:F}, \cref{meth:Dir-1s-1s,meth:Dir-1s-2s} have strong control of finite-sample FWER. 
\end{theorem}

Additionally, we contribute a new, fast approximation that can be used not only for \cref{meth:Dir-1s-1s,meth:Dir-1s-2s} but also to compute GOF $p$-values as well as uniform confidence bands for $F(\cdot)$. 
Solving \cref{eqn:Dir-1s-2s-tilde} requires approximating the $n$-variate Dirichlet distribution by either numerical integration or simulation (i.e., drawing standard uniform order statistics), which can be slow for large $n$. 
Extensive simulations have revealed a closed-form formula to approximate the necessary $\tilde\alpha$ as a function of $\alpha$ and $n$ with a high degree of accuracy. 
Computation now takes only a couple seconds even for $n=\num{100000}$. 

\begin{fact}\label{fact:alpha-tilde-rate}
Under \cref{a:iid,a:F}, for $\alpha\in\{0.001,0.01,0.05,0.1,0.2,0.5,0.7,0.9\}$ and $n\in[4,10^6]$, for two-sided testing,  
\begin{equation*}
\tilde\alpha = \exp\bigl\{-c_1(\alpha)-c_2(\alpha)\sqrt{\ln[\ln(n)]}-c_3(\alpha)[\ln(n)]^{c_4(\alpha)}\bigr\} , 
\end{equation*}
with 
$c_1(\alpha) = -2.75 -1.04 \ln(\alpha) $, 
$c_2(\alpha) =  4.76 -1.20 \alpha $, 
$c_3(\alpha) = 1.15 -2.39\alpha$, and 
$c_4(\alpha) = -3.96 + 1.72 \alpha^{0.171} $, 
provides an approximate solution to \cref{eqn:Dir-1s-2s-tilde}. 
Define the relative approximation error to be $(\alpha^*-\alpha)/\alpha$, where $\alpha$ is the nominal FWER and $\alpha^*$ is the true FWER (i.e., simulated with $10^6$ replications). 
Then, across all $\alpha$ and $n$ listed above, the relative approximation error never exceeds $20\%$ in absolute value. 
Excluding $\alpha=0.001$, absolute relative approximation error never exceeds $11\%$ (e.g., worst-case FWER is $0.111$ when $\alpha=0.1$). 
\end{fact}

For one-sided multiple testing, to apply \cref{fact:alpha-tilde-rate}, the initial $\alpha$ can be adjusted using Theorem 5.1 of \citet{MoscovichEtAl2016}, which is asymptotically exact and slightly conservative in finite samples. 
Specifically, the two-sided FWER $\alpha_2$ and one-sided FWER $\alpha_1$ (either lower or upper) are related by $\alpha_2 = 2\alpha_1 - \alpha_1^2$ asymptotically; in finite samples, $2\alpha_1 \ge \alpha_2 \ge 2\alpha_1 - \alpha_1^2$. 

\begin{figure}[htbp]
\centering
\hfill
\includegraphics[width=0.48\textwidth,clip=true,trim=40 15 20 75]{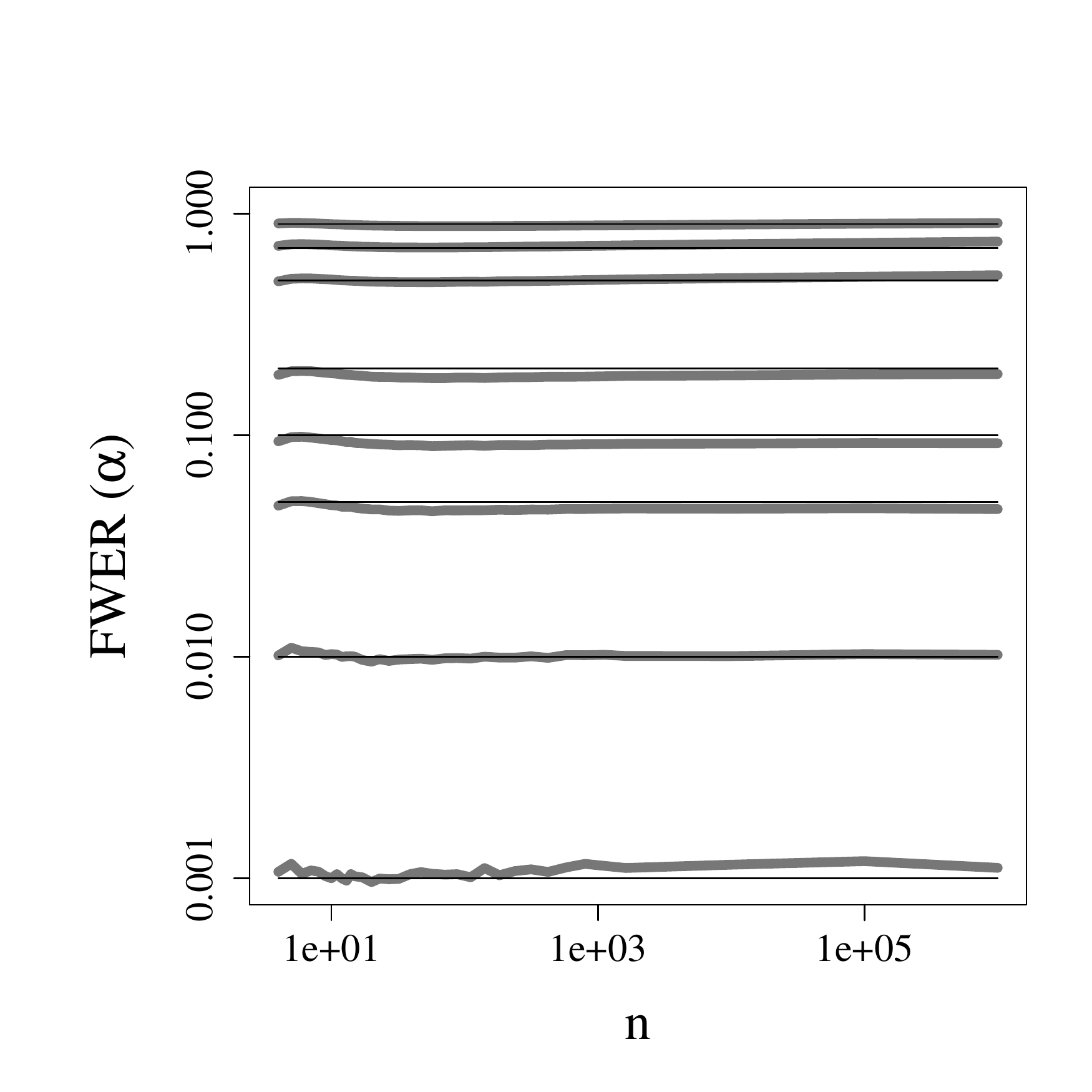}%
\hfill%
\includegraphics[width=0.48\textwidth,clip=true,trim=40 15 20 75]{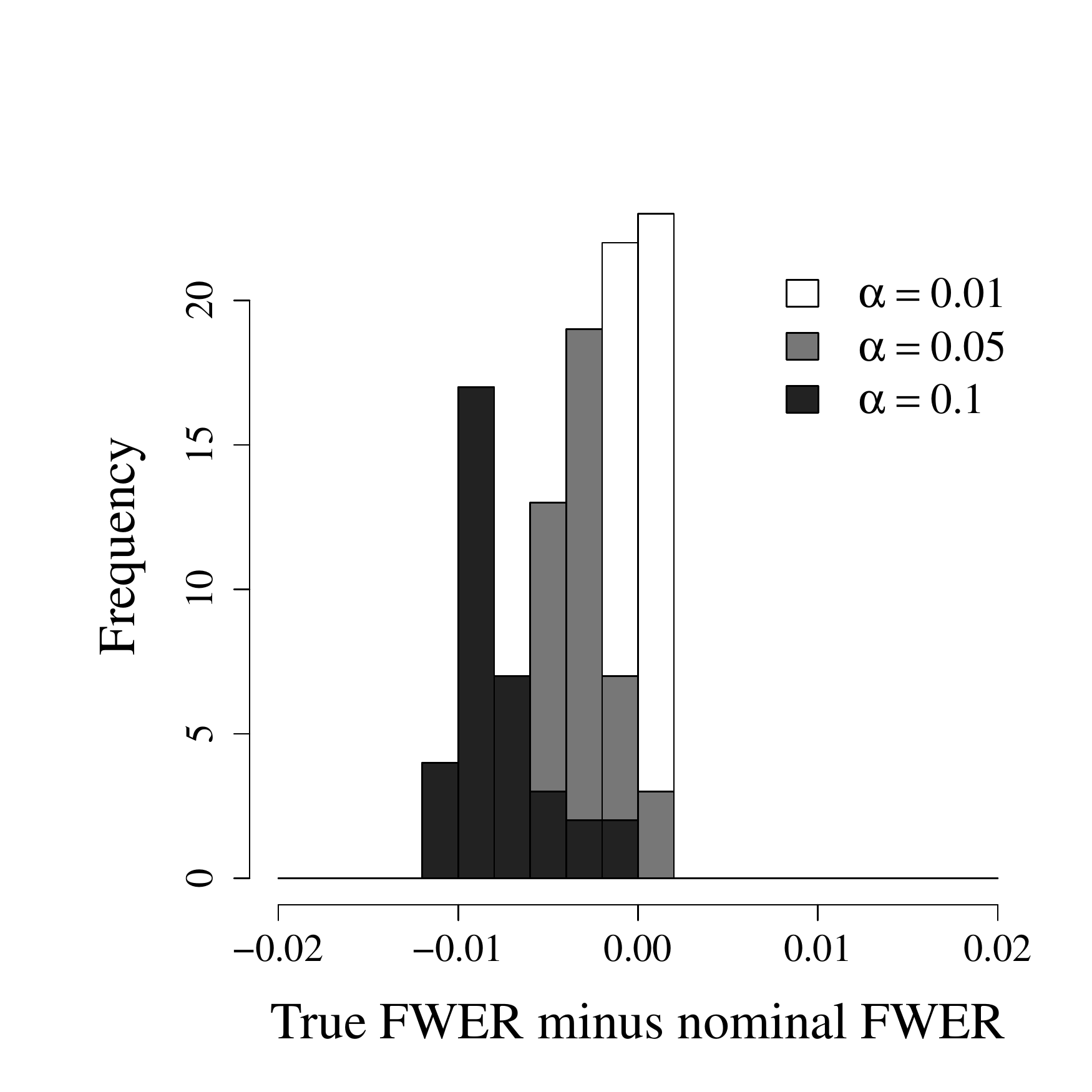}%
\hfill\null
\caption{\label{fig:approx-err-hists}Histograms showing the accuracy of \cref{fact:alpha-tilde-rate} for sample sizes $n\in\{4,5,\ldots,14\}$, $n=\lfloor\exp\{\exp(\kappa)\}\rfloor$ for $\kappa=1.00,1.05,\ldots,2.00$, and $n\in\{10^4,10^5,10^6\}$. 
``True'' FWER is from $10^6$ simulation replications. 
Left: thick gray lines show true FWER; thin black lines show nominal FWER. 
Right: FWER error (true minus nominal) for $\alpha\in\{0.01,0.05,0.1\}$; note that the total height of each bar is the total across all three $\alpha$ values.}
\end{figure}

The accuracy of the \cref{fact:alpha-tilde-rate} formulas is shown in 
\cref{fig:approx-err-hists}. 
As stated in \cref{fact:alpha-tilde-rate}, across all $\alpha$ and $n$, the relative FWER error never exceeds $20\%$ (and is usually close to zero), or $11\%$ when excluding $\alpha=0.001$. 
For example, for $\alpha=0.1$, true FWER is always between $0.089$ and $0.111$. 
To see more specific values, the right panel of \cref{fig:approx-err-hists} shows a histogram of true minus nominal FWER differences. 
For the most commonly used $\alpha\in\{0.01,0.05,0.1\}$, these are seen to err slightly on the conservative side (i.e., true FWER is below nominal) and are often close to zero. 
In our code, we use a version of \cref{fact:alpha-tilde-rate} with coefficients specific to $\alpha$, which further increases the accuracy. 

By monotonicity of the mapping $\tilde\alpha(\alpha,n)$ in $\alpha$ and $n$, additional approximation error from interpolation between the given values of $n$ and $\alpha$ is small.  
We conjecture that the formulas are accurate even outside the ranges given, especially in $n$; moreover, few economic applications require $n>10^6$, $\alpha<0.001$, or $\alpha>0.9$.  

\Cref{fact:alpha-tilde-rate} may also be used to quickly compute GOF $p$-values and uniform confidence bands; see \citet{BujaRolke2006}, \citet{AldorNoimanEtAl2013}, our code, and an earlier version of this paper for details. 

\Cref{fact:alpha-tilde-rate} also illustrates the cost of multiple testing across the entire distribution instead of focusing power on a single quantile. 
Instead of a single quantile test with size $\alpha$, the MTP has pointwise size $\tilde\alpha$, which may be much smaller and goes to zero as $n\to\infty$. 
\Cref{fig:sim-1s-pt} already visualized two examples with FWER $\alpha=0.1$, showing $\tilde\alpha\approx0.01$ when $n=20$ and $\tilde\alpha\approx0.005$ when $n=100$. 
\Cref{fig:sim-2s-pt} visualized two-sample examples. 
The one-sample lookup table included in the replication materials on the latter author's website shows $\tilde\alpha$ values for many more $n$ and $\alpha$.

\subsection{Procedures to improve power}
\label{sec:1s-power}

\subsubsection{Stepdown procedure}
\label{sec:1s-power-stepdown}

Within our quantile multiple testing framework, a stepdown procedure to improve power is possible. 
The general stepdown strategy dates back to \citet{Holm1979}; see also \citet[Ch.\ 9]{LehmannRomano2005text}. 
Although the stepdown procedure strictly improves power (against false hypotheses), it does not ``dominate'' in the general decision-theoretic sense since FWER also increases in some cases, though never exceeding $\alpha$. 

In our context, consider one-sided multiple testing with $\ell_k$. 
If at least one $H_{0\ell_k}$ is rejected by the basic MTP, then we proceed to test the remaining hypotheses as if the rejected ones are indeed false. 
In that case, we can remove such $\ell_k$ from the calibration equation \cref{eqn:Dir-1s-1s-tilde}. 
Intuitively, with fewer true quantile hypotheses, we can test the remainder with greater pointwise size while maintaining the same overall FWER control. 
Mechanically, this is accomplished by changing which order statistic is used to test a hypothesis, to make it more likely to reject.\footnote{If instead $\tilde\alpha$ were increased, then the $\ell_k$ would change. This may be reasonable practically, but it would complicate matters and require a different strategy to prove strong control of FWER.} 
As in other stepdown procedures, the ``trick'' is that if one of the initially rejected hypotheses was in fact true, then a ``familywise error'' has already been made, so rejecting additional hypotheses has no effect on the FWER. 

\begin{method}\label{meth:Dir-1s-1s-stepdown}
For \cref{task:1s-test-FWER-1s}, let $\hat{K}_0\equiv\{1,\ldots,n\}$, and let $r_{k,0}=k$ for $k\in\hat{K}_0$.  
Consider $H_{0\tau} \colon F^{-1}(\tau) \ge F_0^{-1}(\tau)$.  
Let $\ell_k=B^{\tilde\alpha}_{k,n}$, where (as in \cref{meth:Dir-1s-1s}) $\tilde\alpha$ satisfies 
\begin{equation}\label{eqn:meth-1s-stepdown-calib-i0}
\alpha 
  = 1 - \Pr\Bigl( \bigcap_{k\in\hat{K}_0} X_{n:r_{k,0}}\ge F^{-1}(\ell_k) \Bigr) 
    = 1 - \Pr\Bigl( \bigcap_{k\in\hat{K}_0} F(X_{n:r_{k,0}})\ge\ell_k \Bigr) . 
\end{equation}
\Cref{thm:Wilks} determines the joint distribution of the $F(X_{n:k})$ in \cref{eqn:meth-1s-stepdown-calib} and thus the probability. 
Reject $H_{0\tau}$ if $F_0^{-1}(\tau) > \min\{X_{n:k} : \ell_k\ge\tau \}$ (as in \cref{meth:Dir-1s-1s}).  
Then, increment $i$ to $i=1$ and iterate the following. 
\begin{enumeratecomp}[Step 1.]
 \item\label{it:start} Let $\hat{K}_i=\{k : H_{0\ell_k}\textrm{ not yet rejected}\}$.  If $\hat{K}_i=\emptyset$ or $\hat{K}_i=\hat{K}_{i-1}$, then stop. 
 \item Choose integers $r_{k,i}\le r_{k,i-1}$ (based only on $\hat{K}_i$) satisfying:\footnote{\label{foot:rk}This leaves many possibilities for $r_{k,i}$.  In our code, we use a ``greedy'' algorithm \citep[e.g.,][\S4.3]{SedgewickWayne2011}, iteratively decreasing (by one) whichever $r_{k,i}$ achieves the biggest pointwise RP increase, until none can be decreased without violating \cref{eqn:meth-1s-stepdown-calib}.} 
\begin{equation}\label{eqn:meth-1s-stepdown-calib}
\alpha 
  \ge 
    1 - \Pr\Bigl( \bigcap_{k\in\hat K_i} F(X_{n:r_{k,i}})\ge\ell_k \Bigr) . 
\end{equation}
 \item Reject any additional $H_{0\tau}$ for which $F_0^{-1}(\tau) > \min\{X_{n:r_{k,i}} : \ell_k\ge\tau, k\in\hat{K}_i \}$.
 \item Increment $i$ by one and return to Step \ref{it:start}.
\end{enumeratecomp}

For $H_{0\tau} \colon F^{-1}(\tau) \le F_0^{-1}(\tau)$, replace $\ell_k$ with $u_k=B^{1-\tilde\alpha}_{k,n}$ and reverse the inequalities. 
\end{method}

\Cref{meth:Dir-1s-2s-stepdown} in \cref{sec:app-meth} describes a two-sided stepdown procedure. 

\begin{theorem}\label{thm:Dir-1s-stepdown-FWER}
Under \cref{a:iid,a:F}, \cref{meth:Dir-1s-1s-stepdown,meth:Dir-1s-2s-stepdown} have strong control of finite-sample FWER. 
\end{theorem}

\subsubsection{Pre-test procedure}
\label{sec:1s-power-pretest}

For one-sided multiple testing, a pre-test (actually ``pre-MTP'') can improve pointwise power. 
It can also improve power of the corresponding global test, which in the one-sided case has first-order stochastic dominance as the null hypothesis.%
\footnote{Alternatively, one may test a null of non-dominance as suggested by \citet{DavidsonDuclos2013}; there are also big differences between frequentist and Bayesian inference for first-order stochastic dominance, even asymptotically and with nonparametric methods, as pointed out by \citet{KaplanZhuo2017a} and \citet{Zhuo2017}.} 
Intuitively, the basic MTP for $H_{0\tau} \colon F^{-1}(\tau) \ge F_0^{-1}(\tau)$ over $\tau\in(0,1)$ must control FWER even for the most difficult $F(\cdot)$, where $F^{-1}(\tau)=F_0^{-1}(\tau)$ for all $\tau\in(0,1)$. 
In the GOF context, this $F^{-1}(\cdot)=F_0^{-1}(\cdot)$ is commonly called the ``least favorable configuration'': 
it is the $F(\cdot)$ that maximizes the type I error rate. 
If (by pre-testing) $F(\cdot)$ can be restricted to a subset that excludes the least favorable configuration, 
then we can increase RPs while still controlling FWER. 

Specifically, the pre-test determines at which $\tau$ the constraint $H_{0\tau} \colon F^{-1}(\tau) \ge F_0^{-1}(\tau)$ appears slack, i.e., where we can reject $F^{-1}(\tau)\le F_0^{-1}(\tau)$ in favor of $F^{-1}(\tau)>F_0^{-1}(\tau)$.  
Then, we recalibrate $\tilde\alpha$ using only the unrejected $H_{0\tau}$ to improve power. 

Falsely inferring that the constraint is slack leads to over-rejection of the resulting MTP, so the probability of doing so should be small.  
This probability is the FWER of the pre-test.  
If $\alpha_p$ is the FWER level of the pre-test, then $\alpha_p\to0$ ensures zero asymptotic size distortion.\footnote{This idea is found in \citet{LintonEtAl2010}, whose (13) has $c_N\to0$, and in \citet{DonaldHsu2016}, whose (3.4) has $a_N\to-\infty$, among others.} 
Of course, in any finite sample, $\alpha_p>0$, so $\alpha_p$ should be tolerably small.  
We suggest $\alpha_p=\alpha/\ln\left[\ln\left(\max\{n,15\}\right)\right]$. 

The pre-test implemented in our code is described in \cref{meth:1s-pretest} (and its strong control of FWER in \cref{prop:pretest-FWER}) in \cref{sec:app-meth}.  
The overall method (of which the pre-test is the first step) is described in \cref{meth:Dir-1s-pretest}. 

\begin{method}\label{meth:Dir-1s-pretest}
For \cref{task:1s-test-FWER-1s}, consider $H_{0\tau} \colon F^{-1}(\tau) \ge F_0^{-1}(\tau)$ over $\tau\in(0,1)$. 
First pre-test $H_{0\tau} \colon F^{-1}(\tau) \le F_0^{-1}(\tau)$ for $\tau\in(0,1)$ using \cref{meth:1s-pretest} with strong control of FWER at level $\alpha_p=\alpha/\ln\left[\ln\left(\max\{n,15\}\right)\right]$.  
Let $\hat{K}$ denote the set of $k$ such that $H_{0\ell_k}$ was not rejected by the pre-test, defining $\ell_k$ as in \cref{eqn:def-ell-k-u-k}.  
Then choose integers $r_k\ge k$ such that 
\[  \alpha 
  \ge 1 - \Pr\Bigl( \bigcap_{k\in\hat{K}} X_{n:r_k} \ge F^{-1}(\ell_k) \Bigr)
    = 1 - \Pr\Bigl( \bigcap_{k\in\hat{K}} F(X_{n:r_k}) \ge \ell_k \Bigr)
  , \]
computing the probability using \cref{thm:Wilks}.\footnote{Similar remarks to \cref{foot:rk} apply.} 
Reject $H_{0\tau}$ when $\min\{X_{n:k}:\ell_k\ge\tau,k\in\hat K\}<F_0^{-1}(\tau)$. 

For $H_{0\tau} \colon F^{-1}(\tau) \le F_0^{-1}(\tau)$, reverse inequalities and replace $\ell_k$ with $u_k$ (from \cref{eqn:def-ell-k-u-k}). 
\end{method}

\begin{theorem}\label{thm:Dir-1s-pretest-FWER}
Under \cref{a:iid,a:F}, \cref{meth:Dir-1s-pretest} has strong control of finite-sample FWER at level $\alpha+\alpha_p$, approaching $\alpha$ as $n\to\infty$. 
\end{theorem}

The FWER upper bound $\alpha+\alpha_p$ in \cref{thm:Dir-1s-pretest-FWER} is usually far from binding.  
It assumes that a false pre-test rejection always leads to a false rejection, whereas in reality the probability is only somewhat increased. 
Simulations show the FWER level to be much closer to $\alpha$ than $\alpha+\alpha_p$.

\subsection{Additional modifications to improve power}
\label{sec:1s-other}

\paragraph{Shape restrictions}

Shape restrictions may be imposed by making additional rejections (implied by the restrictions) after our initial MTP is run, without affecting the FWER (assuming the shape restrictions are correct).%
\footnote{We thank Matt Webb for this suggestion.} 
For example, consider a prior belief that the hypotheses $H_{0\tau}$ are false in a single, contiguous range of $\tau$ values.%
\footnote{This is implied by $F^{-1}(\cdot)-F_0^{-1}(\cdot)$ being quasiconcave if $H_{0\tau} \colon F^{-1}(\tau) \le F_0^{-1}(\tau)$, or quasiconvex if $\ge$.} 
If the initial MTP rejects the $H_{0\tau}$ in two disjoint ranges of $\tau$, then additionally rejecting all the values in between does not affect FWER if the shape restriction is correct, as seen by considering the following two possibilities. 
The first possibility is that the initial rejections are all correct: then, the additional rejections must also be correct. 
The second possibility is that at least one of the initial rejections is incorrect: then, a familywise error has already been made, so falsely rejecting additional $H_{0\tau}$ does not affect FWER.%
\footnote{This may raise the question whether FWER is the best measure to use, but that is left to other papers.} 

The same arguments apply to imposing shape restrictions to the two-sample MTP in \cref{sec:2s}.

\paragraph{Uneven sensitivity}

Although we have focused on even sensitivity, the Dirichlet framework can be used to direct power in other patterns, too. 
For example, perhaps only the $H_{0\tau}$ for $\tau\in(0,0.5)$ are of interest.%
\footnote{We thank a referee for this point.} 
Power over those $H_{0\tau}$ can be increased by ignoring $\tau\in(0.5,1)$ and determining $\tilde\alpha$ accordingly. 
That is, the equation (like \cref{eqn:Dir-1s-1s-tilde}) that determines $\tilde\alpha$ can be modified to use only the first $n/2$ order statistics. 
More generally, \cref{eqn:def-ell-k-u-k} can be changed to have $\tilde{\alpha}(k)$ depend on the order statistic, $k$, instead of $\tilde{\alpha}(k)=\tilde{\alpha}$ constant for all $k$. 
For example, to shift power toward the lower half of the distribution without ignoring the upper half, one could set $\tilde\alpha(k)=3a$ if $k<n/2$ and $\tilde\alpha(k)=a$ otherwise. 
Then, a simulation can determine the value of $a$ (and thus $\tilde\alpha(k)$) that delivers correct FWER. 
As long as $\tilde\alpha(k)=g(k,a)$ for some chosen function $g(\cdot)$ and scalar $a$, it is easy to solve for the $a$ that yields exact FWER $\alpha$. 

The downsides of such flexibility are that 
1) it takes time to think about the best $\tilde\alpha(k)$, 
2) it takes time to simulate the proper $\tilde\alpha(k)$ since our formula cannot be used, and 
3) it is tempting to snoop around for the $\tilde\alpha(k)$ that provides the most rejections in a particular dataset. 
Consequently, even if even sensitivity is not exactly desired, it may be more convenient and defensible to use our methods as-is anyway.

\section{Two-sample Dirichlet approach}
\label{sec:2s}

\subsection{Main method and results}

Similar to the two-sample KS-based MTP and the two-sample GOF test in \citet[\S5.2]{BujaRolke2006}, our two-sample MTP depends only on the ordering of $X$ and $Y$ observations (see below). 
The difference is which orderings trigger rejections of which hypotheses. 
Compared to the KS-based MTP, as seen in \cref{fig:sim-2s-pt}, our MTP allocates pointwise size (and thus power) more evenly across the distribution. 

Our two-sample MTP differs from the two-sample GOF test in \citet{BujaRolke2006}. 
Like our one-sample MTP, our two-sample MTP determines a particular $\tilde\alpha$ that depends only on the sample sizes and $\alpha$, after which order statistics are compared to different beta distribution quantiles to determine rejection. 
The \citet[\S5.2]{BujaRolke2006} GOF test uses permutations of the observed data. 
Our approach has two advantages. 
First, computationally, our MTP's pointwise $\tilde\alpha$ may be pre-computed (given $\alpha$, $n_X$, and $n_Y$, as we have done in a large reference table), whereas permutations of observed data require just-in-time computation for each new dataset. 
Second, regarding FWER control, it is not clear that the MTP based on the \citet{BujaRolke2006} GOF test satisfies the assumption of \cref{lem:weak-to-strong}; it may still have strong control of FWER, but it would be more difficult to prove. 

\begin{method}\label{meth:Dir-2s}
Using \cref{eqn:def-ell-k-u-k,eqn:def-EDF-2s}, given $\tilde\alpha$, let 
\begin{equation}\label{eqn:def-hat-ell-u-r}
\begin{split}
\hat{\ell}_X(r) &\equiv B_{n_X\hat{F}_X(r), n_X}^{\tilde\alpha} , \quad
\hat{u}_X(r)     \equiv B_{n_X\hat{F}_X(r)+1, n_X}^{1-\tilde\alpha} , \\
\hat{\ell}_Y(r) &\equiv B_{n_Y\hat{F}_Y(r), n_Y}^{  \tilde\alpha} , \quad
\hat{u}_Y(r)     \equiv B_{n_Y\hat{F}_Y(r)+1, n_Y}^{1-\tilde\alpha} , 
\end{split}
\end{equation}
defining $B_{0,n}^{\tilde\alpha}\equiv0$ and $B_{n+1,n}^{\tilde\alpha}\equiv1$ for any $\tilde\alpha$. 
For \cref{task:2s-test-FWER-CDF-2s}, reject $H_{0r} \colon F_X(r) = F_Y(r)$ when either $\hat{\ell}_X(r)>\hat{u}_Y(r)$ or $\hat{\ell}_Y(r)>\hat{u}_X(r)$. 
Using many simulated samples with $X_i\stackrel{iid}{\sim}\UnifDist(0,1)$ for $i=1,\ldots,n_X$ and independent $Y_j\stackrel{iid}{\sim}\UnifDist(0,1)$ for $j=1,\ldots,n_Y$, choose the largest $\tilde\alpha$ such that the (simulated) probability of rejecting any $H_{0r}$ (i.e., FWER) is less than or equal to $\alpha$, and then subtract $0.0001$ to get the final $\tilde\alpha$. 

For \cref{task:2s-test-FWER-CDF-1s}, reject $H_{0r} \colon F_X(r) \le F_Y(r)$ when $\hat{\ell}_X(r)>\hat{u}_Y(r)$, or reject $H_{0r} \colon F_X(r) \ge F_Y(r)$ when $\hat{\ell}_Y(r)>\hat{u}_X(r)$. 
As above, simulate independent standard uniform datasets and choose the largest $\tilde\alpha$ such that the (simulated) probability of rejecting any $H_{0r}$ (i.e., FWER) is less than or equal to $\alpha$, and then subtract $0.0001$ from $\tilde\alpha$. 
\end{method}

The bands $[\hat{\ell}_X(\cdot), \hat{u}_X(\cdot)]$ and $[\hat{\ell}_Y(\cdot), \hat{u}_Y(\cdot)]$ are uniform confidence bands, but they do not have $1-\alpha$ coverage probability: if they did, then the FWER level would be below $\alpha$. 
Instead, $\tilde\alpha$ is chosen to make the FWER level as close to $\alpha$ as possible. 
Plotting the bands shows the results of the MTP: for any $r$ where they do not overlap, $H_{0r}$ is rejected. 
\Cref{fig:emp} in \cref{sec:emp} shows an empirical example of this. 

As seen in \cref{meth:Dir-2s}, whether or not there is at least one $H_{0r}$ rejected (vs.\ zero rejected) 
depends only on the ordering of $X$ and $Y$ values in the sample. 
For example, if $n_X=n_Y=2$, any sample with $X_{2:1}<X_{2:2}<Y_{2:1}<Y_{2:2}$ has the same ordering, XXYY; either all samples with that ordering 
reject at least one $H_{0r}$ (possibly with different $r$), or all samples accept all $H_{0r}$. 
Consequently, there is only a finite number of $\alpha$ (FWER level) that may be achieved exactly. 
Equivalently, the GOF $p$-value distribution is discrete. 
The same issue of a finite number of attainable $\alpha$ applies to the two-sample KS approach since it also depends on (only) the ordering.%
\footnote{Of course, we could propose a randomized MTP that rejects randomly for certain orderings to achieve exact $\alpha$ FWER level, but randomized tests are (appropriately) not popular in practice.} 

Similar to \cref{meth:Dir-1s-1s,meth:Dir-1s-2s}, the key to FWER control for \cref{meth:Dir-2s} is the choice of $\tilde\alpha$. 
\Cref{meth:Dir-2s} chooses $\tilde\alpha$ such that the FWER is no greater than $\alpha$ under the global null $H_0 \colon F_X(\cdot) = F_Y(\cdot)$; i.e., weak control of FWER is ensured by construction. 
Computationally, we now describe two alternative ways to simulate the mapping from $\tilde\alpha$ to $\alpha$ (given $n_X$ and $n_Y$) when $F_X(\cdot)=F_Y(\cdot)$. 

The first strategy for simulating $\alpha$ given $\tilde\alpha$ (and $n_X$ and $n_Y$) employs a convenient, order-preserving transformation. 
From the probability integral transform (and \cref{a:iid}), $F_Y(Y_i)\stackrel{iid}{\sim}\UnifDist(0,1)$, and under $H_0 \colon F_X(\cdot) = F_Y(\cdot)$, then $F_Y(X_i)=F_X(X_i)\stackrel{iid}{\sim}\UnifDist(0,1)$, too. 
From \cref{a:F}, $F_Y(\cdot)$ is order-preserving. 
Thus, we may simulate independent standard uniform samples of sizes $n_X$ and $n_Y$ to compute the FWER of \cref{meth:Dir-2s} given any $\tilde\alpha$, as suggested in \cref{meth:Dir-2s}. 
Since FWER is monotonic in $\tilde\alpha$, which is a scalar, a simple numerical search finds the $\tilde\alpha$ that leads to a desired $\alpha$. 
These simulations may be done ahead of time to generate a reference table of $\tilde\alpha$ values, as we provide along with our code. 

The second strategy for simulating $\alpha$ given $\tilde\alpha$ uses permutations. 
The distribution of $(X_1,\ldots,X_{n_X},Y_1,\ldots,Y_{n_Y})$ under $H_0 \colon F_X(\cdot) = F_Y(\cdot)$ is the same as that of any permutation of that vector, satisfying the ``randomization hypothesis'' in Definition 15.2.1\ of \citet{LehmannRomano2005text}, for example. 
Given this, \citet{BujaRolke2006} propose a GOF test based on permutations of the observed data, implicitly following Theorem 15.2.1\ of \citet{LehmannRomano2005text}.  
Alternatively, we follow Theorem 15.2.2\ of \citet{LehmannRomano2005text} and use the fact that each of the $\binom{n_X+n_Y}{n_X}$ orderings is equally likely under $H_0$. 
This argument is again distribution-free, so our $\tilde\alpha$ is only a function of $\alpha$, $n_X$, and $n_Y$ and can be computed ahead of time.  

The strong control of FWER in \cref{thm:Dir-2s-FWER} comes from the weak control of FWER (by construction) combined with \cref{lem:weak-to-strong}. 
For the one-sided case, as shown formally in the proofs, $F_X(\cdot)=F_Y(\cdot)$ is the least favorable configuration, so 
FWER is only lower for other distributions satisfying $H_0 \colon F_X(\cdot) \le F_Y(\cdot)$. 

\begin{theorem}\label{thm:Dir-2s-FWER}
Under \cref{a:iid,a:F}, given an arbitrarily large number of simulations, \cref{meth:Dir-2s} has strong control of finite-sample FWER at level $\alpha$. 
\end{theorem}

\subsection{Two-sample quantile MTP: difficulties} 

\Cref{thm:Dir-2s-FWER} does not have a corollary for interpreting \cref{meth:Dir-2s} as a quantile MTP for the hypotheses $H_{0\tau} \colon F_X^{-1}(\tau) = F_Y^{-1}(\tau)$.%
\footnote{We thank the referees for pushing us to determine this definitively.} 
In terms of the proof of \cref{thm:Dir-2s-FWER}, the use of \cref{lem:weak-to-strong} would not be valid: rejection of $H_{0\tau}$ depends on order statistics $X_{n_X:k}$ and $Y_{n_Y:m}$ for some $k$ and $m$, but their finite-sample distributions depend on more than just $F_X^{-1}(\tau)$ and $F_Y^{-1}(\tau)$. 
Specifically, from \cref{thm:Wilks}, $F_X(X_{n_X:k})\sim\BetaDist(k,n_X+1-k)$, so the distribution of $X_{n_X:k}$ depends on all of $F_X^{-1}(\cdot)$ in finite samples. 

More simply, a quantile version of \cref{thm:Dir-2s-FWER} for \cref{meth:Dir-2s} cannot be proved because it is false. 
A counterexample shows this. 
Let $X=0.5$ be a degenerate random variable.%
\footnote{Technically, this violates \cref{a:F}, but $X \sim \UnifDist(0.5-e, 0.5+e)$ for arbitrarily small $e>0$ results in FWER that is arbitrarily close to the $e \to 0$ limit. The same comment applies to the mass points in the distribution of $Y$.} 
Let $\Pr(Y=0)=\Pr(Y=1)=0.5-e$ and $\Pr(Y=U)=2e$, for $U \sim \UnifDist(0,1)$ and small $e>0$. 
Thus, among the $H_{0\tau} \colon F_X^{-1}(\tau)=F_Y^{-1}(\tau)$, only the $\tau=0.5$ hypothesis is true. 
Let $n_X=6$ and $n_Y=12$. 
For $\alpha=0.05$, \cref{meth:Dir-2s} has 
$\tilde{\alpha} = \num[round-mode=figures,round-precision=3]{0.15351700}$. 
Thus, 
$B^{\tilde{\alpha}}_{n_X,n_X}>0.5$, 
$B^{1-\tilde{\alpha}}_{1,n_X}<0.5$, 
$B^{\tilde{\alpha}}_{9,n_Y}>0.5$, 
and $B^{1-\tilde{\alpha}}_{4,n_Y}<0.5$, 
so $H_{0,\tau=0.5}$ is rejected if $Y_{n_Y:9}=0$ or $Y_{n_Y:4}=1$. 
As $e\to0$, $\Pr(Y_{n_Y:9}=0)$ is the probability that no more than $n_Y-9=3$ observations have $Y_i=1$, which is the $\mathrm{Binomial}(n_Y,0.5)$ CDF evaluated at $n_Y-9$, which is 
$\num[round-mode=figures,round-precision=3]{0.07299805}$. 
By symmetry, 
$\Pr(Y_{n_Y:9}=0) = \num[round-mode=figures,round-precision=3]{0.07299805}$, too. 
Thus, the quantile FWER is 
$\num[round-mode=figures,round-precision=3]{0.1459961}$, 
slightly below $\tilde{\alpha}$ but well above $\alpha=0.05$. 
If $\alpha=0.01$, $n_X=6$, and $n_Y=11$, then 
$\tilde{\alpha}=\num[round-mode=figures,round-precision=3]{0.07164299}$, 
and similar calculations show 
$\FWER=\num[round-mode=figures,round-precision=3]{0.06542969}$, 
again slightly below $\tilde{\alpha}$ but well above $\alpha$. 

The preceding counterexample's intuition follows. 
For testing $H_{0\tau} \colon F_X^{-1}(\tau)=F_Y^{-1}(\tau)$ for a single $\tau$, using a Dirichlet-based method similar to ours, \citet[\S3.3]{GoldmanKaplan2017b} show the importance of setting $\tilde{\alpha}$ based not only on $\alpha$, $n_X$, and $n_Y$, but also the ratio of the quantile function derivatives at $\tau$, $Q_X'(\tau) / Q_Y'(\tau)$. 
(This relates to the ratio of asymptotic variances of the corresponding sample quantiles.) 
They show that when this ratio equals one, $\tilde{\alpha}$ should be much larger than $\alpha$ to have near-exact size, but when the ratio approaches zero or infinity, $\tilde{\alpha}=\alpha$ is required. 
In \cref{meth:Dir-2s}, $\tilde{\alpha}$ is calibrated to the case when $F_X(\cdot)=F_Y(\cdot)$ and thus implicitly $Q_X'(\tau)/Q_Y'(\tau)=1$ for all $\tau\in(0,1)$. 
In the counterexample, this implicit assumption is violated: $Q_X'(0.5)=0$ and $Q_Y'(0.5) \to \infty$ as $e \to 0$, so $Q_X'(0.5) / Q_Y'(0.5) = 0$. 
Consequently, the rejection probability is closer to $\tilde{\alpha}$ than $\alpha$. 
However, $\tilde{\alpha} \to 0$ as $n_X,n_Y \to \infty$, so $\tilde{\alpha}>\alpha$ only in very small samples. 
This provides a helpful bound on FWER with this particular DGP (and one could thus control FWER by ensuring $\tilde{\alpha} \le \alpha$), but it is unclear whether such a bound applies to more complex DGPs with multiple true $H_{0\tau}$.

Unlike strong control of FWER, weak control of FWER can be established for the quantile hypotheses. 
For two-sided hypotheses, as in \cref{def:FWER-control}, weak control of FWER for the quantile hypotheses $H_{0\tau}$ means $\FWER \le \alpha$ if all $H_{0\tau}$ are true, i.e., if $F_X^{-1}(\tau)=F_Y^{-1}(\tau)$ for all $\tau\in[0,1]$, or more simply if $F_X^{-1}(\cdot)=F_Y^{-1}(\cdot)$. 
Since $F_X^{-1}(\cdot)=F_Y^{-1}(\cdot)$ is equivalent to $F_X(\cdot)=F_Y(\cdot)$, and since \cref{meth:Dir-2s} has $\FWER \le \alpha$ in that case (i.e., has any rejection with less than $\alpha$ probability), then the quantile interpretation would also have $\FWER \le \alpha$ in that case. 

However, a test with only weak control of FWER is in principle no more informative than a GOF test. 
Consider the MTP that rejects $H_{0\tau}$ for all $\tau\in(0,1)$ whenever a level-$\alpha$ GOF test rejects, and otherwise the MTP rejects none of the $H_{0\tau}$. 
Under $F_X^{-1}(\cdot)=F_Y^{-1}(\cdot)$, the GOF test's rejection probability is below $\alpha$, so the MTP's familywise rejection probability (and thus FWER) is also below $\alpha$, satisfying weak control of FWER. 
However, if enough of the $H_{0\tau}$ are false that the GOF test rejects $80\%$ of the time (as it should), then the MTP has $80\%$ FWER since it falsely rejects all the true $H_{0\tau}$ along with the false $H_{0\tau}$. 
Not only does this technically violate strong control of FWER, it makes the MTP's rejection of a particular $H_{0\tau}$ uninformative in practice: we do not know if that $H_{0\tau}$ is actually false, or if it is rejected only because some other $H_{0\tau}$ is false. 
Our own MTP is not nearly so egregious, with only small FWER distortion even in our highly contrived example (and with proper FWER control in many other examples we tried), but we are reluctant to endorse a method that we know lacks strong control of FWER for the foregoing reason. 

If quantiles are truly desired and distributional hypotheses do not suffice, then one-sample uniform confidence bands for the two quantile functions could be combined, but the result will be conservative. 
The two true quantile functions have probability $1-\alpha$ of both lying in their respective $\sqrt{1-\alpha}$ uniform confidence bands, so the quantile difference function $F_Y^{-1}(\cdot)-F_X^{-1}(\cdot)$ has at least $1-\alpha$ probability of lying in the ``difference'' of the two bands. 
Note that our \cref{meth:Dir-2s} MTP is also constructed (implicitly) using uniform confidence bands, but with lower than $1-\alpha$ coverage, whereas $\sqrt{1-\alpha}>1-\alpha$, so this approach is conservative.

\subsection{Two-sample quantile MTP: methods} 
\label{sec:2s-quantile}

As detailed in \cref{sec:2s-power}, 
as an alternative to \cref{meth:Dir-2s}, we propose an MTP for a smaller number of quantile hypotheses. 
Stepdown and pre-test procedures to improve power are also described. 
Although performance in simulations is reasonable, these methods are not as elegant as the one-sample methods. 
They test fewer than $n$ quantile hypotheses, and we provide only heuristic asymptotic justification. 

\section{Extensions}
\label{sec:ext}

In this section, we discuss how to apply our MTPs to regression discontinuity designs and to conditional distributions. 
It may also be possible to extend our approach to consider families of distributions (e.g., Gaussian), regression residuals, and the changes-in-changes model in \citet{AtheyImbens2006}, among other possibilities.

\subsection{Regression discontinuity}
\label{sec:ext-RD}

Our two-sample MTP may be applied to regression discontinuity (RD) designs. 
We show how recent results from \citet{CattaneoEtAl2015}, \citet{ShenZhang2016}, and \citet{CanayKamat2017} justify the use of our MTP in both sharp and fuzzy RD designs. 
\Cref{sec:emp-RD} contains an empirical example. 

Notationally, let $Y_0$ and $Y_1$ denote the untreated and treated potential outcomes, respectively. 
Let $X$ be the running variable, with threshold $x_0$. 
Let $T$ be the treatment dummy, so the observed outcome is $Y = T Y_1 + (1-T) Y_0$. 

Statistically, our MTP considers the family of hypotheses 
\begin{equation}\label{eqn:RD-H0s}
H_{0r} \colon \Delta(r) = 0 , \; r\in\R , \quad
\Delta(r) \equiv 
  \lim_{x\downarrow x_0} F_{Y|X}(r \mid x)
 -\lim_{x  \uparrow x_0} F_{Y|X}(r \mid x) , 
\end{equation}
or the one-sided hypotheses replacing $=$ with $\ge$ or $\le$. 
This $\Delta(\cdot)$ is the ``reduced-form distributional effect'' of \citet{ShenZhang2016}. 

Economically, under certain assumptions, testing \cref{eqn:RD-H0s} corresponds to testing (local) distributional treatment effects for either sharp or fuzzy RD. 
For sharp RD, \citet{ShenZhang2016} provide conditions (Assumptions 1.4 and 2) under which $\Delta(\cdot)$ equals the distributional treatment effect at $X=x_0$, $F_{Y_1|X}(\cdot\mid x_0)-F_{Y_0|X}(\cdot\mid x_0)$. 
For fuzzy RD, they provide conditions (Assumptions 1.1--1.3) under which $\Delta(\cdot)$ has the same sign as the ``local'' (i.e., for compliers) distributional treatment effect at $X=x_0$. 

\Citet{CattaneoEtAl2015} consider when the ``local experiment'' idea of RD can be taken literally in finite samples. 
Given our \cref{eqn:RD-H0s}, instead of using their Assumption 1 (p.\ 4), the most direct way to justify finite-sample inference is to assume there exists $\underline{x}$ such that $F_{Y|X}(\cdot \mid x)=F^{-}_{Y}(\cdot)$ is independent of $x$ over $x\in[-\underline{x},x_0)$, and similarly $F_{Y|X}(\cdot \mid x)=F^{+}_{Y}(\cdot)$ over some $x\in[x_0,\bar{x}]$. 
If we draw an iid sample including some observations with $X_i\in[-\underline{x},x_0)$ and some with $X_i\in[x_0,\bar{x}]$, then our \Cref{a:iid} is satisfied. 
If additionally $F^{-}_{Y}(\cdot)$ and $F^{+}_Y(\cdot)$ satisfy our \Cref{a:F}, then \cref{thm:Dir-2s-FWER} establishes the strong control of finite-sample FWER of our MTP for \cref{eqn:RD-H0s}. 

\Citet{CanayKamat2017} provide weaker conditions that asymptotically justify a permutation test. 
Their Theorem 4.1 also justifies our MTP. 
Imagine ordering the sample observations by their $X_i$ values and looking at the values closest to $x_0$. 
For some $q$, let 
\begin{equation*}
\cdots 
\le X^{-}_{(q)} \le X^{-}_{(q-1)} 
\le \cdots \le X^{-}_{(1)} 
< 0 
\le X^{+}_{(1)} \le X^{+}_{(2)} \le \cdots \le X^{+}_{(q)} 
\end{equation*}
denote the observed $X_i$ just to the left and right of $x_0$. 
Let the corresponding outcomes be $Y^{-}_{[j]}=Y_k$ if $X^{-}_{(j)}=X_k$, and similarly $Y^{+}_{[j]}=Y_k$ if $X^{+}_{(j)}=X_k$. 
Under 
assumptions on continuity and positive probability of $X$ near $x_0$, 
Theorem 4.1 in \citet{CanayKamat2017} states that for a fixed $q$ as $n\to\infty$, the joint distribution of the $Y^{-}_{[j]}$ and $Y^{+}_{[j]}$ is asymptotically equivalent to independent, iid samples from $\lim_{x\uparrow x_0} F_{Y|X}(\cdot \mid x)$ and $\lim_{x\downarrow x_0} F_{Y|X}(\cdot \mid x)$, respectively. 
These are the two distributions defining $\Delta(\cdot)$ in \cref{eqn:RD-H0s}. 

We now formally state the method and asymptotic strong control of FWER. 
\begin{method}\label{meth:RD}
To test the family of hypotheses in \cref{eqn:RD-H0s}, or the corresponding one-sided versions, apply \cref{meth:Dir-2s} to $(Y^{-}_{[1]}, \ldots, Y^{-}_{[q]})$ and $(Y^{+}_{[1]}, \ldots, Y^{+}_{[q]})$. 
\end{method}
\begin{theorem}\label{thm:RD}
Assume 
(i) iid sampling; 
(ii) $\lim_{x\uparrow x_0} F_{Y|X}(\cdot \mid x)$ and $\lim_{x\downarrow x_0} F_{Y|X}(\cdot \mid x)$ satisfy \cref{a:F}; 
(iii) $F_{Y|X}(y \mid x)$ is left-continuous in $x$ over $x\in[x_0-\epsilon,x_0)$ for some $\epsilon>0$; 
(iv) either $\Pr(X=x_0)>0$ or else $F_{Y|X}(y \mid x)$ is right-continuous in $x$ at $x=x_0$ for all $y\in\R$; 
(v) for any $\epsilon>0$, $\Pr(X\in[x_0-\epsilon,x_0))>0$ and $\Pr(X\in[x_0,x_0+\epsilon])>0$. 
Then, as $n\to\infty$, \cref{meth:RD} has strong control of FWER at a level approaching $\alpha$ asymptotically. 
If (iii) and (iv) are replaced by $F_{Y|X}(\cdot \mid x)=F^{-}_{Y}(\cdot)$ over $x\in[-\underline{x},x_0)$ and $F_{Y|X}(\cdot \mid x)=F^{+}_{Y}(\cdot)$ over $x\in[x_0,\bar{x}]$ for some $\underline{x}<x_0\le\bar{x}$, then \cref{meth:RD} has strong control of finite-sample FWER. 
\end{theorem}

We leave optimal selection of $q$ to future work. 
Smaller $q$ is better for FWER control but worse for power. 
For now, other proposals could be used. 
For example, \citet[\S3]{CattaneoEtAl2015} provide a procedure to determine the largest neighborhood around $x_0$ where the local experiment assumption seems plausible, based on covariate balance. 

\subsection{Conditional distributions}
\label{sec:ext-cond}


Both our one-sample and two-sample MTPs may be applied to conditional distributions.%
\footnote{We thank a referee for this suggestion.}
We focus on the two-sample, two-sided case here. 
It is often of interest to determine where two conditional distributions differ, such as income distributions conditional on different demographic and other characteristics. 
Under conditional independence \citep[e.g.,][Assumption 1, p.\ 545]{MaCurdyEtAl2011}, such a comparison identifies conditional distributional treatment effects. 

Let $Y$ denote the outcome (like income), $\vecf{X}$ a conditioning vector (like education, age, etc.), and $T$ a binary variable (like a dummy for being male). 
The distributional hypotheses are 
\begin{equation}\label{eqn:cond-H0}
H_{0r} \colon F_{Y|\vecf{X},T}(r \mid \vecf{x}_0 , 0) = F_{Y|\vecf{X},T}(r \mid \vecf{x}_0 , 1) , 
\end{equation}
comparing the distribution of $Y$ given $(\vecf{X}=\vecf{x}_0, T=0)$ to the distribution of $Y$ given $(\vecf{X}=\vecf{x}_0, T=1)$. 

If $\vecf{X}$ contains only discrete variables, 
then our MTPs can apply immediately. 
Assuming the sample contains some observations with 
$(\vecf{X}_i,T_i)=(\vecf{x}_0,0)$ and some with 
$(\vecf{X}_i,T_i)=(\vecf{x}_0,1)$, the MTP can use the corresponding $Y_i$ as the two respective (local) samples. 

If $\vecf{X}$ contains any continuous variables, then $\Pr(\vecf{X}=\vecf{x})=0$ for any $\vecf{x}$, so we must smooth over observations with $\vecf{X}_i$ near (but not equal to) $\vecf{x}_0$. 
First, among observations with $T_i=0$, order the $\vecf{X}_i$ by their distance from $\vecf{x}_0$ according to some norm, $\norm{\cdot}$. 
Specifically, let $\vecf{X}_{(k),T=0}$ satisfy $\norm{\vecf{X}_{(1),T=0}-\vecf{x}_0} \le \norm{\vecf{X}_{(2),T=0}-\vecf{x}_0} \le \cdots$, with each $\vecf{X}_{(k),T=0}=\vecf{X}_i$ for some $i$ where $T_i=0$ (with each $i$ used exactly once). 
Define $\vecf{X}_{(k),T=1}$ similarly for the observations with $T_i=1$. 
If the norm is chosen to make $\norm{\vecf{x}}$ extremely large when the discrete components of $\vecf{x}$ are non-zero, then the smallest values of $\norm{\vecf{X}_i-\vecf{x}_0}$ will have the discrete components of $\vecf{X}_i$ exactly equal those of $\vecf{x}_0$. 
Second, let $Y_{[k],T=0}$ and $Y_{[k],T=1}$ denote the corresponding outcomes: $Y_{[k],T=0}=Y_i$ iff $\vecf{X}_{(k),T=0}=\vecf{X}_i$, and similarly for $Y_{[k],T=1}$. 
Third, apply our two-sample MTP to the local samples 
\begin{equation}\label{eqn:CQD-samples}
\vecf{Y}_{n,T=0} \equiv ( Y_{[1],T=0}, \ldots, Y_{[q_0],T=0} ) , \quad
\vecf{Y}_{n,T=1} \equiv ( Y_{[1],T=1}, \ldots, Y_{[q_1],T=1} ) . 
\end{equation}
We leave optimal selection of $q_1$ and $q_0$ to future work; as in \cref{sec:ext-RD}, smaller $q$ is better for FWER but worse for power. 

\begin{method}\label{meth:CQD}
Let $H_{0r} \colon F_{Y|\vecf{X},T}(r \mid \vecf{x}_0, 0) = F_{Y|\vecf{X},T}(r \mid \vecf{x}_0, 1)$, $r\in\R$. 
Given $q_0$, $q_1$, and $\norm{\cdot}$, determine local samples $\vecf{Y}_{n,T=0}$ and $\vecf{Y}_{n,T=1}$ as in \cref{eqn:CQD-samples}. 
Apply \cref{meth:Dir-2s} to $\vecf{Y}_{n,T=0}$ and $\vecf{Y}_{n,T=1}$. 
\end{method}

Theoretically, as in \cref{sec:ext-RD}, 
asymptotic justification of \cref{meth:CQD} 
comes from Theorem 4.1 of \citet{CanayKamat2017}. 
Their result implies that as $n\to\infty$, the local samples in \cref{eqn:CQD-samples} converge in distribution to independent, iid samples from the conditional distributions of interest, $F_{Y|\vecf{X},T}(\cdot \mid \vecf{x}_0 , 0)$ and $F_{Y|\vecf{X},T}(\cdot \mid \vecf{x}_0 , 1)$. 
\begin{theorem}\label{thm:CQD}
Let $\vecf{X}=(\vecf{X}^D,\vecf{X}^C)$ and $\vecf{x}_0=(\vecf{x}^D_0,\vecf{x}^C_0)$, where $\vecf{X}^D$ contains $d_D$ discrete random variables satisfying $\Pr(\vecf{X}^D=\vecf{x}^D_0)>0$, and $\vecf{X}^C$ contains $d_C$ continuous (or mixed) random variables, so $\dim(\vecf{X})=d\equiv d_D+d_C$. 
Assume for both $t=0$ and $t=1$: 
(i) $F_{Y|\vecf{X}^C,\vecf{X}^D,T}(y \mid \vecf{x}^C, \vecf{x}^D, t)$ is continuous in $\vecf{x}^C$ at $(\vecf{x}^C,\vecf{x}^D)=(\vecf{x}^C_0,\vecf{x}^D_0)$ for all $y$ in the support of $Y$; 
(ii) $\Pr\bigl( (\vecf{X}^C - \vecf{x}^C_0) \in [-\epsilon,\epsilon]^{d_C} \mid \vecf{X}^D=\vecf{x}^D_0, T=t \bigr) > 0$ for all $\epsilon>0$. 
Further assume: 
(iii) iid sampling; 
(iv) $F_{Y|\vecf{X}^C,\vecf{X}^D,T}(y \mid \vecf{x}^C_0, \vecf{x}^D_0, 0)$ and $F_{Y|\vecf{X}^C,\vecf{X}^D,T}(y \mid \vecf{x}^C_0, \vecf{x}^D_0, 1)$ satisfy \cref{a:F}. 
Then, as $n\to\infty$, \cref{meth:CQD} has strong control of FWER at a level approaching $\alpha$ asymptotically. 
If $d_C=0$, then \cref{meth:CQD} has strong control of FWER in finite samples. 
\end{theorem}


\section{Empirical examples}
\label{sec:emp}

\subsection{Gift wage experiment}
\label{sec:emp-gift}

We revisit data from \citet[Tables I and V]{GneezyList2006}.
Results may be replicated using code from the latter author's website. 
The global $p$-values are from the method proposed by \citet{BujaRolke2006}, implemented in our code with our faster computation. 

The experiment of \citet{GneezyList2006} pays control group individuals an advertised hourly wage and treatment group individuals an unexpectedly larger ``gift'' wage upon arrival.  
The ``gift exchange'' question is whether the higher wage induces higher effort in return.  
The experiment is run separately for library data entry and door-to-door fundraising tasks.  
The sample sizes are small: $10$ and $9$ for control and treatment (respectively) for the library task, and $10$ and $13$ for fundraising. 
With small samples, our methods' finite-sample FWER control is especially desirable. 

The main finding of \citet{GneezyList2006} is that the ``gift wage'' treatment raises productivity significantly in the first time period but not thereafter.  
Complementing the original results, we examine heterogeneity in the period 1 treatment effect, testing across the productivity distribution.  

\Cref{fig:emp} shows the two-sided bands used by \cref{meth:Dir-2s}, $[\hat{\ell}_X(\cdot),\hat{u}_X(\cdot)]$ and $[\hat{\ell}_Y(\cdot),\hat{u}_Y(\cdot)]$. 
Wherever the bands do not overlap, the pointwise null hypothesis $H_{0r} \colon F_X(r) = F_Y(r)$ is rejected. 
With such small sample sizes, discreteness precludes an exact $10\%$ FWER level, so exact $8.5\%$ (library) and $9.3\%$ (fundraising) FWER levels are used instead. 

For the library task, with $8.5\%$ FWER level, our MTP does not reject equality of the treatment and control productivity CDFs at any point.  
However, there is almost a rejection near 56--58 books, around the upper quartile; increasing the FWER level to $14\%$ triggers rejection here.  
With one-sided global testing, i.e., testing first-order stochastic dominance, the Dirichlet test cannot reject that the treatment distribution dominates the control distribution ($p=0.996$), whereas it does reject at a $10\%$ level that the control distribution dominates the treatment distribution ($p=0.076$) because of the pointwise rejection near the $0.8$-quantile. 
In contrast, the KS test fails to reject at a $10\%$ level: its one-sided $p$-value is $0.13$. 

\begin{figure}[htbp]
\centering
\hfill
\includegraphics[width=0.45\textwidth,clip=true,trim=5 25 10 10]{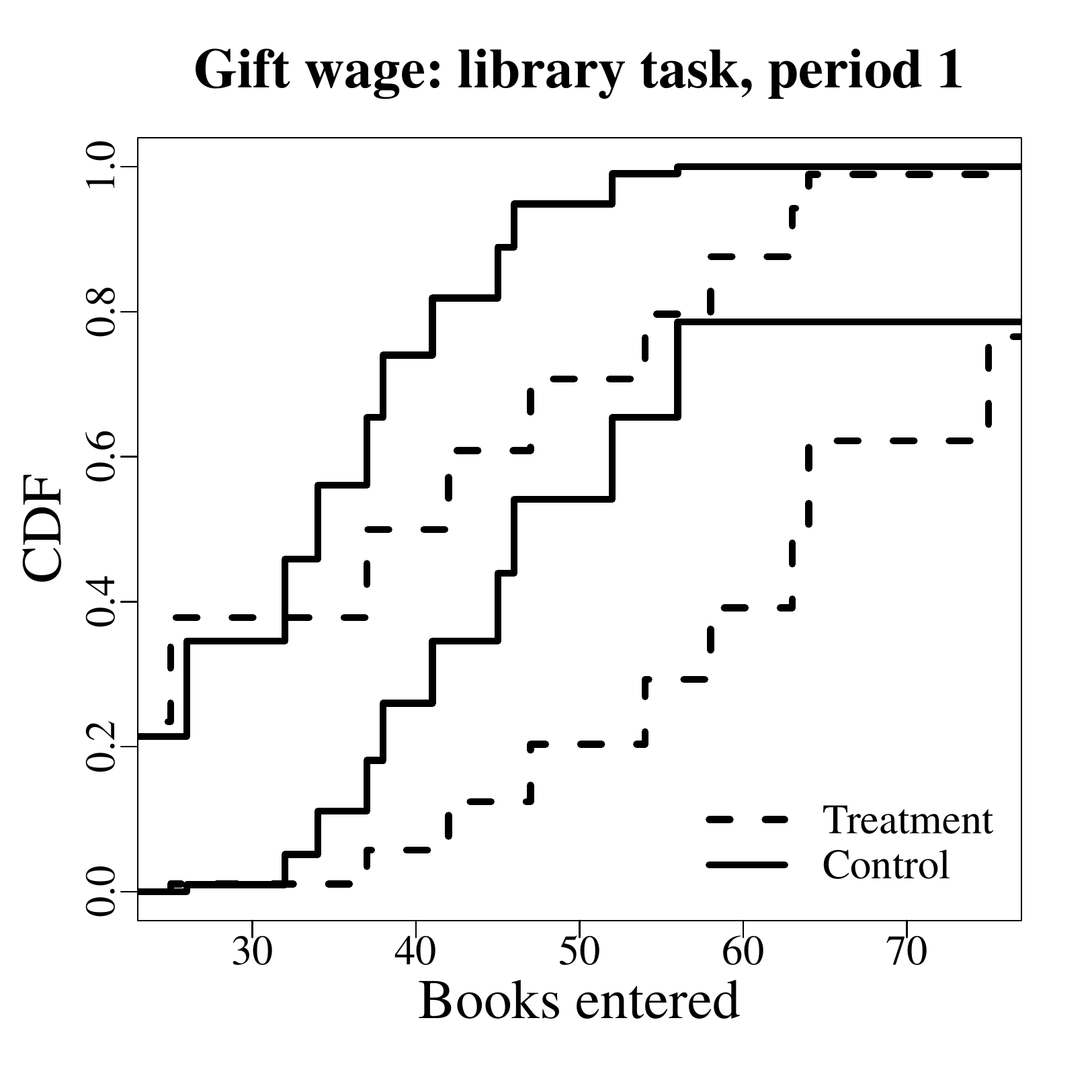}%
\hfill%
\includegraphics[width=0.45\textwidth,clip=true,trim=5 25 10 10]{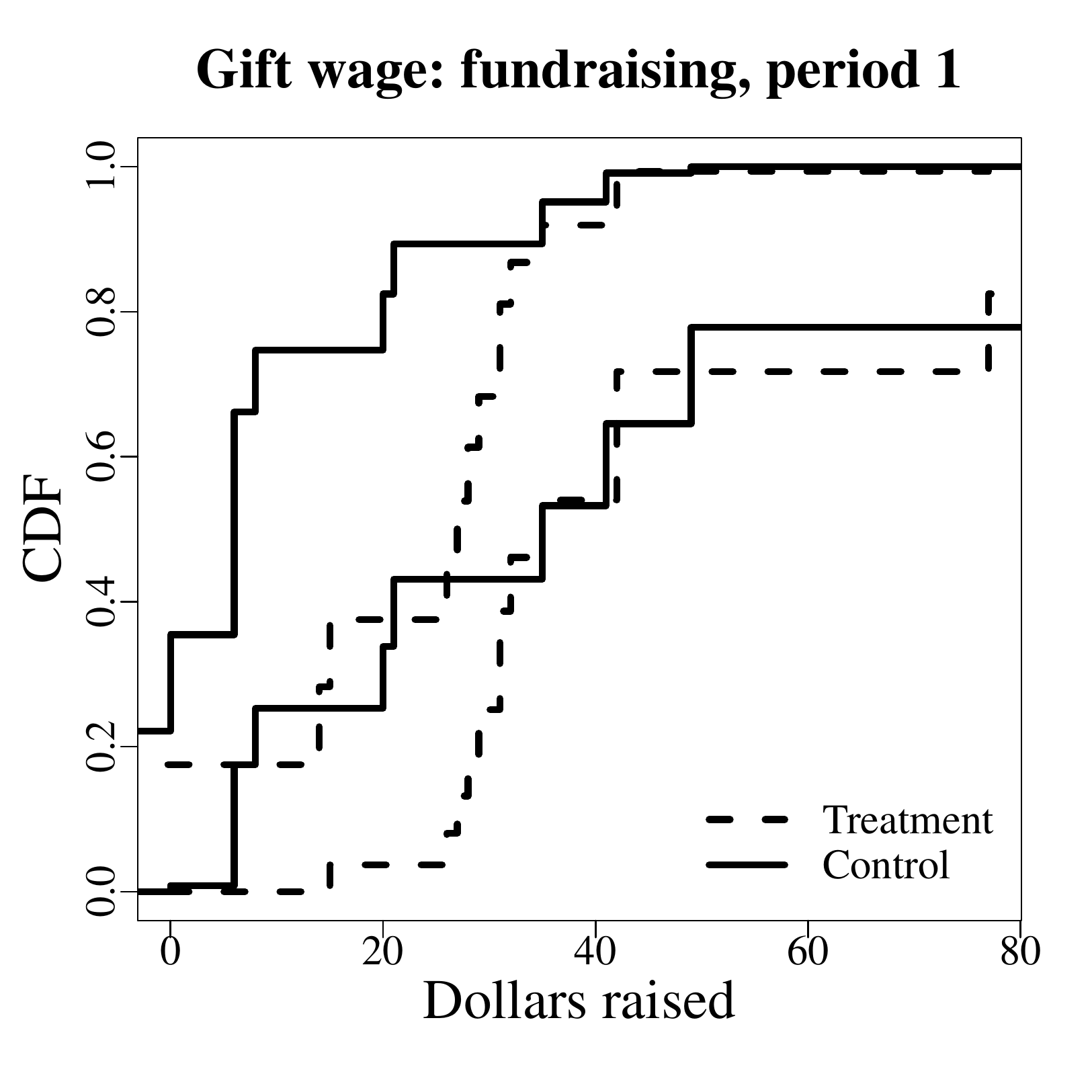}%
\hfill\null
\caption{\label{fig:emp}Comparison of treatment and control group productivity in the first period of the library (left) and fundraising (right) tasks in \citet{GneezyList2006}: bands for two-sided MTP (rejecting wherever there is no overlap) with exact FWER levels $8.5\%$ (left) and $9.3\%$ (right).}
\end{figure}

For the fundraising task, with $9.3\%$ FWER level, \cref{fig:emp} shows two ranges near the lower quartile of the distributions where a zero treatment effect is rejected: 8--14 and 21--26 dollars raised. 
%
Opposite the library task, where the upper part of the distribution showed the most significant (statistically and economically) treatment effect, the gift wage treatment most affects the lower part of the fundraising distribution. 
Testing first-order stochastic dominance, the Dirichlet test cannot reject that the treatment dominates the control distribution ($p=1$), whereas it can reject at a $5\%$ level that the control distribution dominates the treatment distribution ($p=0.021$).  
The one-sided KS $p=0.034$: higher, but still below $0.05$. 
For two-sided testing of zero treatment effect, the Dirichlet test rejects at a $5\%$ level while the KS cannot: the Dirichlet $p=0.041$, while the KS $p=0.069$. 

\begin{table}[htbp]
\centering
\begin{threeparttable}
\caption{\label{tab:emp}Intervals of $r$ for which $H_{0r}$ is rejected.}
\begin{tabular}{lrccrcc}
\toprule
 && \multicolumn{2}{c}{Library} 
 && \multicolumn{2}{c}{Fundraising} \\
\cmidrule{3-4}\cmidrule{6-7}
Method && 2-sided & 1-sided 
       && 2-sided & 1-sided \\
\midrule
\multicolumn{7}{l}{\textit{FWER level: $\alpha=0.05$}} \\
Dirichlet && none & none 
          && $(8,14)$ & $(8,14)\cup(21,26)$ \\
KS-based  && none & none 
          && none & $(21,26)$ \\[6pt]
\multicolumn{7}{l}{\textit{FWER level: $\alpha=0.1$}} \\
Dirichlet && none & $(56,58)$ 
          && $(8,14)\cup(21,26)$ & $(6,15)\cup(20,27)$ \\
KS-based  && none & none 
          && $(21,26)$ & $(8,13)\cup(21,26)$ \\
\bottomrule
\end{tabular}
\begin{tablenotes}
\item \textbf{Note:} Units are books entered (library) or dollars raised (fundraising). 
With $F_T(\cdot)$ the treated population CDF and $F_C(\cdot)$ the control CDF, ``2-sided'' means $H_{0r} \colon F_T(r) = F_C(r)$, and ``1-sided'' means $H_{0r} \colon F_T(r) \ge F_C(r)$.
\end{tablenotes}
\end{threeparttable}
\end{table}

\Cref{tab:emp} summarizes the results from running the Dirichlet and KS-based MTPs. 
In addition to the Dirichlet rejecting some $H_{0r}$ in many cases where the KS-based MTP cannot, the Dirichlet MTP rejects more $H_{0r}$ in cases where both MTPs reject at some $r$. 
The economic interpretation of the rejection of $H_{0r}$ for $r\in(56,58)$ for the library task (one-sided, $\alpha=0.1$) can be expressed as: 
 the data suggest that the gift wage treatment increases the probability of an individual entering at least $57$ books. 

As discussed in \cref{sec:1s-other}, a prior belief that the gift wage treatment should only affect a single, contiguous range of $r$ values would lead to additional rejections in the fundraising data. 
The two-sided $H_{0r}$ for all $r\in(8,26)$ would be rejected, and similarly $r\in(8,26)$ or $r\in(6,27)$ for one-sided $H_{0r}$ with $\alpha=0.05$ or $\alpha=0.1$, respectively.

\subsection{Regression discontinuity}
\label{sec:emp-RD}



Applying our MTP to a regression discontinuity (RD) design, we study the incumbency advantage in the U.S.\ Senate using the same data from \citet{CattaneoEtAl2015}.%
\footnote{\url{https://sites.google.com/site/rdpackages/rdlocrand/r/rdlocrand_senate.csv}} 
The running variable $Z$ measures the Democratic candidate's margin of victory in an election, in percentage points; it is negative if the Democrat loses to the Republican. 
The outcome $Y$ is the vote share of the Democratic candidate in the next election for the same Senate seat, six years later. 
This is a sharp RD where the incumbency ``treatment'' dummy is simply $D=\Ind{ Z>0 }$. 
For other data details, see \citet{CattaneoEtAl2015}. 

The economic question is how incumbency (i.e., currently holding the seat) affects voting. 
\Citet{CattaneoEtAl2015} argue that for close elections, the results (and thus the incumbent in the next election) are as good as random. 
To quantify ``close,'' \citet{CattaneoEtAl2015} apply their window selection procedure (based on tests of covariate balance) to get a percentage Democratic margin of victory of $-0.75 \le Z \le 0.75$. 
We take their suggestion and apply our two-sample Dirichlet MTP (\cref{meth:Dir-2s}) to the ``control'' group outcomes $\{ Y_i : -0.75 \le Z_i < 0\}$ and the ``treatment'' group outcomes $\{ Y_i : 0 < Z_i < 0.75 \}$. 

\begin{figure}[htbp]
\centering
\includegraphics[width=0.55\textwidth,clip=true,trim=20 10 5 85]{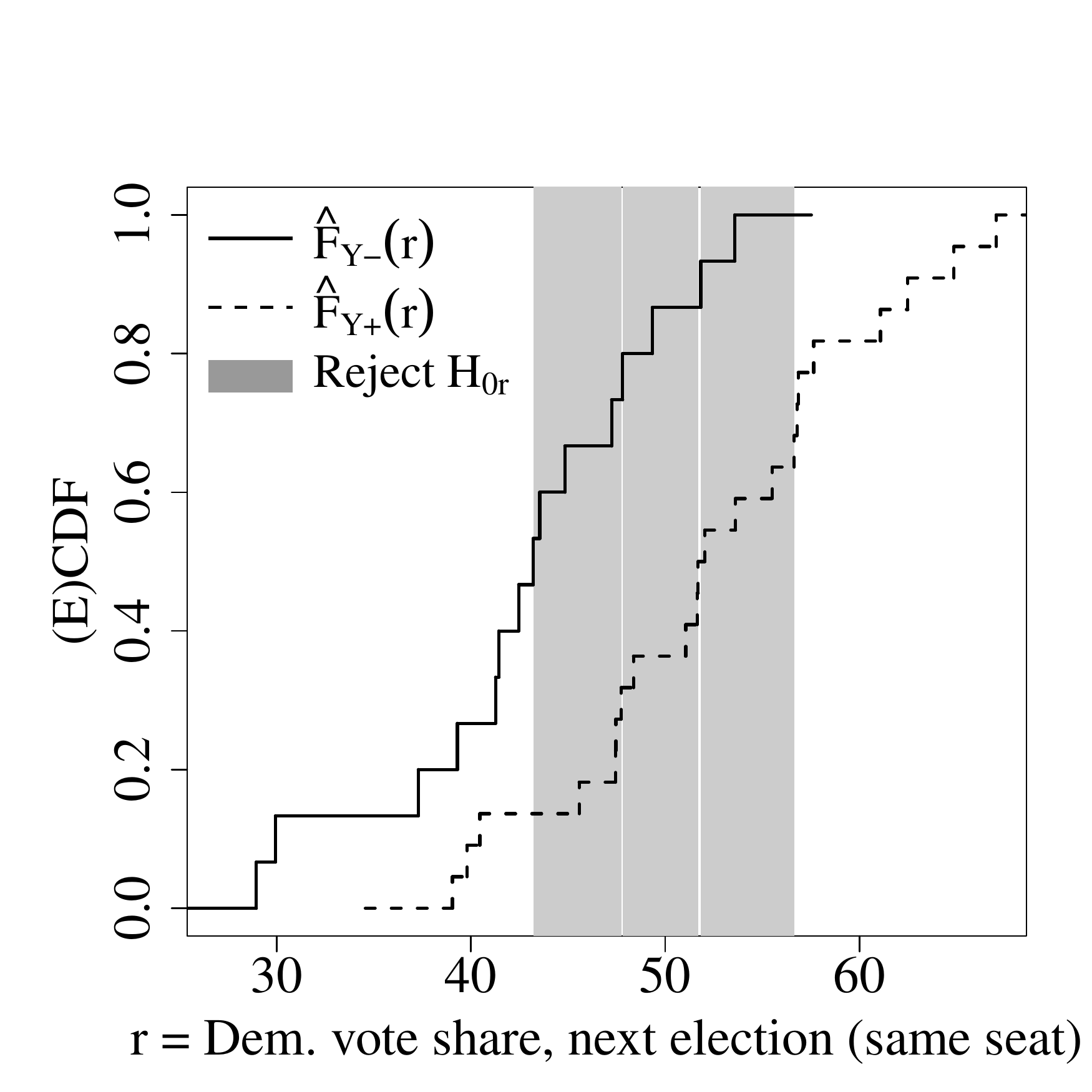}%
\caption{\label{fig:emp-RD}Application of Dirichlet MTP to U.S.\ Senate elections data; one-sided, nominal $5\%$ FWER level. In the legend, $\hat{F}_{Y-}(\cdot)$ is the empirical CDF of $\{ Y_i : -0.75 \le Z_i < 0\}$, while $\hat{F}_{Y+}(\cdot)$ is the empirical CDF of $\{ Y_i : 0 < Z_i < 0.75 \}$.}
\end{figure}

\Cref{fig:emp-RD} show the results. 
Like \citet{CattaneoEtAl2015}, we find a very low $p$-value for the global null hypothesis that the control and treatment distributions are identical. 
Unlike the permutation test in \citet{CattaneoEtAl2015}, our MTP assesses each pointwise $H_{0r}$.%
\footnote{This can be done with a permutation test only by invoking additional strong assumptions like a constant treatment effect, as in Assumption 3 of \citet{CattaneoEtAl2015}. Section 2.3 of \citet{CattaneoEtAl2015} also discusses quantile treatment effect confidence intervals, but they are based on intervals for the quantiles of the marginal distributions, which makes them conservative, and they are only pointwise for individual $\tau$, not the full distribution.} 
With one-sided FWER level $\alpha=0.05$, our MTP rejects $H_{0r}$ for almost all $r\in[43.2, 56.6]$ (percent). 
Beyond just saying that incumbency affects some (unspecified) part of the vote share distribution, the MTP results show statistical significance for the beneficial incumbency effect across most of the distribution, implying at least a restricted first-order stochastic dominance relationship. 


\section{Simulations}
\label{sec:sim}

All simulations may be replicated with code from the latter author's website.  
For comparison with our Dirichlet methods, we use KS-based MTPs as described in \cref{sec:KS}. 
The unweighted KS-based MTP uses the KS implementation \texttt{ks.test} in the \texttt{stats} package in R \citep{R.core}. 
For the weighted KS, asymptotic critical values from \citet{Jaeschke1979} and \citet{ChicheporticheBouchaud2012} were inaccurate,\footnote{\citet[p.\ 108]{Jaeschke1979} appropriately warns, ``Since\ldots the rate of convergence\ldots is very slow, we would not encourage anyone to use the confidence intervals based on the asymptotic analysis.''} 
so we simulate exact critical values.  
However, this simulation is time-consuming, which is a practical disadvantage. 

Earlier, \cref{fig:sim-1s-pt,fig:sim-2s-pt} showed simulation results on the uneven sensitivity of weighted and unweighted KS-based MTPs and the (relatively) even sensitivity of the Dirichlet MTPs, in terms of pointwise type I error rates. 
Intuitively, those differences translate into differences in pointwise power, as shown in \cref{sec:sim-pt-pwr}. 
The Dirichlet's more even sensitivity also achieves generally better global power than the (GOF) KS test. 
This is illustrated in \cref{sec:sim-pt-pwr}, as well as in Table 1 and Figure 8 of \citet{AldorNoimanEtAl2013}, who also show a power advantage over the Anderson--Darling (i.e., weighted Cram\'er--von Mises) test for a variety of distributions. 
That is, there is not a tradeoff between even sensitivity and global power; the Dirichlet approach has both more even sensitivity and better global power. 

In this section, we focus on our methods' strong control of FWER, the power improvements of stepdown and pre-test procedures, and the computational benefit of \cref{fact:alpha-tilde-rate}.

\subsection{FWER}\label{sec:sim-FWER}

\Cref{tab:sim-1s-overall} shows weak control of FWER for one-sample, two-sided MTPs, i.e., it shows FWER simulated under $F(\cdot)=F_0(\cdot)$. 
Since all MTPs considered are distribution-free under \cref{a:iid,a:F}, the DGP is $X_i\stackrel{iid}{\sim}\UnifDist(0,1)$.  
For our \cref{meth:Dir-1s-2s} (``Dirichlet''), this is exact by construction, up to the approximation error in \cref{fact:alpha-tilde-rate}.  
This error is negligible in \cref{tab:sim-1s-overall}.  
Earlier, \cref{fig:approx-err-hists} showed simulated FWER for additional $n$ and nominal $\alpha$ when using \cref{fact:alpha-tilde-rate}. 

\begin{table}[htbp]
\centering
\begin{threeparttable}
\caption{\label{tab:sim-1s-overall}Simulated FWER, one-sample, two-sided.}
\begin{tabular}{cccccc}
\toprule
$\alpha$ & $n$ & Dirichlet & KS & KS (exact) & weighted KS (exact) \\
\midrule
0.10 & \phantom{0}20 & 0.101 & 0.100 & 0.100 & 0.099 \\
0.10 & 100           & 0.101 & 0.094 & 0.100 & 0.098 \\
0.05 & \phantom{0}20 & 0.050 & 0.050 & 0.050 & 0.053 \\
0.05 & 100           & 0.050 & 0.045 & 0.050 & 0.049 \\
\bottomrule
\end{tabular}
\begin{tablenotes}
\item \textbf{Note:} $F(\cdot)=F_0(\cdot)$, $10^6$ replications. 
\end{tablenotes}
\end{threeparttable}
\end{table}

\Cref{tab:sim-1s-FWER} shows strong control of FWER for one-sample, one-sided MTPs: the basic Dirichlet test in \cref{meth:Dir-1s-1s}, as well as the stepdown and pre-test procedures in \cref{meth:Dir-1s-1s-stepdown,meth:Dir-1s-pretest}, respectively.  
The null distribution $F_0$ is $\UnifDist(-1,1)$ and $H_{0\tau} \colon F^{-1}(\tau) \le F_0^{-1}(\tau)$. 
\Cref{fig:H1s} shows $F_0^{-1}(\cdot)$ and $F^{-1}(\cdot)$ for each row in \cref{tab:sim-1s-FWER}. 
The Dirichlet MTP in \cref{meth:Dir-1s-1s} always controls FWER, but well below the required level $\alpha=0.1$ when $F(\cdot)\ne F_0(\cdot)$.  
The FWERs for the stepdown method and combined pre-test/stepdown method are higher but still below $\alpha=0.1$, as desired. 
Of course, all else equal, a higher error rate is never desired, but (as seen later) there is a corresponding gain in power, so the tradeoff may be beneficial from a minimax risk sort of perspective: a slight increase in FWER when FWER is near zero is not very costly, while improving worst-case power near zero is very beneficial. 

\begin{table}[htbp]
\centering
\begin{threeparttable}
\caption{\label{tab:sim-1s-FWER}Simulated FWER, one-sample, one-sided, 
$\alpha=0.1$, $n=100$.}
\begin{tabular}{llccc}
\toprule
$H_{0\tau}\textrm{ true}$ & $F^{-1}(\tau)=F_0^{-1}(\tau)$ & 
Dirichlet & Stepdown & Pre+Step \\
\midrule
$\tau\in[0,1]$   & $\tau\in[0,1]$   & 0.101 & 0.101 & 0.101 \\
$\tau\in[0,0.5]$ & $\tau\in[0,0.5]$ & 0.048 & 0.083 & 0.082 \\
$\tau\in[0,1]$   & $\tau\in[0.5,1]$ & 0.068 & 0.068 & 0.079 \\
$\tau\in[0,0.5]$ & $\tau\in\{0.5\}$ & 0.004 & 0.017 & 0.024 \\
\bottomrule
\end{tabular}
\begin{tablenotes}
\item \textbf{Note:} $H_{0\tau} \colon F^{-1}(\tau) \le F_0^{-1}(\tau)$, $F_0=\UnifDist(-1,1)$ so $F_0^{-1}(\tau)=2(\tau-0.5)$, $n=100$, $\num{1000}$ replications.  
For $\tau$ where $F^{-1}(\tau)\ne F_0^{-1}(\tau)$, $F^{-1}(\tau)=4(\tau-0.5)$; see \cref{fig:H1s}.  
``Dirichlet'' is \cref{meth:Dir-1s-1s}, ``Stepdown'' is \cref{meth:Dir-1s-1s-stepdown}, and ``Pre+Step'' is \cref{meth:Dir-1s-1s-stepdown,meth:Dir-1s-pretest} combined.
\end{tablenotes}
\end{threeparttable}
\end{table}

\begin{figure}[htbp]
\centering
\hfill
\includegraphics[width=0.24\textwidth,clip=true,trim=0 0 0 0]{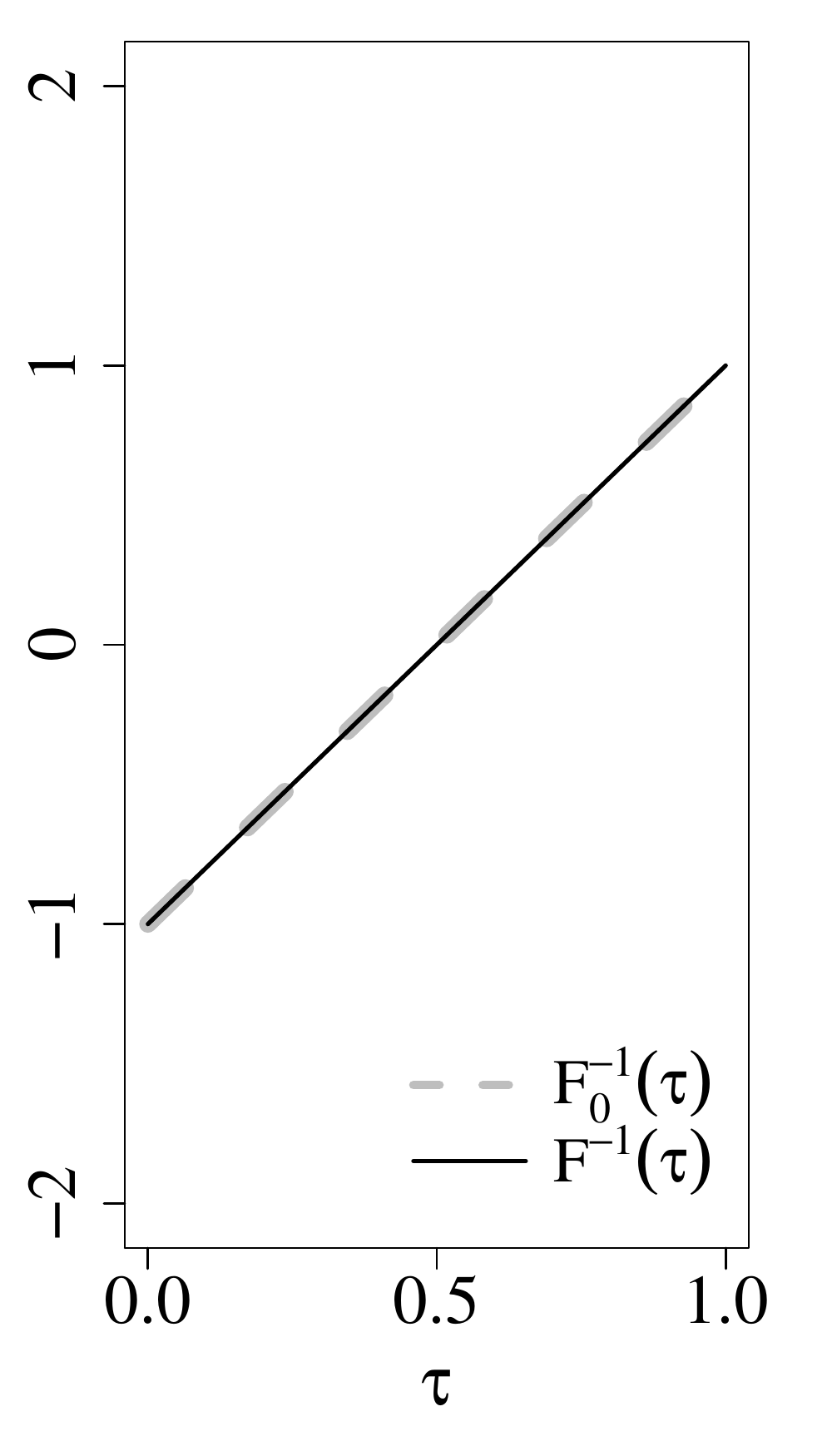}%
\hfill%
\includegraphics[width=0.24\textwidth,clip=true,trim=0 0 0 0]{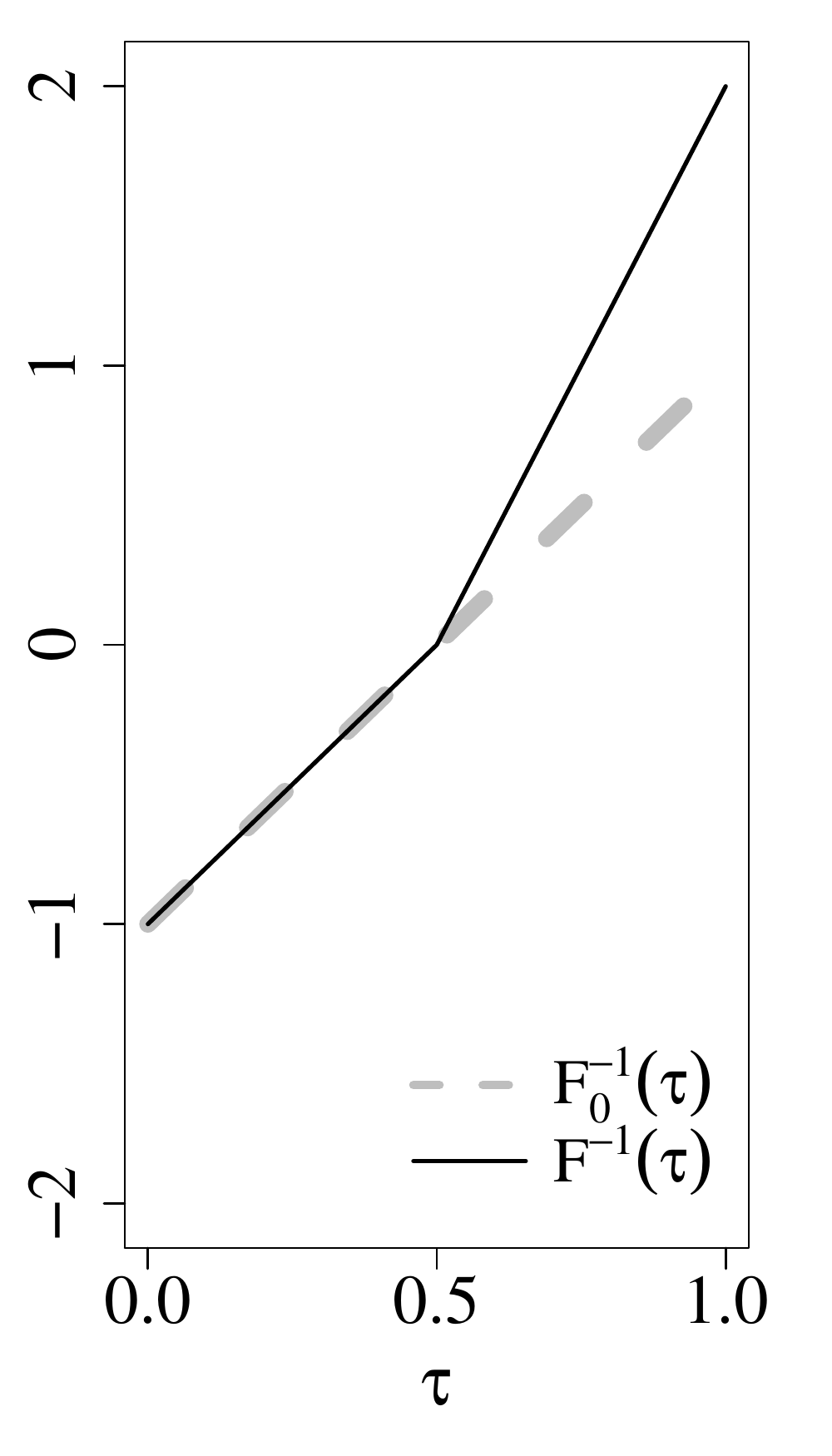}%
\hfill%
\includegraphics[width=0.24\textwidth,clip=true,trim=0 0 0 0]{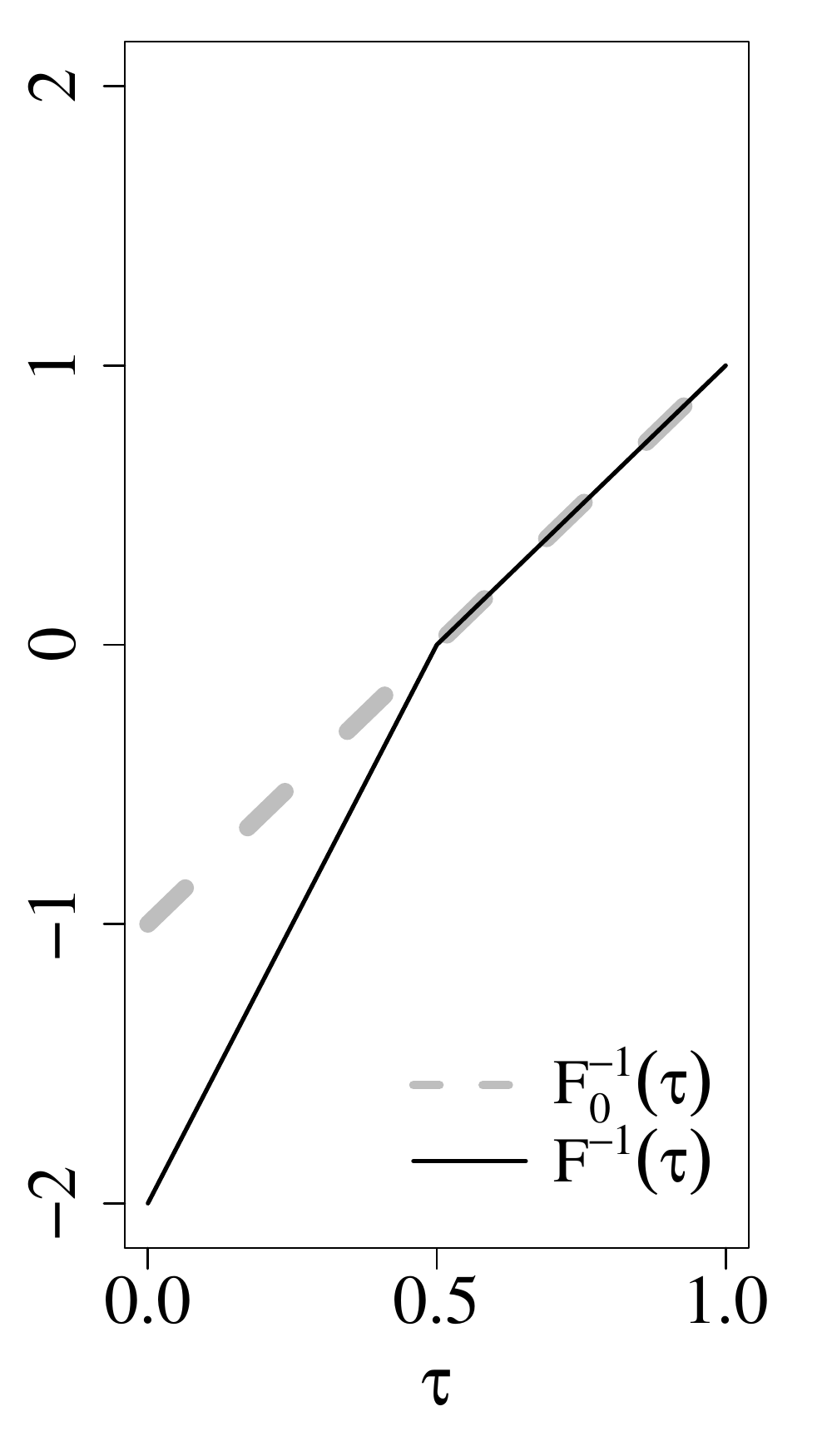}%
\hfill%
\includegraphics[width=0.24\textwidth,clip=true,trim=0 0 0 0]{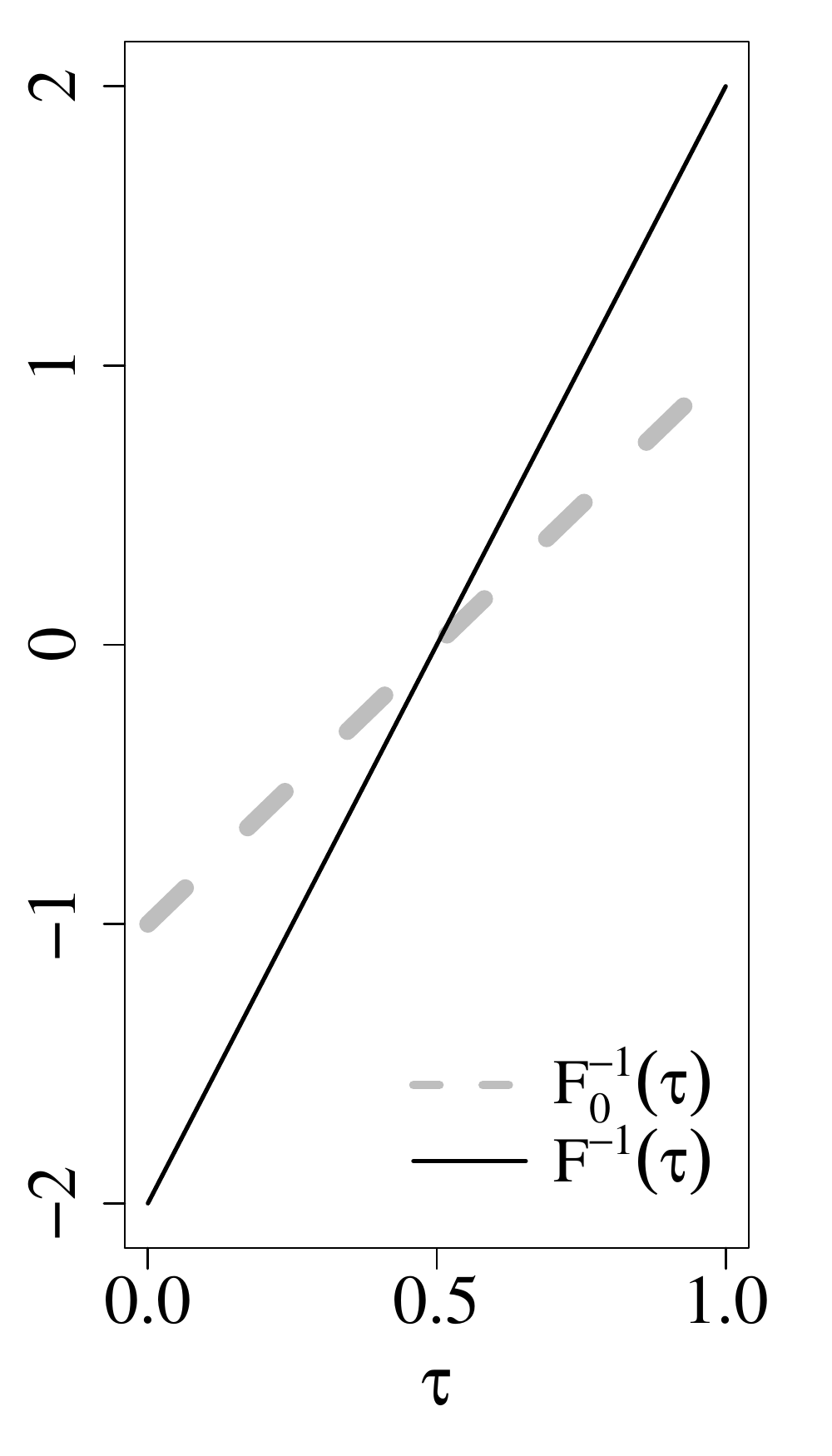}%
\hfill\null
\caption{\label{fig:H1s}Null and true quantile functions, $F_0^{-1}(\cdot)$ and $F^{-1}(\cdot)$, for the four rows in \cref{tab:sim-1s-FWER}.}
\end{figure}

\Cref{tab:sim-2s-size} shows weak control of FWER for two-sample, two-sided MTPs, i.e., FWER under $H_0 \colon F_X(\cdot) = F_Y(\cdot)$.  
Since all MTPs shown are distribution-free in this case, both samples are iid $\UnifDist(0,1)$.  

\begin{table}[htbp]
\centering
\caption{\label{tab:sim-2s-size}Simulated FWER, two-sample, two-sided.}
\begin{threeparttable}
\begin{tabular}{crrccc}
\toprule
$\alpha$ & \multicolumn{1}{c}{$n_X$} & \multicolumn{1}{c}{$n_Y$} & Dirichlet & KS & KS (exact) \\
\midrule
0.05 &  25 & 500 & 0.050 & 0.039 & 0.049 \\
0.10 &  25 & 500 & 0.100 & 0.082 & 0.095 \\
0.10 &  30 &  30 & 0.101 & 0.071 & 0.071 \\
0.10 &  29 &  30 & 0.101 & 0.079 & 0.099 \\
0.10 & 100 & 100 & 0.101 & 0.078 & 0.078 \\
0.10 &  99 & 100 & 0.106 & 0.090 & 0.099 \\
\bottomrule
\end{tabular}
\begin{tablenotes}
\item \textbf{Note:} $F_X(\cdot)=F_Y(\cdot)$, $10^6$ replications. 
\end{tablenotes}
\end{threeparttable}
\end{table}
%

\Cref{tab:sim-2s-size} shows our MTP's nearly exact FWER.  
The asymptotic KS-based MTP is somewhat conservative in these cases, as is the ``exact'' KS-based MTP. 
The exact and asymptotic KS can be identical due to the discreteness of the GOF $p$-value distributions (as discussed in \cref{sec:2s}), if the exact and asymptotic $p$-values lie on the same side of $\alpha=0.1$ for every possible data ordering (permutation). 
This discreteness makes the exact KS notably conservative when $n_X=n_Y=30$ and even $n_X=n_Y=100$, but the effect vanishes when reducing $n_X$ by one so that $n_X\ne n_Y$. 
The effect of discreteness on the Dirichlet MTP is negligible in all cases.

\Cref{tab:sim-2s-FWER} is the two-sample analog of \cref{tab:sim-1s-FWER}, showing strong control of FWER for one-sided MTPs. 
\Cref{fig:H1s} again visualizes the quantile functions for each row of the table, but now $F_Y=F_0$ and $F_X=F$. 
We compare MTPs for \cref{task:2s-test-FWER-CDF-1s,task:2s-test-FWER-1s} with $\alpha=0.05$: the basic Dirichlet MTP in \cref{meth:Dir-2s}, 
the joint quantile difference MTP in \cref{sec:2s-power}, 
the stepdown procedure in \cref{meth:Dir-2s-stepdown}, and 
the combined pre-test/stepdown procedure in \cref{meth:Dir-2s-pretest}.  
For \cref{meth:Dir-2s}, we forgo the adjustment of $\alpha$ for one-sided testing in favor of using our $\tilde\alpha$ reference table for faster computation. 

\begin{table}[htbp]
\centering
\begin{threeparttable}
\caption{\label{tab:sim-2s-FWER}Simulated FWER, two-sample, one-sided, $\alpha=0.05$, $n_X=n_Y=200$.}
\begin{tabular}{lllcccc}
\toprule
\multicolumn{1}{l}{$H_{0r}\textrm{ true}$} & 
\multicolumn{1}{l}{$H_{0\tau}\textrm{ true}$} & 
\multicolumn{1}{l}{$F_X^{-1}(\tau)=F_Y^{-1}(\tau)$} & 
Basic & Joint & Stepdown & Pre+Step \\
\midrule
$r\in[-1,1]$   & $\tau\in[0,1]$   & $\tau\in[0,1]$   & 0.049 & 
0.044 & 0.044 & 0.044 \\
$r\in[-1,0]$   & $\tau\in[0,0.5]$ & $\tau\in[0,0.5]$ & 0.031 & 
0.031 & 0.044 & 0.044 \\
$r\in[-1,1]$   & $\tau\in[0,1]$   & $\tau\in[0.5,1]$ & 0.013 & 
0.026 & 0.026 & 0.032 \\
$r\in[-1,0]$   & $\tau\in[0,0.5]$ & $\tau\in\{0.5\}$ & 0.003 & 
0.000 & 0.002 & 0.006 \\
\bottomrule
\end{tabular}
\begin{tablenotes}
\item \textbf{Note:} $H_{0r} \colon F_X(r) \ge F_Y(r)$ or $H_{0\tau} \colon F_X^{-1}(\tau) \le F_Y^{-1}(\tau)$, $F_Y=\UnifDist(-1,1)$ so $F_Y^{-1}(\tau)=2(\tau-0.5)$, $1000$ replications.  
For $\tau$ where $F_X^{-1}(\tau)\ne F_Y^{-1}(\tau)$, $F_X^{-1}(\tau)=4(\tau-0.5)$; see \cref{fig:H1s}, where $F_Y=F_0$ and $F_X=F$.  
``Basic'' is \cref{meth:Dir-2s}, 
``Joint'' uses only iteration $i=0$ from \cref{meth:Dir-2s-stepdown}, 
``Stepdown'' is \cref{meth:Dir-2s-stepdown}, and 
``Pre+Step'' is \cref{meth:Dir-2s-pretest}. 
The Basic test is evaluated at $r=-0.99,-0.98,\ldots,0.99$.
\end{tablenotes}
\end{threeparttable}
\end{table}

\Cref{tab:sim-2s-FWER} shows strong control of FWER for all four methods. 
The stepdown and pre-test procedures' strong control of FWER in \cref{tab:sim-2s-FWER} supports our heuristic arguments. 
Overall, the patterns are similar to those in \cref{tab:sim-1s-FWER}.

\subsection{Power comparison}
\label{sec:sim-power}

We now illustrate the power improvement from the stepdown and pre-test procedures. 
For pointwise and global power comparisons with the KS-based MTP, see \cref{sec:sim-pt-pwr}. 

For one-sample multiple testing, \cref{fig:sim-1s-stepdown} compares the pointwise (by $\tau$) RPs of the same methods shown in \cref{tab:sim-1s-FWER}.  
The DGPs are the same as the rows in \cref{tab:sim-1s-FWER} where 
$\{\tau : H_{0\tau}\textrm{ is true}\}=[0,0.5]$, 
visualized in the second and fourth panels in \cref{fig:H1s}. 
Since all methods (correctly) have RP near zero for $\tau<0.5$, only larger $\tau$ are shown. 
Compared with the basic Dirichlet MTP, the stepdown procedure weakly increases pointwise power, and adding the pre-test does, too. 
The pre-test is only helpful in the right panel where the null hypothesis constraint is slack for $\tau<0.5$. 

\begin{figure}[htbp]
\centering
\hfill
\includegraphics[width=0.4\textwidth,clip=true,trim=10 15 10 80]{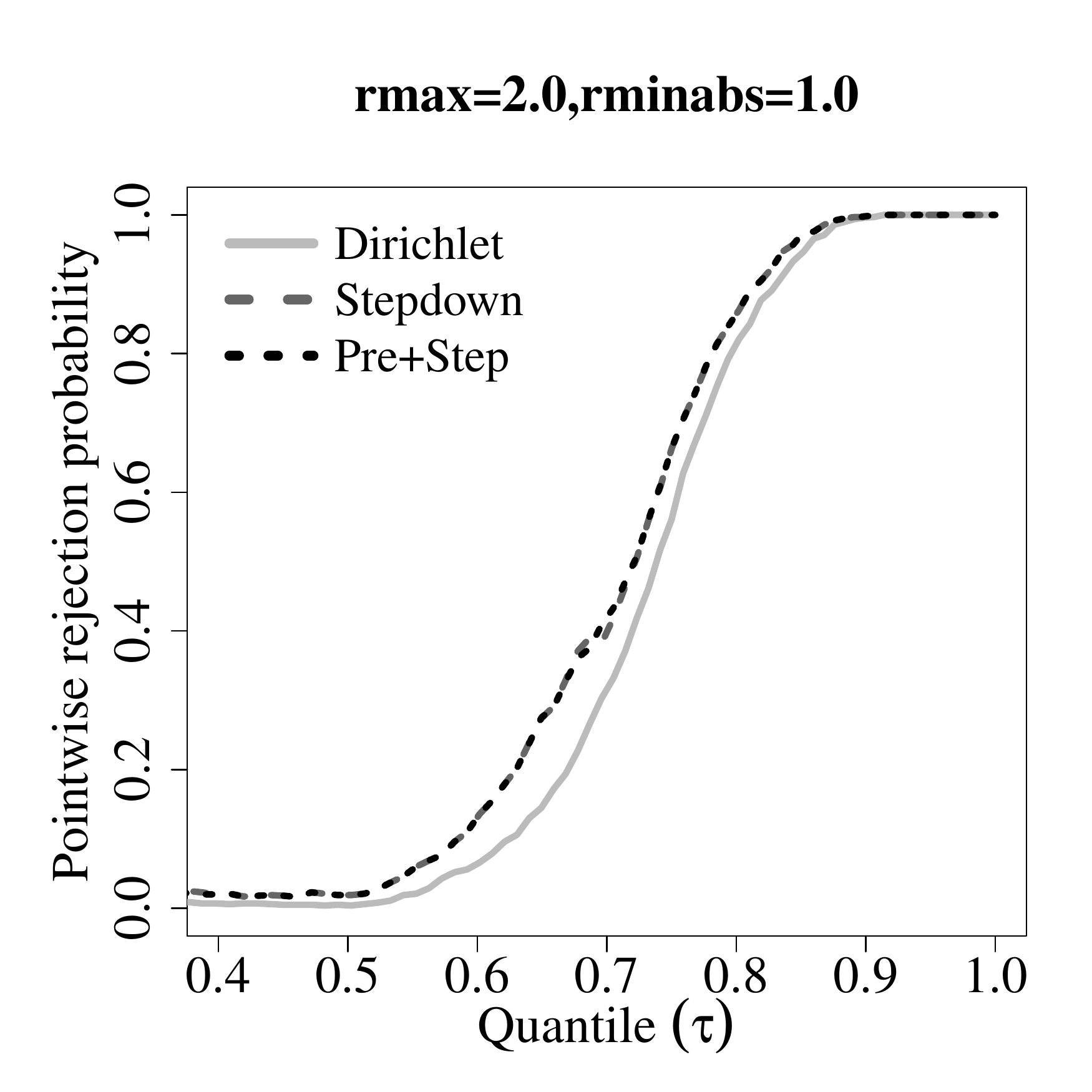}%
\hfill%
\includegraphics[width=0.4\textwidth,clip=true,trim=10 15 10 80]{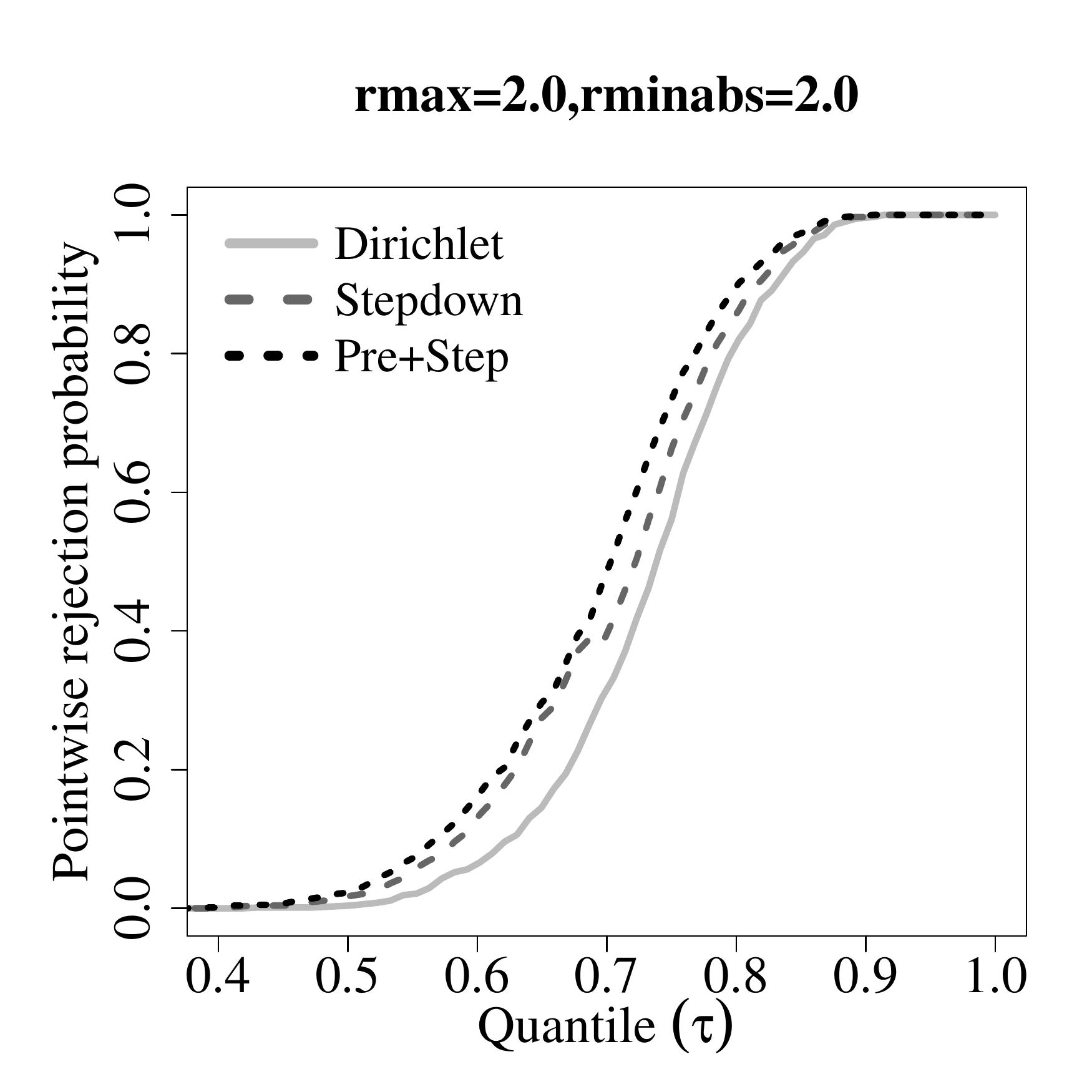}%
\hfill\null
\caption{\label{fig:sim-1s-stepdown}Simulated pointwise RP by quantile ($\tau$), same DGP and methods as \cref{tab:sim-1s-FWER}. 
Left: $F^{-1}(\tau)=F_0^{-1}(\tau)$ for $\tau\le0.5$, $F^{-1}(\tau)=4(\tau-0.5)$ otherwise.  
Right: $F^{-1}(\tau)=4(\tau-0.5)$.  
}
\end{figure}

For two-sample multiple testing, \cref{fig:sim-2s-stepdown} compares the pointwise (by $\tau$) RPs of the methods shown in \cref{tab:sim-2s-FWER}. 
For the basic MTP that tests $H_{0r}$ (with $r\in\R$) instead of $H_{0\tau}$, we plot the RP of $H_{0r}$ at $\tau=F_Y(r)=(r+1)/2$. 
The DGPs are the same as the rows in \cref{tab:sim-2s-FWER} where 
$\{\tau : H_{0\tau}\textrm{ is true}\}=[0,0.5]$.  
The power improvements due to the stepdown and pre-test procedures are similar to \cref{fig:sim-1s-stepdown}: modest but noticeable over a range of $\tau>0.5$. 
A bigger power difference is between the basic MTP and the joint quantile difference MTP (iteration $i=0$ of \cref{meth:Dir-2s-stepdown}). 
These MTPs' powers differ because the latter MTP explicitly focuses on fewer quantiles, in this case only $10$, so more power can be focused on each quantile. 
This may be a reasonable way to improve power, especially in small samples, or if one assumes the quantile differences do not vary too quickly with $\tau$. 
One could further increase pointwise power by examining yet fewer quantiles, but the choice of quantiles becomes arbitrary and subject to manipulation. 

\begin{figure}[htbp]
\centering
\hfill
\includegraphics[width=0.45\textwidth,clip=true,trim=10 15 10 80]{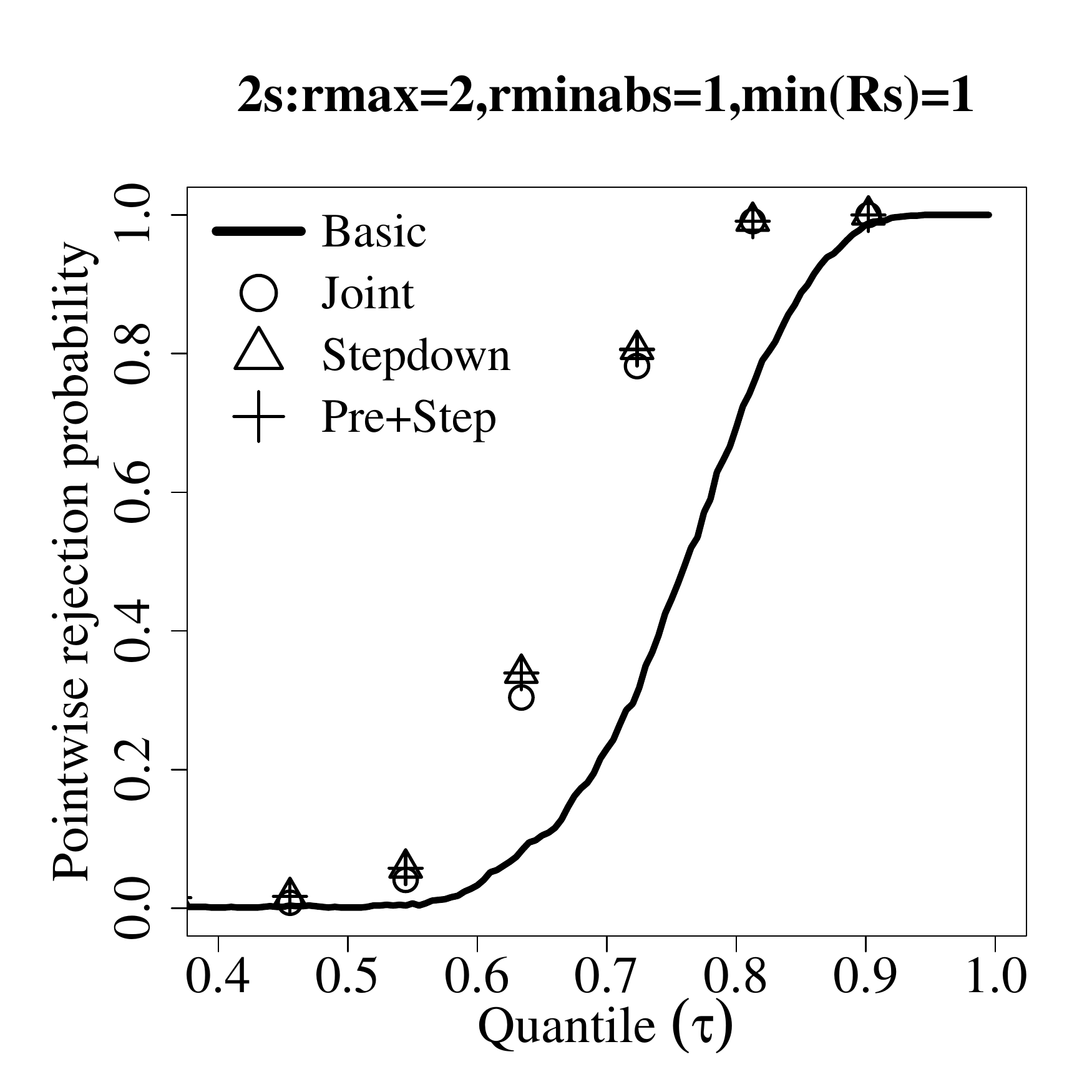}%
\hfill%
\includegraphics[width=0.45\textwidth,clip=true,trim=10 15 10 80]{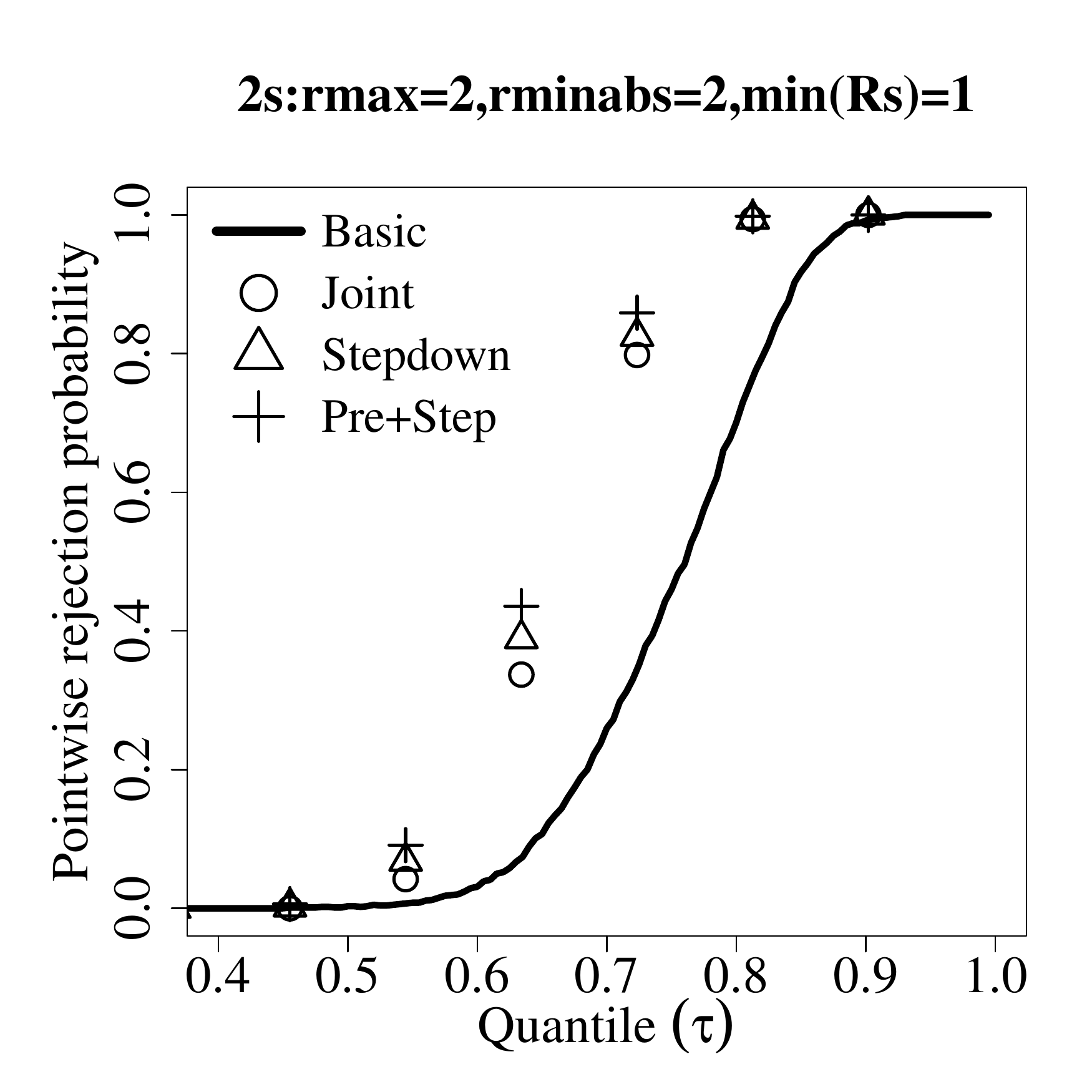}%
\hfill\null
\caption{\label{fig:sim-2s-stepdown}Simulated pointwise RP by quantile, same DGP and methods as \cref{tab:sim-2s-FWER}. For the Basic method, the RP is plotted for $\tau=F_Y(r)=(r+1)/2$. 
Left: $F_X^{-1}(\tau)=F_Y^{-1}(\tau)$ for $\tau\le0.5$, $F_X^{-1}(\tau)=4(\tau-0.5)$ otherwise.  
Right: $F_X^{-1}(\tau)=4(\tau-0.5)$.  
}
\end{figure}

\subsection{Computation time}\label{sec:sim-1s-comp}

\Cref{tab:comp-time} shows computation times for one-sample, two-sided methods: the Dirichlet MTP, the asymptotic KS test, and the exact KS test. 
Each value in the table has been averaged over at least four repetitions, using a standard desktop computer (8GB RAM, 3.2GHz processor). 
The time to simulate $\tilde\alpha$ (to the same degree of precision as \cref{fact:alpha-tilde-rate}) is also shown; this is the time saved by \cref{fact:alpha-tilde-rate} compared with just-in-time simulation as in \citet{BujaRolke2006}. 
The simulation time depends on the starting value of $\tilde\alpha$ in the numerical search; we use five search iterations to be comparable to \citet[p.\ 254]{AldorNoimanEtAl2013}, who report a runtime of $10$ seconds for $n=100$ (compared to $9.47$ seconds in our table).

\begin{table}[htbp]
\centering
\caption{\label{tab:comp-time}Computation time (seconds); one-sample, two-sided, $\alpha=0.1$.}
\begin{threeparttable}
\begin{tabular}{ccccc}
\toprule
$\log_{10}(n)$ & \cref{fact:alpha-tilde-rate} & \citet{BujaRolke2006} 
& KS & KS (exact) \\
\midrule
2 & 0.00 & \phantom{18}9.47 & 0.00 & \phantom{2}0.00 \\
3 & 0.02 & \phantom{5}14.84 & 0.00 & \phantom{2}0.00 \\
4 & 0.23 & \phantom{1}82.48 & 0.00 & \phantom{2}0.08 \\
5 & 2.20 & \phantom{}851.14 & 0.01 & 25.25 \\
\bottomrule
\end{tabular}
%
\end{threeparttable}
\end{table}

In \cref{tab:comp-time}, the asymptotic KS test runs instantly even for $n=\num{100000}$.  
The exact KS slows significantly around $n=\num{100000}$, requiring over $20$ seconds per test.  
With \cref{fact:alpha-tilde-rate}, the Dirichlet MTP only takes a few seconds even with $n=\num{100000}$, faster than the exact KS and orders of magnitude faster than just-in-time simulation.

\subsection{Empirical-based DGP}
\label{sec:sim-emp}

A DGP based on the ``gift wage'' empirical example in \cref{sec:emp-gift} was constructed as follows. 
For both the library and fundraising tasks, using the Period 1 data, for both treatment and control groups, first a piecewise linear quantile function was interpolated between points $(\tau,F^{-1}(\tau))$ consisting of $(k/(n+1),Y_{n:k})$, $(0,0)$, and $(1,Y_{n:n}+10)$. 
Second, this was modified to only include integer values (as in the data) by applying the floor function $\lfloor\cdot\rfloor$ to generate a step function $F^{-1}(\cdot)$. 
These are the true population quantile functions for our simulations. 
The sample sizes are the same as in our empirical example (library: $10$ control, $9$ treatment; fundraising: $10$ control, $13$ treatment), as is the nominal one-sided FWER level $\alpha=0.1$. 
There were $\num{10000}$ simulation replications. 

We compare five methods: 
``Basic'' is our Dirichlet-based MTP in \cref{meth:Dir-2s}, 
``KS'' is the KS-based MTP as in \cref{prop:KS-2s-FWER}, 
``Joint'' is a joint quantile difference test (iteration $i=0$ from \cref{meth:Dir-2s-stepdown}), 
``Stepdown'' is \cref{meth:Dir-2s-stepdown}, and 
``Pre+Step'' is \cref{meth:Dir-2s-pretest}. 
Basic and KS test $H_{0r} \colon F_T(r) \ge F_C(r)$ over each integer $r$ between $0$ and the maximum possible value, where subscript $T$ stands for ``treatment'' and $C$ for ``control.'' 
In the library task, $H_{0r}$ is true for $0 \le r \le 27$; 
in the fundraising task, $H_{0r}$ is true for $42 \le r \le 44$. 
The quantile tests evaluate $H_{0\tau}$ for $\tau\in\{0.30, 0.50, 0.70\}$ for the library task and $\tau\in\{ 0.22, 0.41, 0.59, 0.78 \}$ for fundraising; all $H_{0\tau}$ are false. 

The FWER is nearly zero for both the Basic and KS MTPs: $0.002$ for the library DGP and $0.001$ for fundraising (for both methods). 
The FWER is close to zero because $H_{0r}$ is true for relatively small ranges of $r$. 
This is similar to the FWER in the last row of \cref{tab:sim-2s-FWER}, where it is shown how FWER only gets close to the nominal level when nearly all $H_{0r}$ are true.  

The Basic MTP has the best global power against $H_0 \colon F_T(\cdot) \ge F_C(\cdot)$. 
That is, it has the highest probability of rejecting at least one $H_{0r}$. 
Next best are the joint quantile tests (which are all the same because the pre-test does not help for these DGPs and the stepdown cannot increase \emph{global} power). 
The KS has the worst global power. 
Although not surprising that the Basic MTP has better global power than the KS, it is surprising that it fares better than the joint quantile tests that focus power on a smaller number of points where all $H_{0\tau}$ are false. 
The simple explanation may be that these few $\tau$ do not match up with the most statistically obviously false $H_{0\tau}$. 
So at least here, the ``evenly sensitive'' approach of the Basic MTP actually leads to the best global power, too. 


\begin{table}[htbp]
\centering
\caption{\label{tab:sim-emp}FWER and global power for empirical simulation.}
\begin{threeparttable}
\begin{tabular}{lcclcc}
\toprule
& \multicolumn{2}{c}{FWER} && \multicolumn{2}{c}{Global power} \\
\cmidrule{2-3}\cmidrule{5-6}
Method & Library & Fundraising && Library & Fundraising \\
\midrule
Basic    & 0.002 & 0.001 && 0.647 & 0.815 \\
KS       & 0.002 & 0.001 && 0.477 & 0.714 \\
Joint    & 0.000 & 0.000 && 0.583 & 0.758 \\
Stepdown & 0.000 & 0.000 && 0.583 & 0.758 \\
Pre+Step & 0.000 & 0.000 && 0.583 & 0.759 \\
\bottomrule
\end{tabular}
%
\end{threeparttable}
\end{table}

\Cref{fig:sim-emp} shows pointwise rejection probability (RP), similar to \cref{fig:sim-2s-stepdown}. 
The joint quantile tests generally have higher pointwise RP, although the magnitude partly depends on whether $r$ is compared with $F_T^{-1}(\tau)$ (as in \cref{fig:sim-emp}) or $F_C^{-1}(\tau)$. 
The pre-test is useless because all $H_{0\tau}$ are false. 
The stepdown procedure improves pointwise RP by a few percentage points, sometimes more, sometimes less. 
Most notably, even though we used a nominal FWER level slightly above $10\%$ for the KS MTP and slightly below $10\%$ for the Dirichlet MTP, the Dirichlet MTP has significantly higher pointwise RP (i.e., power) than the KS MTP across a range of $r$. 

\begin{figure}[htbp]
\centering
\includegraphics[width=0.47\textwidth,clip=true,trim=20 15 30 80]{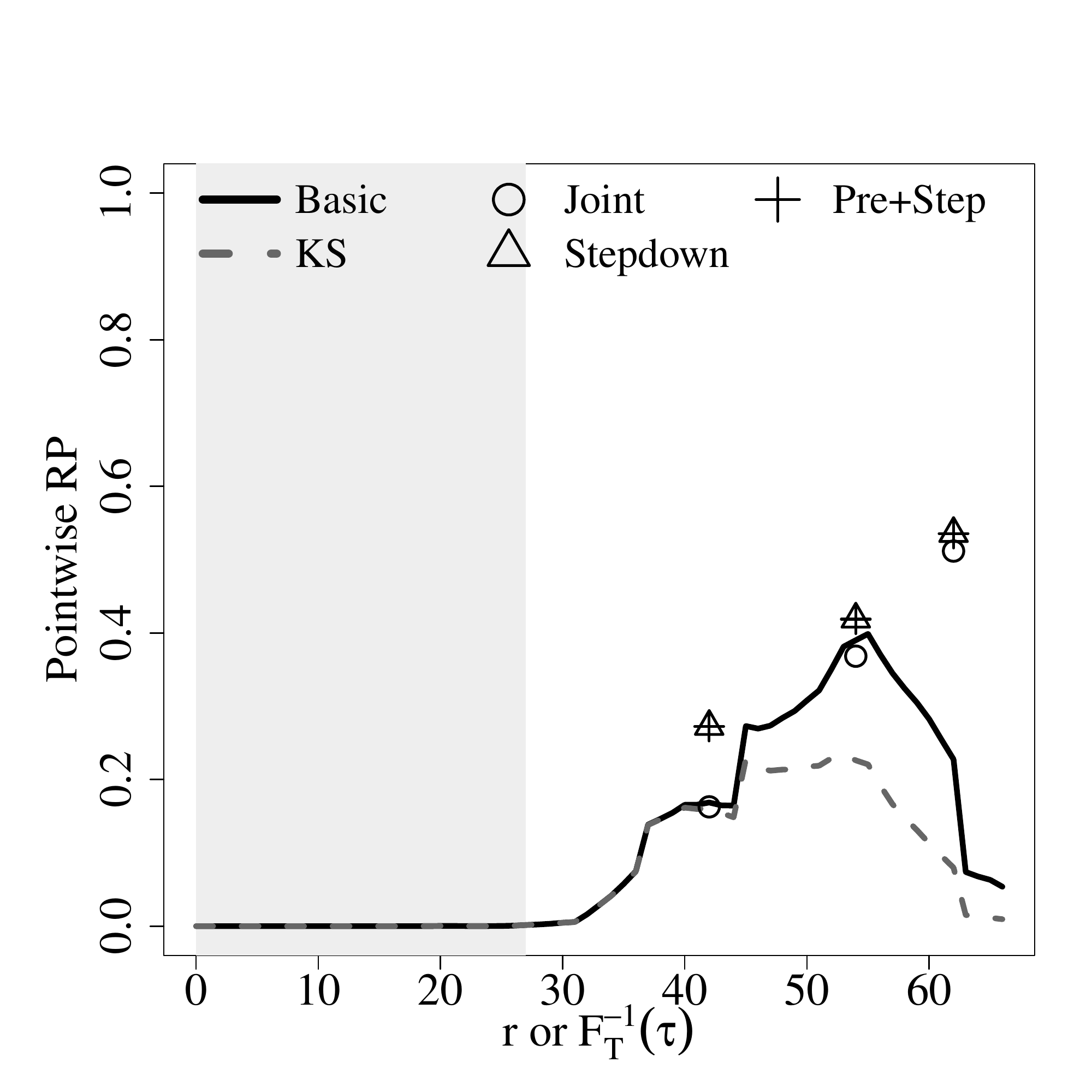}%
\hfill%
\includegraphics[width=0.47\textwidth,clip=true,trim=20 15 30 80]{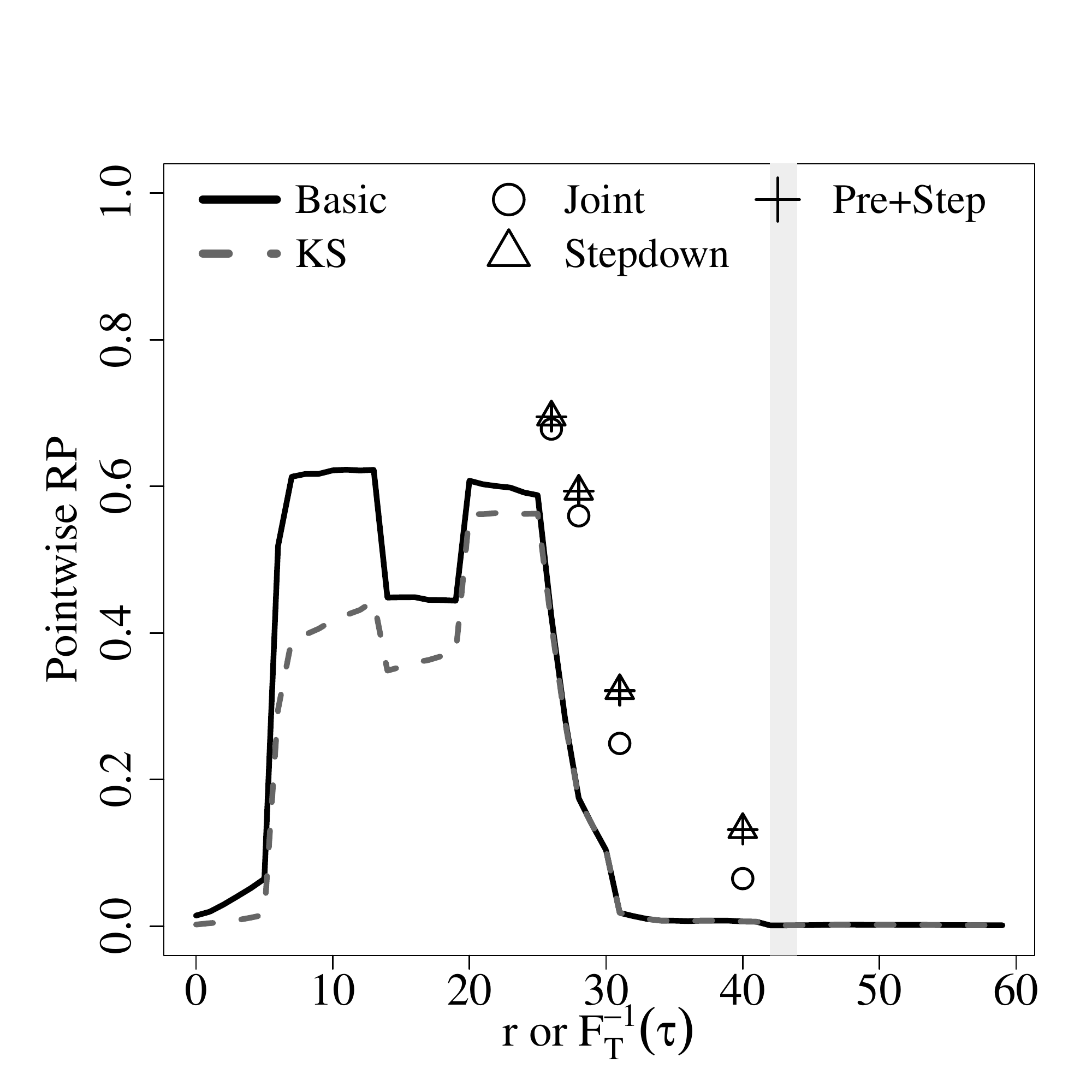}%
\caption{\label{fig:sim-emp}Simulated pointwise RP, empirical simulations, $H_{0r} \colon F_T(r) \ge F_C(r)$. 
For the Basic and KS MTPs, the horizontal axis shows the value of $r$; otherwise, it shows $F_T^{-1}(\tau)$. 
Gray shading indicates true $H_{0r}$. 
Left: library data entry task (units: books). 
Right: door-to-door fundraising task (units: dollars). 
}
\end{figure}

\section{Conclusion}
\label{sec:conclusion}

We have considered the question, ``At which quantiles or CDF values do two distributions differ?'' 
Framed as multiple testing across the continuum of quantiles $\tau\in(0,1)$ or values $r\in\R$, we have shown KS-based multiple testing procedures to have strong control of FWER, for both one-sided and two-sided, one-sample and two-sample inference. 
Our newly proposed Dirichlet-based procedures also have strong control of finite-sample FWER, along with other advantages: more even sensitivity than KS, improved global power, stepdown and pre-test power improvements, and fast computation. 

Extensions to conditional distributions and regression discontinuity have also been provided here. 
Future work may include extensions to other models, stepdown and pre-test procedures to enhance \cref{meth:Dir-2s}, derivation of FWER bounds for the two-sample quantile MTP, and exploration of the connection with the continuity-corrected Bayesian bootstrap of \citet{Banks1988}.

\Supplemental{}{
\bibliographystyle{elsarticle-harv}
\singlespacing

\doublespacing}

\Supplemental{\pagebreak\setcounter{page}{2}}{}

\appendix

\singlespacing

\section{Additional methods}
\label{sec:app-meth}

\subsection{One-sample methods}
\label{sec:app-meth-1s}

\begin{method}\label{meth:Dir-1s-2s-stepdown}
For \cref{task:1s-test-FWER-2s}, modify \cref{meth:Dir-1s-1s-stepdown} as follows.  
Define $\ell_k$ and $u_k$ as in \cref{eqn:def-ell-k-u-k}.  
Instead of only $r_{k,i}$ corresponding to either $\ell_k$ or $u_k$, include both $r_{k,\ell,i}$ corresponding to $\ell_k$ and $r_{k,u,i}$ corresponding to $u_k$.  
Instead of $\hat{K}_0=\{1,\ldots,n\}$, let $\hat{K}_0^\ell=\hat{K}_0^u=\{1,\ldots,n\}$, where $\hat{K}_0^\ell$ corresponds to the $\ell_k$ and $\hat{K}_0^u$ to the $u_k$.  
Replace \cref{eqn:meth-1s-stepdown-calib} with
\begin{align}\notag
\alpha 
  &\ge 1 - \Pr\Bigl( \bigcap_{k\in\hat{K}_i^\ell} \left\{ X_{n:r_{k,\ell,i}}  \ge F^{-1}(\ell_k)\right\} 
                \cap \bigcap_{k\in\hat{K}_i^u}    \left\{ X_{n:r_{k,u,i}}     \le F^{-1}(u_k)   \right\} 
              \Bigr) \\
  &\ge 1 - \Pr\Bigl( \bigcap_{k\in\hat{K}_i^\ell} \left\{F(X_{n:r_{k,\ell,i}})\ge\ell_k\right\} 
                \cap \bigcap_{k\in\hat{K}_i^u}    \left\{F(X_{n:r_{k,u,i}})   \le u_k  \right\} \Bigr) .
\label{eqn:meth-1s-stepdown-calib-2s}
\end{align}
Check for rejections of both $F^{-1}(\tau)\ge F_0^{-1}(\tau)$ and $F^{-1}(\tau)\le F_0^{-1}(\tau)$ as described in \cref{meth:Dir-1s-1s-stepdown}; either implies rejection of $H_{0\tau} \colon F^{-1}(\tau) = F_0^{-1}(\tau)$.  
\end{method}

\begin{method}[Pre-test only]\label{meth:1s-pretest}
Consider the pre-test null hypotheses $H_{0\ell_k} \colon F^{-1}(\ell_k) \le F_0^{-1}(\ell_k)$, defining $\ell_k$ as in \cref{eqn:def-ell-k-u-k}. 
Given $\tilde\alpha$, let $\underline{k}=\min\{k:\ell_k\ge B^{1-\tilde\alpha}_{1,n}\}$ and $r_k=\max\{k':B^{1-\tilde\alpha}_{k',n}\le\ell_k\}$ (for $k\ge\underline{k}$), where both $k$ and $k'$ are restricted to integers $\{1,\ldots,n\}$.  
Using \cref{thm:Wilks}, calculate
\[ \alpha_p(\tilde\alpha,n) 
  =1 - \Pr\Bigl( \bigcap_{k=\underline{k}}^n X_{n:r_k}\le F^{-1}(\ell_k) \Bigr) 
  =1 - \Pr\Bigl( \bigcap_{k=\underline{k}}^n F(X_{n:r_k})\le \ell_k \Bigr) 
  . \]
Adjust $\tilde\alpha$ until $\alpha_p(\tilde\alpha,n)$ equals (approximately) the desired FWER.  Reject $H_{0\tau} \colon F^{-1}(\tau) \le F_0^{-1}(\tau)$ when $\max\{X_{n:r_k}:\ell_k\le\tau\}>F_0^{-1}(\tau)$. 

To instead pre-test $H_{0u_k} \colon F^{-1}(u_k) \ge F_0^{-1}(u_k)$, reverse all inequalities and $\min$/$\max$, and replace $\ell_k$ with $u_k$ (also from \cref{eqn:def-ell-k-u-k}), $B^{1-\tilde\alpha}_{k,n}$ with $B^{\tilde\alpha}_{k,n}$, $\underline{k}=\min\{k:\ell_k\ge B^{1-\tilde\alpha}_{1,n}\}$ with $\bar{k}=\max\{k:u_k\le B^{\tilde\alpha}_{n,n}\}$, and $\bigcap_{k=\underline{k}}^n$ with $\bigcap_{k=1}^{\bar{k}}$. 
\end{method}

\subsection{Two-sample quantile MTP and procedures to improve power}
\label{sec:2s-power}

We propose a two-sample quantile MTP along with stepdown and pre-test procedures. 
Unlike the other methods in this paper, these are not based on finite-sample distributions of order statistics. 
Instead, we (slightly) extend results from \citet{GoldmanKaplan2017b}. 
This requires that the quantiles not be too close together. 
To be more explicit about how the methods work, we present modified tasks that they address.

\begin{enumeratecomp}[\bfseries T{a}sk 1]
\setcounter{enumi}{\theallenumi}
\item\label{task:2s-test-FWER-2s} Testing a family of $M_n=\lfloor n^{2/5}\rfloor$ two-sample quantile equality hypotheses with strong control of FWER; specifically, for $j=1,\ldots,M_n$, 
$H_{0j} \colon F_X^{-1}(t) = F_Y^{-1}(t)$ 
for all $t\in[(j-0.5)/(M_n+1),(j+0.5)/(M_n+1)]$. 
\item\label{task:2s-test-FWER-1s} Same as \cref{task:2s-test-FWER-2s} but with $F_X^{-1}(t)\le F_Y^{-1}(t)$ or $F_X^{-1}(t)\ge F_Y^{-1}(t)$. 
\end{enumeratecomp}

Consider a fixed set of $M$ quantiles, $\tau_1,\ldots,\tau_M$, and let $\Delta_j\equiv F_Y^{-1}(\tau_j)-F_X^{-1}(\tau_j)$.  
\citet{GoldmanKaplan2017b} use ``fractional order statistics'' to construct a CI for each $\Delta_j$ with $1-\alpha+O\bigl(n^{-2/3}\log(n)\bigr)$ coverage probability, and CIs for all $F_X^{-1}(\tau_j)$ or $F_Y^{-1}(\tau_j)$ that have joint (over $j=1,\ldots,M$) coverage probability of $1-\alpha+O(n^{-1})$. 
It is a small step to infer that CIs for all $\Delta_j$ can be constructed with joint $1-\alpha+O\bigl(n^{-2/3}\log(n)\bigr)$, using the modified calibration (of $\tilde\alpha$) seen in our code. 
For a lower one-sided CI, the upper endpoints are $\hat{Q}^L_Y\bigl(u_{y,j}^h(\tilde\alpha)\bigr)-\hat{Q}^L_X\bigl(u_{x,j}^l(\tilde\alpha)\bigr)$, where $u^h_{y,j}(\tilde\alpha)\approx \tau_j + n_Y^{-1/2}z_{1-\tilde\alpha}\sqrt{\tau_j(1-\tau_j)}$, $u^l_{x,j}(\tilde\alpha)\approx \tau_j - n_X^{-1/2}z_{1-\tilde\alpha}\sqrt{\tau_j(1-\tau_j)}$, 
$z_{1-\tilde\alpha}$ is the standard normal distribution's $(1-\tilde\alpha)$-quantile, 
$\tilde\alpha$ solves 
\[ 1-\alpha = \Pr\left( \bigcap_{j=1}^{M} \left\{ \tilde{Q}^I_{U_y}\bigl(u^h_{y,j}(\tilde\alpha)\bigr) - \tilde{Q}^I_{U_x}\bigl(u^l_{x,j}(\tilde\alpha)\bigr) > 0 \right\} \right) , \]
$\tilde{Q}^I_{U_x}$ is a Dirichlet process with index measure $\nu(\cdot)$ where $\nu([0,t])=(n_X+1)t$ for $t\in[0,1]$ \citep{Stigler1977}, 
and $\hat{Q}^L_X(u)\equiv X_{n_X:k} + [u(n_X+1)-k] X_{n_X:k+1}$, $k=\lfloor u(n_X+1)\rfloor$, and similarly for $\hat{Q}^L_Y(u)$. 
The upper one-sided CI is defined similarly, and the two-sided CI is the intersection of upper and lower one-sided CIs. 

Let $\widehat{\textrm{CI}}_j$ denote the CI for $\Delta_j$.  
Letting $I=\{j : H_{0j}\textrm{ is true}\}$, 
\begin{equation*}
\FWER
  = 1 - \Pr\left( \bigcap_{j\in I} \{\Delta_j\in\widehat{\textrm{CI}}_j\} \right) 
  \le 1 - \Pr\left( \bigcap_{j=1}^{M} \{\Delta_j\in\widehat{\textrm{CI}}_j\} \right) 
  \to 1-(1-\alpha) 
  = \alpha . 
\end{equation*}

If $M_n\to\infty$ too quickly, then the arguments from \citet{GoldmanKaplan2017b} break down, but we conjecture they still hold with $M_n=O(n^{2/5})$. 

\begin{method}\label{meth:Dir-2s-stepdown}
For \cref{task:2s-test-FWER-2s}, let $\tau_j=j/(M_n+1)$ for $j=1,\ldots,M_n$. 
Let $\hat T_0\equiv\{1,\ldots,M_n\}$. 
Given a pointwise $\tilde\alpha$, let $k_{X,j}^u$ and $k_{X,j}^\ell$ be such that 
\[ \Pr\left(\BetaDist(k_{X,j}^u,n_X+1-k_{X,j}^u)<\tau_j\right)=\tilde\alpha/2=\Pr\left(\BetaDist(k_{X,j}^\ell,n_X+1-k_{X,j}^\ell)>\tau_j\right) , \] 
and similarly for $k_{Y,j}^u$ and $k_{Y,j}^\ell$ (with $n_Y$ instead of $n_X$).  
These $k$ may have fractional (non-integer) values. 
For iteration $i$, CIs with joint $1-\alpha$ coverage probability are constructed with $\tilde\alpha$ chosen such that
\begin{equation}\label{eqn:2s-stepdown-calib}
1-\alpha 
= \Pr\left( \bigcap_{j\in\hat T_i} \left\{D_{X,j}^\ell 
      < D_{Y,j}^u, D_{Y,j}^\ell < D_{X,j}^u\right\} \right) ,
\end{equation}
defining $F_X(X_{n_X:0})\equiv0$, $F_X(X_{n_X:n_X+1})\equiv1$, $X_{n_X:k}\equiv (1-k+\lfloor k\rfloor)X_{n_X:\lfloor k\rfloor}+(k-\lfloor k\rfloor) X_{n_X:\lfloor k\rfloor+1}$ for fractional $k$, and using the distribution
\begin{equation*}
\begin{split}
& \left(D_{X,1}, D_{X,2}-D_{X,1}, \ldots, D_{X,2M_n}-D_{X,2M_n-1}, 1-D_{X,2M_n}\right) \\
&\quad  \sim \textrm{Dir}\left(k_1,k_2-k_1,\ldots,k_{2M_n}-k_{2M_n-1},n_X+1-k_{2M_n}\right)
\end{split}
\end{equation*}
with vector $k=(k_1,\ldots,k_{2M_n})$ containing all the $k_{X,j}^\ell$ and $k_{X,j}^u$ in ascending order so that $k_1\le\cdots\le k_{2M_n}$; and defining all these objects similarly for $Y$, with $\vecf{D}_X\independent\vecf{D}_Y$. 
For iteration $i=0$, reject any $H_{0j}$ for which the CI $[Y_{n_Y:k^\ell_{Y,j}}-X_{n_X:k^u_{X,j}} , Y_{n_Y:k^u_{Y,j}}-X_{n_X:k^\ell_{X,j}}]$ does not contain zero.  
Then, iteratively perform the following steps, starting with $i=1$. 
\begin{enumeratecomp}[Step 1.]
 \item\label{step:2s-stepdown-Ti} Let $\hat T_i=\{j : H_{0j}\textrm{ not yet rejected}\}$.  If $\hat T_i=\emptyset$ or $\hat T_i=\hat T_{i-1}$, then stop. 
 \item Use $\hat T_i$ and \cref{eqn:2s-stepdown-calib} to construct new joint CIs. 
 \item Reject any additional $H_{0j}$ for which the corresponding CI does not contain zero.  
 \item Increment $i$ by one and return to Step \ref{step:2s-stepdown-Ti}. 
\end{enumeratecomp}
For \cref{task:2s-test-FWER-1s}, use the above with only upper (or lower) endpoints. 
\end{method}

\begin{method}\label{meth:Dir-2s-pretest}
For \cref{task:2s-test-FWER-1s}, using notation from \cref{meth:Dir-2s-stepdown}, consider $H_{0j} \colon F_X^{-1}(\tau_j) \ge F_Y^{-1}(\tau_j)$.  
First run a pre-test of $H_{0j}' \colon F_X^{-1}(\tau_j) \le F_Y^{-1}(\tau_j)$ using iteration $i=0$ of \cref{meth:Dir-2s-stepdown} (i.e., the basic method without stepdown) with FWER level $\alpha_p=\alpha/\ln[\ln(\max\{n,15\})]$. 
Then, use \cref{meth:Dir-2s-stepdown} starting with $\hat T_0$ containing all $j$ such that $H_{0j}$ was not rejected by the pre-test. 
\end{method}

We conjecture that under \cref{a:iid,a:F}, \cref{meth:Dir-2s-stepdown,meth:Dir-2s-pretest} have strong control of asymptotic FWER.

\section{Mathematical proofs}
\label{sec:app-pfs}


\subsection{Proof of Proposition \ref{prop:KS-1s-FWER}}
\begin{proof}
The two-sided proof is in the main text. 

The one-sided case follows the same argument (after modifying $D_n^x$ and $D_n^{x,0}$), with the additional inequality that if $H_{0x} \colon F(x) \le F_0(x)$ is true, then $\hat{F}(x)-F_0(x) \le \hat{F}(x)-F(x)$. 
Let $D_n^x \equiv \sqrt{n} \bigl( \hat{F}(x) - F(x) \bigr)$, $D_n \equiv \sup_{x\in\R} D_n^x $, and $c_n(\alpha)$ now satisfies $\Pr\bigl(D_n>c_n(\alpha)\bigr)=\alpha$ in finite samples. 
Let
\begin{equation*}
D_n^{x,0} \equiv \sqrt{n} \bigl( \hat{F}(x) - F_0(x) \bigr), \quad
I \equiv \{x : H_{0x}\textrm{ is true}\} , \quad
D_n^I \equiv \sup_{x\in I} D_n^{x,0} \le \sup_{x\in I} D_n^x , 
\end{equation*}
where the last inequality follows because $F(x)\le F_0(x)$ for $x\in I$, so $D_n^{x,0}\le D_n^x$ (whereas before this was an equality). 
Then, since $I\subseteq\R$, 
\begin{equation*}
\FWER
  \equiv \Pr\bigl( D_n^I > c_n(\alpha) \bigr)
   \le \Pr\bigl( D_n > c_n(\alpha) \bigr) 
   = \alpha .
\end{equation*}
The one-sided argument with $H_{0x} \colon F(x) \ge F_0(x)$ is identical when using $-D_n^x$ instead. 

Alternatively, the results can be derived using the fact that the KS test can be inverted to give a uniform confidence band, and any MTP based on a uniform confidence band has strong control of FWER. 
\end{proof}

\subsection{Proof of Lemma \ref{lem:weak-to-strong}}
\begin{proof}
For the EDFs defined in \cref{eqn:def-EDF-2s}, pointwise, 
$n_X\hat{F}_X(r)\sim\textrm{Binomial}\bigl(n_X,F_X(r)\bigr)$, 
$n_Y\hat{F}_Y(r)\sim\textrm{Binomial}\bigl(n_Y,F_Y(r)\bigr)$, 
and by \cref{a:iid} 
$\hat{F}_X(\cdot)\independent\hat{F}_Y(\cdot)$. 
Since (by assumption) rejection of $H_{0r}$ depends only on $\hat{F}_X(r)$ and $\hat{F}_Y(r)$, the RP depends only on $F_X(r)$ and $F_Y(r)$.  
More generally, the distribution of 
\[ \bigl(n_X\hat{F}_X(r_1),n_X[\hat{F}_X(r_2)-\hat{F}_X(r_1)],\ldots,n_X[\hat{F}_X(r_m)-\hat{F}_X(r_{m-1})]\bigr) \]
is multinomial with parameters $n_X$ and $\bigl(F_X(r_1),F_X(r_2)-F_X(r_1),\ldots,F_X(r_m)-F_X(r_{m-1})\bigr)$, 
and similarly for $Y$. 
Even if set $S$ is a continuum, the distribution of $\{\hat{F}_X(r),\hat{F}_Y(r)\}_{r\in S}$ depends only on $n_X$, $n_Y$, and $\{F_X(r),F_Y(r)\}_{r\in S}$. 
Consequently, RPs of $H_{0r}$ over $r\in S$ depend only on $n_X$, $n_Y$, and $\{F_X(r),F_Y(r)\}_{r\in S}$, too. 

As in \cref{def:FWER}, let $I\equiv\{r : H_{0r}\textrm{ is true}\}$, so $I\subseteq\R$. 
Define $G_X(\cdot)$ such that $G_X(r)=F_X(r)$ if $H_{0r}$ is true and $G_X(r)=F_Y(r)$ if $H_{0r}$ is false. 
Thus, if we had $G_X(\cdot)$ instead of $F_X(\cdot)$, $H_{0r}$ would be true for all $r\in\R$. 
Then, 
\begin{align*}
\FWER
  &\equiv \overbrace{\Pr\left( \textrm{reject $H_{0r}$ for any $r\in I$} \mid F_X, F_Y \right)}^{\textrm{by \cref{def:FWER}}}
\\&=      \overbrace{\Pr\left( \textrm{reject $H_{0r}$ for any $r\in I$} \mid G_X, F_Y \right)}^{\textrm{by above properties and $F_X(r)=G_X(r)$ for $r\in I$}}
\\&\le \overbrace{\Pr\left( \textrm{reject $H_{0r}$ for any $r\in\R$} \mid G_X, F_Y \right)}^{\textrm{by $I\subseteq\R$}} 
\\&\le \alpha 
\end{align*}
by assumption of weak control of FWER at level $\alpha$, since all $H_{0r}$ are true given $G_X(\cdot)$ and $F_Y(\cdot)$. 
\end{proof}

\subsection{Proof of Proposition \ref{prop:KS-2s-FWER}}
\begin{proof}
The method rejects $H_{0r}$ depending only on $\hat{F}_X(r)$ and $\hat{F}(r)$, through their difference $\hat{F}_X(r)-\hat{F}_Y(r)$. 
It is well known that the two-sample KS GOF test controls size, which is equivalent to weak control of FWER. 
Thus, the assumptions of \cref{lem:weak-to-strong} are satisfied, so the method has strong control of FWER. 
\end{proof}

\subsection{Proof of Theorem \ref{thm:Dir-1s-FWER}}
\begin{proof}
The one-sided proof is entirely in the main text. 

For the two-sided case, we have a parallel argument. 
Let 
\[ K^\ell\equiv\{k:F^{-1}(\ell_k)=F_0^{-1}(\ell_k)\} , \quad
   K^u   \equiv\{k:F^{-1}(u_k)   =F_0^{-1}(u_k)   \} , \]
the sets of true hypotheses. 
Then, 
\begin{align*}
\FWER
  &=   \overbrace{1 - \Pr(\textrm{no rejections among }k\in \{K^\ell\cup K^u\})}^{\textrm{by definition of FWER}}
\\&=   \overbrace{1 - \Pr\Bigl( \bigcap_{k\in K^\ell}  F_0^{-1}(\ell_k) \le X_{n:k}
                           \cap\bigcap_{k\in K^u} X_{n:k} \le F_0^{-1}(u_k)
                        \Bigr) }^{\textrm{by definition of $H_{0\ell_k}$, $H_{0u_k}$}}
\\&\le \overbrace{1- \Pr\Bigl( \bigcap_{k\in K^\ell}  F^{-1}(\ell_k) \le X_{n:k}
                           \cap\bigcap_{k\in K^u} X_{n:k} \le F^{-1}(u_k)
                        \Bigr) }^{\textrm{because $F^{-1}(\ell_k)\ge F_0^{-1}(\ell_k)$ for all $k\in K^\ell$, $F^{-1}(u_k)\le F_0^{-1}(u_k)$ for all $k\in K^u$}}
\\&\le \overbrace{1 - \Pr\Bigl( \bigcap_{k=1}^{n} F^{-1}(\ell_k) \le X_{n:k} \le F^{-1}(u_k) \Bigr) }^{\textrm{because }K^\ell,K^u\subseteq\{1,2,\ldots,n\}}
\\&= \overbrace{\alpha }^{\textrm{from \cref{eqn:Dir-1s-2s-tilde}}} 
. \qedhere
\end{align*}
\end{proof}

\subsection{Proof of Theorem \ref{thm:Dir-1s-stepdown-FWER}}
\begin{proof}
Consider the one-sided case with $H_{0\tau} \colon F^{-1}(\tau) \ge F_0^{-1}(\tau)$ for $\tau\in(0,1)$. 
We again focus on the $n$ hypotheses $F^{-1}(\ell_k)\ge F_0^{-1}(\ell_k)$ for $k=1,\ldots,n$ since rejections at other $\tau$ are simply by logical implication of the monotonicity of $F_0^{-1}(\cdot)$ and thus do not affect FWER. 
(The same is true in the two-sided case since it essentially combines lower and upper one-sided MTPs.) 

Let $K\equiv\{k:F^{-1}(\ell_k)\ge F_0^{-1}(\ell_k)\}$, the (true) set of true hypotheses. 
Let $r_{k^*}$ denote the order statistic indices that would be chosen by \cref{meth:Dir-1s-1s-stepdown} when attention is restricted to $k\in K$. 
(Many choices of $r_{k^*}$ still control FWER, but the choice must rely only on the set $K$.) 
Thus, for $k\in K$, the $r_{k^*}$ satisfy $r_{k^*}\le k$ and
\begin{equation}\label{eqn:1s-1s-FWER-pf-calib}
\alpha \ge 1 - \Pr\Bigl( \bigcap_{k\in K} \left\{X_{n:r_{k^*}} \ge F_0^{-1}(\ell_k) \right\} \Bigr) . 
\end{equation}
The stepdown procedure specifies monotonicity in the $r_{k,i}$ and $\hat{K}_i$ over iterations $i=0,1,\ldots$, where $\hat{K}_0=\{1,\ldots,n\}$ and $r_{k,0}=k$.  
Specifically, $\hat{K}_0\supset\hat{K}_1\supset\cdots$, and for each $k$, $r_{k,0}\ge r_{k,1}\ge\cdots$.  
This monotonicity is similar in spirit to (15.37) in \citet{LehmannRomano2005text}. 

The proof is by induction.  
Consider any dataset where 
\[ 1 = \prod_{k\in K}\Ind{ X_{n:r_{k^*}} \ge F_0^{-1}(\ell_k) } . \]
In iteration $i$, if $\hat{K}_i\supseteq K$, then none of the true hypotheses are rejected since $r_{k,i}\ge r_{k^*}$, which implies $X_{n:r_{k,i}}\ge X_{n:r_{k^*}} \ge F_0^{-1}(\ell_k)$.  
Consequently, $\hat{K}_{i+1}\supseteq K$, too.  
Since $\hat{K}_0\supseteq K$, the stepdown procedure does not reject any true hypothesis in such a dataset.  
Along with \cref{eqn:1s-1s-FWER-pf-calib}, this implies $\FWER \le \alpha$. 

The other one-sided case with $H_{0\tau} \colon F^{-1}(\tau) \le F_0^{-1}(\tau)$ is entirely parallel, simply reversing inequalities and replacing $\ell_k$ with $u_k$. 

For the two-sided case with $H_{0\tau} \colon F^{-1}(\tau) = F_0^{-1}(\tau)$, the key is again the monotonicity (by construction) in the $r_{k,\ell,i}$, $r_{k,u,i}$, $\hat{K}_i^\ell$, and $\hat{K}_i^u$. 
Specifically, $\hat{K}_0^\ell\supset\hat{K}_1^\ell\supset\cdots$, $\hat{K}_0^u\supset\hat{K}_1^u\supset\cdots$, and for each $k$, $r_{k,\ell,0}\ge r_{k,\ell,1}\ge\cdots$ and $r_{k,u,0}\le r_{k,u,1}\le\cdots$. 
Let $K^\ell\equiv\{k:F^{-1}(\ell_k)\ge F_0^{-1}(\ell_k)\}$ and $K^u\equiv\{k:F^{-1}(u_k)\le F_0^{-1}(u_k)\}$, the (true) sets of true hypotheses.  
For $k\in K^\ell$, let $r_{k^*,\ell}$ satisfy $r_{k^*,\ell}\le k$; 
for $k\in K^u$,    let $r_{k^*,u}$    satisfy $r_{k^*,u}   \ge k$. 
Also, these satisfy 
\begin{equation}\label{eqn:1s-2s-FWER-pf-calib}
\alpha 
\ge 1 - \Pr\Bigl( \bigcap_{k\in K^\ell} \left\{F_0^{-1}(\ell_k) \le X_{n:r_{k^*,\ell}} \right\}
              \cap\bigcap_{k\in K^u}    \left\{X_{n:r_{k^*,u}}  \le F_0^{-1}(u_k) \right\}
           \Bigr) . 
\end{equation}
As for the one-sided case, by induction, consider any dataset where 
\[ 1 = \prod_{k\in K^\ell}\Ind{ X_{n:r_{k^*,\ell}} \ge F_0^{-1}(\ell_k) }
       \prod_{k\in K^u}   \Ind{ X_{n:r_{k^*,u}}    \le F_0^{-1}(u_k) } . \]
In iteration $i$, if $\{\hat{K}_i^\ell \cup \hat{K}_i^u\}\supseteq\{K^\ell \cup K^u\}$, then none of the true hypotheses are rejected since $r_{k,\ell,i}\ge r_{k^*,\ell}$ and $r_{k,u,i}\le r_{k^*,u}$, which implies 
\[ X_{n:r_{k,\ell,i}} \ge X_{n:r_{k^*,\ell}} \ge F_0^{-1}(\ell_k) , \quad
   X_{n:r_{k, u  ,i}} \ge X_{n:r_{k^*,u   }} \le F_0^{-1}(u_k)
. \]
Consequently, $\{\hat{K}_{i+1}^\ell \cup \hat{K}_{i+1}^u\}\supseteq\{K^\ell \cup K^u\}$, too.  
Since $\hat{K}_0^\ell\supseteq K^\ell$ and $\hat{K}_0^u\supseteq K^u$, the stepdown procedure does not reject any true hypothesis in such a dataset.  
Along with \cref{eqn:1s-2s-FWER-pf-calib}, this implies $\FWER \le \alpha$. 
\end{proof}

\subsection{Proof of Proposition \ref{prop:pretest-FWER} (for proof of Theorem \ref{thm:Dir-1s-pretest-FWER})}
\begin{proposition}\label{prop:pretest-FWER}
Under \cref{a:iid,a:F}, \cref{meth:1s-pretest} has strong control of finite-sample FWER. 
\end{proposition}
\begin{proof}
The MTP is based on a one-sided uniform confidence band, so it strongly controls FWER by the same argument as in the proof of \cref{thm:Dir-1s-FWER}. 
That is, the uniform confidence band covers the entire $F(\cdot)$ with at least $1-\alpha$ probability, so it covers any subset of $F(\cdot)$ with at least $1-\alpha$ probability, too. 
Thus, FWER is below $1-(1-\alpha)=\alpha$. 
\end{proof}

\subsection{Proof of Theorem \ref{thm:Dir-1s-pretest-FWER}}
\begin{proof}
The stated FWER bound is conservative, relying on the following two worst-case assumptions.  
First, assume that any false pre-test rejection leads to a false rejection of the overall test.  
Second, assume that \cref{meth:Dir-1s-1s}, i.e., the test without using a pre-test, never falsely rejects when the pre-test falsely rejects.  
Then, the worst-case (i.e., upper bound for) FWER is $\alpha+\alpha_p$, where the $\alpha_p$ is guaranteed by \cref{prop:pretest-FWER}. 
\end{proof}

\subsection{Proof of Theorem \ref{thm:Dir-2s-FWER}}
\begin{proof}
To apply \cref{lem:weak-to-strong}, the method must reject $H_{0r}$ depending only on $\hat{F}_X(r)$ and $\hat{F}_Y(r)$, and it must have weak control of FWER. 
First, by construction, as seen in \cref{eqn:def-hat-ell-u-r}, given $n_X$, $n_Y$, and $\alpha$, the method will reject $H_{0r}$ depending only on $\hat{F}_X(r)$ (which determines the $\hat{u}_X(r)$ and $\hat{\ell}_X(r)$) and on $\hat{F}_Y(r)$ (which determines the $\hat{u}_Y(r)$ and $\hat{\ell}_Y(r)$). 

Second, for the two-sided MTP, weak control of FWER is by construction (up to simulation error). 
Weak control of FWER is equivalent to size control of the corresponding GOF test. 
When $F_X(\cdot)=F_Y(\cdot)$, the distribution of the ordering of the $X$ and $Y$ values is distribution-free given \cref{a:iid,a:F}; this (finite-sample) distribution is used explicitly to control the probability of any $H_{0r}$ being rejected at level $\alpha$. 

For the one-sided case, consider the GOF null $H_0 \colon F_X(\cdot) \le F_Y(\cdot)$. 
This is rejected based on the ordering of the $X_i$ and $Y_j$, which is the same as the ordering of the $F_Y(X_i)$ and $F_Y(Y_j)$. 
By construction, the ordering of $F_X(X_i)\stackrel{iid}{\sim}\UnifDist(0,1)$ and $F_Y(Y_j)$ will lead to rejection of $H_0$ with less than or equal to $\alpha$ probability. 
Under $H_0$, $F_Y(X_i)\ge F_X(X_i)$ for any $X_i$, so rejection of $H_0$ is even less likely with $F_Y(X_i)$ and size remains below $\alpha$. 
This corresponds to the intuition that $F_X(\cdot)=F_Y(\cdot)$ is the least favorable configuration (i.e., results in highest RP) among all distributions satisfying $H_0 \colon F_X(\cdot) \le F_Y(\cdot)$. 

Since the assumptions are met, \cref{lem:weak-to-strong} gives strong control of FWER. 
\Cref{sec:comp-discreteness} discusses the suggested $0.0001$ adjustment of $\tilde\alpha$ in light of possible simulation error and discontinuity in the mapping from $\tilde\alpha$ to $\alpha$. 
\end{proof}

\subsection{Proof of Theorem \ref{thm:RD}}
\begin{proof}
By Theorem 4.1 of \citet{CanayKamat2017},%
\footnote{The full continuity assumed in \citet{CanayKamat2017} is clearly not necessary, otherwise the RHS of the result would not include the $+$ and $-$ superscripts.} 
\begin{equation}\label{eqn:RD-Thm4.1}
\Pr\left( \bigcap_{j=1}^{q} \{ Y^{-}_{[j]} \le y^{-}_j \}  
          \bigcap_{j=1}^{q} \{ Y^{+}_{[j]} \le y^{+}_j \} \right) 
= \prod_{j=1}^{q} \lim_{x \uparrow x_0}
                  F_{Y|X}(y^{-}_j \mid x)
  \prod_{j=1}^{q} F_{Y|X}(y^{+}_j \mid x_0)
 +o(1) 
\end{equation}
as $n\to\infty$ for any $(y^{-}_1, \ldots, y^{-}_q, y^{+}_1, \ldots, y^{+}_q) \in \R^{2q}$. 
Let
\begin{equation*}
\vecf{Y}^{-}_n \equiv (Y^{-}_{[1]}, \ldots, Y^{-}_{[q]}) , \quad
\vecf{Y}^{+}_n \equiv (Y^{+}_{[1]}, \ldots, Y^{+}_{[q]}) . 
\end{equation*}
The result in \cref{eqn:RD-Thm4.1} is equivalent to 
$(\vecf{Y}^{-}_{n}, \vecf{Y}^{+}_{n}) \dconv (\vecf{Y}^{-}, \vecf{Y}^{+})$, where 
the $q$ elements of $\vecf{Y}^{-}$ are sampled iid from $\lim_{x \uparrow x_0} F_{Y|X}(\cdot \mid x)$, 
the $q$ elements of $\vecf{Y}^{+}$ are sampled iid from $F_{Y|X}(\cdot \mid x_0)$, and 
$\vecf{Y}^{-} \independent \vecf{Y}^{+}$. 

By the portmanteau lemma \citep[e.g.,][Lemma 2.2(i,vii)]{vanderVaart1998}, \cref{eqn:RD-Thm4.1} is equivalent to 
\begin{equation}\label{eqn:RD-Thm4.1-Borel}
    \Pr\bigl( (\vecf{Y}^{-}_{n}, \vecf{Y}^{+}_{n}) \in B \bigr)  
\to \Pr\bigl( (\vecf{Y}^{-}    , \vecf{Y}^{+})     \in B \bigr) 
\end{equation}
for any continuity set $B$ of $(\vecf{Y}^{-}, \vecf{Y}^{+})$. 
The result follows by setting $B$ equal to the rejection region for the MTP, specifically the set of values of $(\vecf{y}^{-},\vecf{y}^{+})$ for which any true $H_{0r}$ is (falsely) rejected. 
For such $B$, the finite-sample strong control of FWER in \cref{thm:Dir-2s-FWER} guarantees 
\begin{equation*}
\Pr\bigl( (\vecf{Y}^{-}, \vecf{Y}^{+}) \in B \bigr)
\le \alpha , 
\end{equation*}
so 
$\lim_{n\to\infty} 
 \Pr\bigl( (\vecf{Y}^{-}_n, \vecf{Y}^{+}_n) \in B \bigr) 
 \le \alpha$. 
\end{proof}

\subsection{Proof of Theorem \ref{thm:CQD}}
\begin{proof}
Since $\Pr(\vecf{X}^D=\vecf{x}^D_0)>0$, the number of observations in the subsample with $\vecf{X}^D_i=\vecf{x}^D_0$ goes to infinity almost surely. 
(By the Borel--Cantelli Lemma, it is almost surely of order $n$.) 
Thus, below we consider $\vecf{X}$ to contain only the continuous random variables. 

The proof largely parallels that of \cref{thm:RD}, with only minor modifications. 
Making the trivial change to allow vector $\vecf{X}$ (of finite, fixed dimension) instead of scalar $X$, Theorem 4.1 in \citet{CanayKamat2017} states that if $F_{Y|\vecf{X},T}(y \mid \vecf{x}, 0)$ is continuous in $\vecf{x}$ at $\vecf{x}=\vecf{x}_0$ for all $y$ in the support of $Y$ (like our theorem's Assumption (i)), and 
if 
$\Pr( \vecf{X} \in [-\epsilon,\epsilon]^{d} \mid T=0)>0$ 
for all $\epsilon>0$ (our Assumption (ii)), 
then as $n\to\infty$, for any $(y_1,\ldots,y_{q_0})\in\R^{q_0}$, 
\begin{equation}\label{eqn:CQD-Thm4.1}
\Pr\left( \bigcap_{j=1}^{q_0} \{ Y_{[j],T=0} \le y_j \} \right) 
= \prod_{j=1}^{q_0} F_{Y|\vecf{X},T}(y_j \mid \vecf{x}_0, 0)
 +o(1) , 
\end{equation}
i.e., $\vecf{Y}_{n,T=0} \dconv \vecf{Y}_{T=0}$, a vector whose ${q_0}$ elements are sampled iid from $F_{Y|\vecf{X},T}(\cdot \mid \vecf{x}_0 , 0)$. 
By the portmanteau lemma \citep[e.g.,][Lemma 2.2(i,vii)]{vanderVaart1998}, \cref{eqn:CQD-Thm4.1} is equivalent to 
\begin{equation}\label{eqn:CQD-Thm4.1-Borel}
\Pr( \vecf{Y}_{n,T=0} \in B )  \to  \Pr( \vecf{Y}_{T=0} \in B ) 
\end{equation}
for any continuity set $B$ of $\vecf{Y}_{T=0}$. 

The preceding arguments apply similarly to $T=1$, with parallel results like $\vecf{Y}_{n,T=1} \dconv \vecf{Y}_{T=1}$. 
Under iid sampling, $\vecf{Y}_{n,T=1} \independent \vecf{Y}_{n,T=0}$, so 
\begin{equation}\label{eqn:CQD-Thm4.1-Borel-2}
    \Pr\bigl( (\vecf{Y}_{n,T=0},\vecf{Y}_{n,T=1}) \in B \bigr)  
\to \Pr\bigl( (\vecf{Y}_{  T=0},\vecf{Y}_{  T=1}) \in B \bigr) 
\end{equation}
for any continuity set $B$ of $(\vecf{Y}_{  T=0},\vecf{Y}_{  T=1})$. 

The result follows by setting $B$ in \cref{eqn:CQD-Thm4.1-Borel-2} equal to the rejection region for the MTP, specifically the set of values of $(\vecf{y}_{T=0},\vecf{y}_{T=1})$ for which any true $H_{0r}$ is (falsely) rejected. 
For such $B$, the finite-sample, strong control of FWER in \cref{thm:Dir-2s-FWER} guarantees 
\begin{equation*}
\Pr\bigl( (\vecf{Y}_{  T=0},\vecf{Y}_{  T=1}) \in B \bigr)
\le \alpha , 
\end{equation*}
so $\lim_{n\to\infty} \Pr\bigl( (\vecf{Y}_{n,T=0},\vecf{Y}_{n,T=1}) \in B \bigr) \le \alpha$. 
\end{proof}

\section{Computational details}
\label{sec:app-comp}

We discuss some computational details of our code's implementation of our methods, specifically the simulation of the mapping from $\tilde\alpha$ to $\alpha$. 

\subsection{Calibration of \texorpdfstring{$\tilde\alpha$}{\~a}}\label{sec:comp-calib} 

Consider a given $n$.  
The joint distribution of the uniform order statistics is
\begin{equation*}
\left(U_{n:1},U_{n:2}-U_{n:1},U_{n:3}-U_{n:2},\ldots,U_{n:n}-U_{n:n-1},1-U_{n:n}\right)
  \sim \textrm{Dirichlet}\overbrace{\left(1,\ldots,1\right)}^\textrm{$n+1$} .
\end{equation*}
We simulate this with repeated random draws $U_i^{(m)}\stackrel{iid}{\sim}\UnifDist(0,1)$ for observations $i=1,\ldots,n$ in samples $m=1,\ldots,M$. 
Given $\tilde\alpha$, which determines all $\ell_k$ and $u_k$, the simulated two-sided FWER (for example) is
\begin{equation}\label{eqn:coverage-1s}
\hat\alpha
  = 1 - \frac{1}{M}\sum_{m=1}^M \Ind{\ell_1 < U_{n:1}^{(m)} < u_1} 
                              \times \cdots \times
  \Ind{\ell_n < U_{n:n}^{(m)} < u_n} .
\end{equation}

While \cref{eqn:coverage-1s} alone is sufficient for global (GOF) $p$-value computation, we need to search for the $\tilde\alpha$ that leads to a specific desired $\alpha$ for the simulations informing \cref{fact:alpha-tilde-rate}.  
Given search tolerance $T$ (see \cref{sec:comp-accuracy}), we stop the search over $\tilde\alpha$ if $\abs{\hat\alpha-\alpha}<T$.  
Otherwise, if $\hat\alpha<\alpha$ then $\tilde\alpha$ is increased, and if $\hat\alpha>\alpha$ then $\tilde\alpha$ is decreased. 
Since $\hat\alpha$ is a monotonic function of $\tilde\alpha$, which is a scalar, this is an easy search problem.  
Note that the random draws do not need to be repeated each iteration, only the $2n$ beta quantile function calls; or, the simulation is easily parallelized by slicing the $M$ samples across CPUs. 

With two samples, the only difference is \cref{eqn:coverage-1s}.  
The GOF null $H_0 \colon F_X(\cdot) = F_Y(\cdot)$ is rejected whenever there is at least one point where the band for one distribution lies strictly above the other band, i.e., at least one $H_{0r}$ is rejected. 
This depends on $\tilde\alpha$ and the relative ordering of values in the two samples, but not on the sample values themselves (more below).  
Because of this difference, with small sample sizes, there can be jumps of bigger than $T$ in $\hat\alpha$ as a function of $\tilde\alpha$, in which case we pick $\tilde\alpha$ slightly smaller than the point of discontinuity.  

The fact that the test's rejection is determined only by the ordering of values from the two samples (rather than the values themselves) is apparent from the construction of the test, as discussed in the main text. 
Each ordering of $X$ and $Y$ values is equally likely under $H_0 \colon F_X(\cdot) = F_Y(\cdot)$ and \cref{a:iid,a:F}; as usual, with larger sample sizes, permutations are randomly sampled rather than fully enumerated.

\subsection{Calibration accuracy}
\label{sec:comp-accuracy} 

As introduced in \cref{sec:comp-calib}, to search for the $\tilde\alpha$ that maps to a desired $\alpha$, the required number of Dirichlet draws ($M$) and the tolerance parameter ($T$) must be specified.  They may be determined given the desired overall simulation error.  
Given $\alpha$, we chose to determine $\tilde\alpha$ such that the true FWER would be within $c\alpha$ of the desired $\alpha$ for some small $c>0$, like $c=\calibdecerrfiveten$ for $\alpha=0.05$ implying FWER of $0.05\pm\caliberrfive$.  
As in \cref{sec:comp-calib}, the search stops when $\lvert \hat\alpha-\alpha \rvert < T$.  The $M$ Dirichlet draws are iid, so the total number of draws with a familywise error follows a binomial distribution.  
Since $M$ is large, the normal approximation is quite accurate.  
We want the simulation to have a high probability, like $1-p=0.95$, of estimating $\hat\alpha>\alpha+T$ when $\tilde\alpha$ yields a true FWER above $\alpha(1+c)$. 
If the true FWER is $\alpha(1+c)$, then the total number of simulated familywise errors follows a $\textrm{Binomial}\left(M,\alpha(1+c)\right)$ distribution, so $\hat\alpha\stackrel{a}{\sim}\Normal\bigl(\alpha(1+c), \alpha(1+c)\left[1-\alpha(1+c)\right]/M \bigr)$, and we choose $T$ and $M$ to equate $T$ with the $p$-quantile of this distribution:
\begin{align*}
\alpha+T
  &= \alpha(1+c) + \Phi^{-1}(p)\sqrt{\alpha(1+c)(1-\alpha(1+c))} / \sqrt M , \\
T &= c\alpha - \Phi^{-1}(1-p)\sqrt{\alpha(1+c)(1-\alpha(1+c))} / \sqrt M , \\
M &= \left(\frac{\Phi^{-1}(1-p)\sqrt{\alpha(1+c)(1-\alpha(1+c))}}{c\alpha-T}\right)^2 .
\end{align*}
For $\alpha\in\{0.10,0.05\}$, we used $M=\NDRAWSfiveten$, $p=0.05$, and $c=\calibdecerrfiveten$, leading to $T\approx0.00019$ for $\alpha=0.05$ and $T\approx0.00089$ for $\alpha=0.10$, as seen in the lookup table. 
For $\alpha=0.01$, we used $M=\NDRAWSone$, $p=0.05$, and $c=\calibdecerrone$, leading to $T\approx0.00033$. 
The foregoing discussion applies equally to one-sample and two-sample inference.

\subsection{Two-sample adjustment for discreteness}
\label{sec:comp-discreteness}

In the two-sample setting, the mapping from $\tilde\alpha$ to $\alpha$ is still monotonic but not continuous: it is a step function. 
Consequently, we suggest subtracting a small amount like $0.0001$ from whichever $\tilde\alpha$ is found by the numerical solver. 
Additionally, in our lookup table of pre-computed values, we report both the smaller and larger $\alpha$ values at the discontinuity, to show how big the possible FWER inflation is if the simulation error is large enough that actually the next-highest $\alpha$ is the true FWER. 

The subtraction of $0.0001$ from the simulated $\tilde\alpha$ is because simulation error does not necessarily go to zero as the number of simulations goes to infinity, because the number of attainable $\alpha$ is finite. 
That is, the mapping from $\tilde\alpha$ to FWER is a step function, so if one picks the largest possible $\tilde\alpha$ such that FWER is below $\alpha$, even an infinitessimal amount of simulation error could mean that actual FWER is above $\alpha$. 
For example, if actual FWER equals $0.08+0.04\Ind{\tilde\alpha\ge0.03}$, but simulated FWER is $0.08+0.04\Ind{\tilde\alpha>0.03}$, then $\tilde\alpha=0.03$ appears to control FWER below $\alpha=0.1$ in the simulation, but actual FWER is $0.12$, above $\alpha$. 
Subtracting any small, fixed amount from the simulated $\tilde\alpha$ is sufficient to overcome this problem (with probability approaching one) as the number of simulation draws grows arbitrarily large. 


\section{Additional simulations}\label{sec:app-sim}

\subsection{Power compared to KS-based methods}\label{sec:sim-pt-pwr}

Earlier, \cref{fig:sim-1s-pt,fig:sim-2s-pt} showed simulation results on the uneven sensitivity of KS-based MTPs and the (relatively) even sensitivity of the Dirichlet MTPs, in terms of pointwise type I error rates. 
Naturally, those differences translate into corresponding differences in pointwise power.  
\Cref{fig:sim-1s-power-shift,fig:sim-1s-power-spread} show patterns similar to \cref{fig:sim-1s-pt}: the KS-based MTP has the highest (among the three methods) pointwise power against deviations near the median of a distribution and lowest pointwise power in the tails, and the weighted KS-based MTP is usually the opposite (depending whether the null is above or below the true distribution; see below). 
The Dirichlet MTP has the highest pointwise power against deviations in between the middle and the tails, and it never has the lowest.  

\Cref{fig:sim-1s-power-shift,fig:sim-1s-power-spread} show examples of pointwise power for two-sided $H_{0\tau} \colon F^{-1}(\tau) = F_0^{-1}(\tau)$ over $\tau\in(0,1)$.  
The left column graphs show $F_0\bigl(F^{-1}(\tau)\bigr)$ (dashed line). 
If $H_0$ were true, then $F_0\bigl(F^{-1}(\tau)\bigr)=\tau$ (solid line). 
Similar to \cref{fig:sim-1s-pt}, the right column graphs show RPs due to each order statistic. 

\Cref{fig:sim-1s-power-shift} shows $X_i\stackrel{iid}{\sim}\Normalp{0.3}{1}$ when the null is $\Normalp{0}{1}$.  
As the left column shows, this leads to larger deviations in the middle of the distribution than in the tails.  
The largest peak in pointwise power is in the middle of the distribution for KS: this is where both the deviations are largest and the KS pointwise size is largest.  
The Dirichlet pointwise power peaks in a similar range, but at a lower level, corresponding to its lower pointwise size in that range.  
The weighted KS pointwise power peaks in the lower tail, at a much lower level since the deviations are smaller. 

\begin{figure}[htb]
\centering
\hfill
\includegraphics[width=0.42\textwidth,clip=true,trim=20 25 10 82]{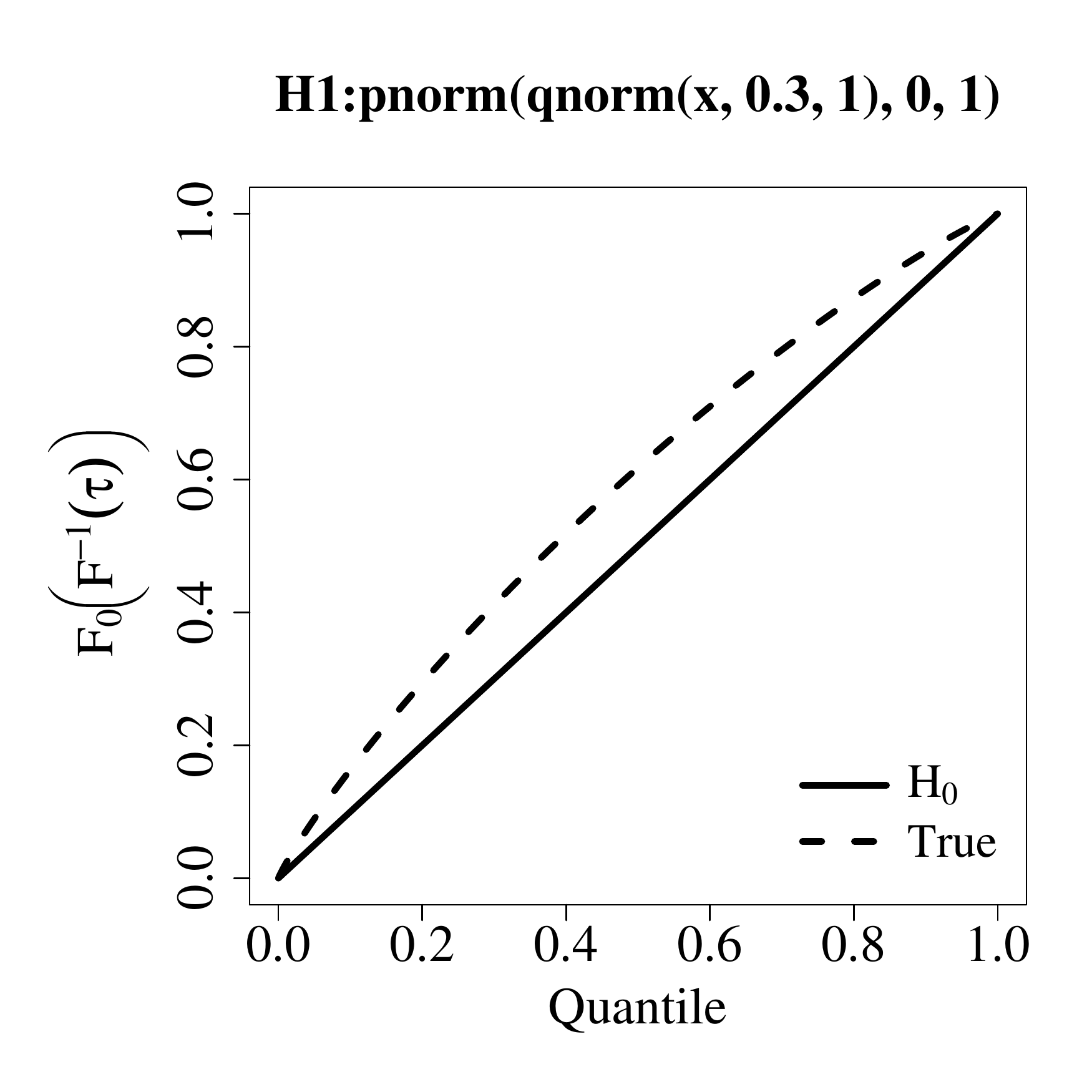}%
\hfill%
\includegraphics[width=0.4\textwidth,clip=true,trim=0 0 0 82]{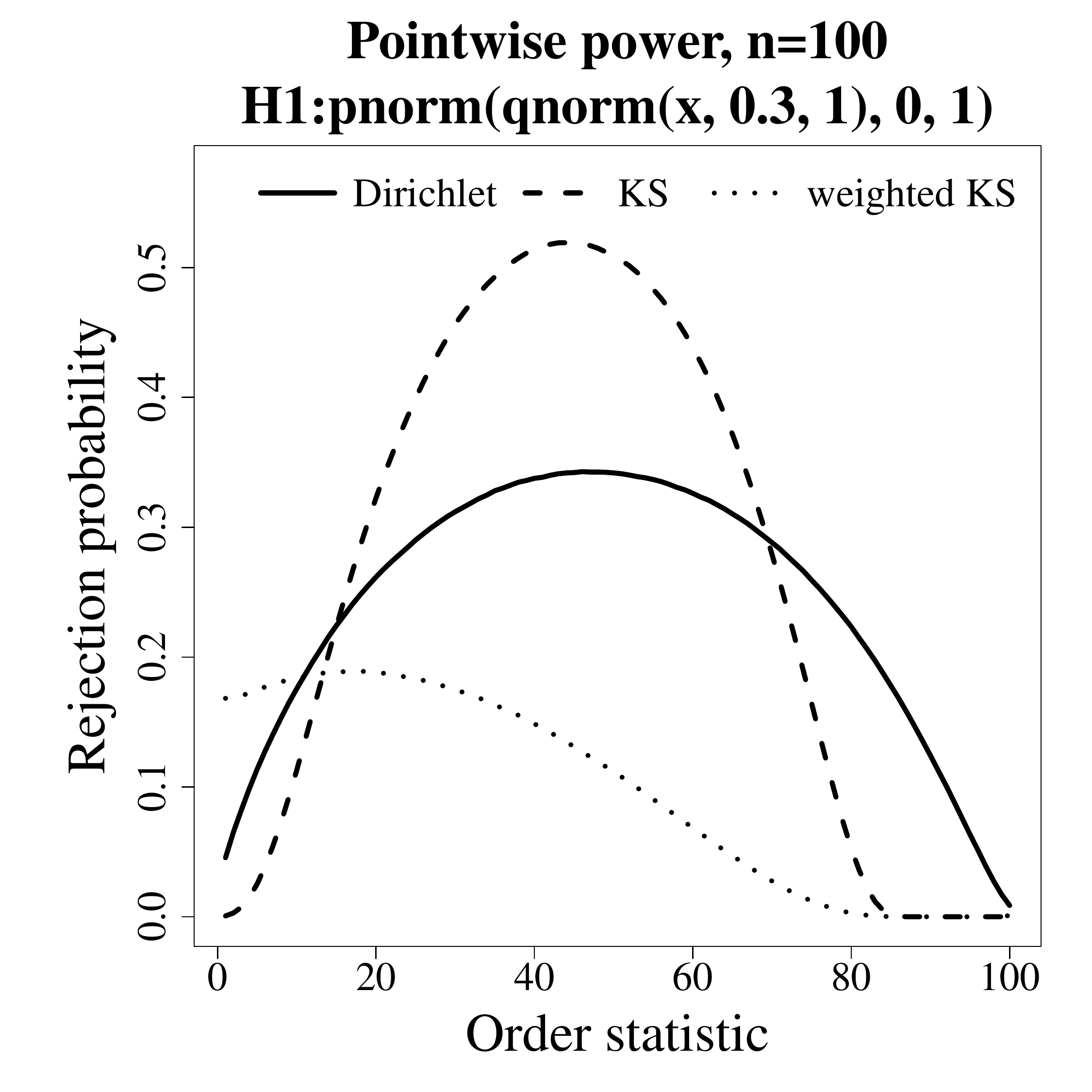} %
\hfill\null
\caption{\label{fig:sim-1s-power-shift}Simulated one-sample, two-sided RPs by order statistic when all $H_{0\tau}$ are false, $F_0=\Normalp{0}{1}$, $X_i\stackrel{iid}{\sim}\Normalp{0.3}{1}$, FWER $\alpha=0.1$, $n=100$, $10^6$ replications.}
\end{figure}

In \cref{fig:sim-1s-power-shift}, the effect having pointwise equal-tailed (like Dirichlet) or symmetric (like KS) tests is apparent.  
Even though the weighted KS has greater (than Dirichlet) two-sided pointwise type I error rate in the upper tail, it has essentially zero power in the upper tail in the examples provided, whereas Dirichlet has substantial power.  
This is because $F_0(x)>F(x)$ in the upper tail; regardless of weighting, KS-based MTPs (or tests) are insensitive to such deviations, whereas the Dirichlet MTP is sensitive to both upper and lower deviations. 

In the row of \cref{fig:sim-1s-power-spread} where $\sigma=1.2$, the weighted KS again has pointwise power near zero even in the tails. 
This is an example of the same general feature seen in \cref{fig:sim-1s-power-spread}: because of being pointwise symmetric instead of equal-tailed, the KS approach (whether weighted or not) has low power against a null with smaller variance than the DGP. 
The Dirichlet has two pointwise power peaks, reflecting the varying distance between the two curves in the corresponding left column graph. 
The KS has a much smaller pointwise power peak surrounding the median, where the deviations are small (and even zero right at the median) but its sensitivity is highest.

For the graph in \cref{fig:sim-1s-power-spread} with $\sigma=0.7$, the weighted KS pointwise power has the highest peak, in the tails (and highest at the extremes) where the deviations are large and its sensitivity is large.  
The Dirichlet has a somewhat smaller peak, also in the tails but not at the extremes. 
Even smaller and closer to the middle is the KS peak.  
The weighted KS and KS can have very high peaks since their peak pointwise type I error rate is higher than Dirichlet's (which has no peak), but they perform poorly when their peak pointwise type I error rate coincides with low deviations from the null hypothesis. 
The Dirichlet is more even-keeled, yet it can still have the highest peak pointwise power of the three methods, especially if the deviations are largest in between the tails and median (where its pointwise size is largest), a case not even shown in these graphs.

\begin{figure}[htb]
\centering
\hfill
\includegraphics[width=0.46\textwidth,clip=true,trim=20 25 10 80]{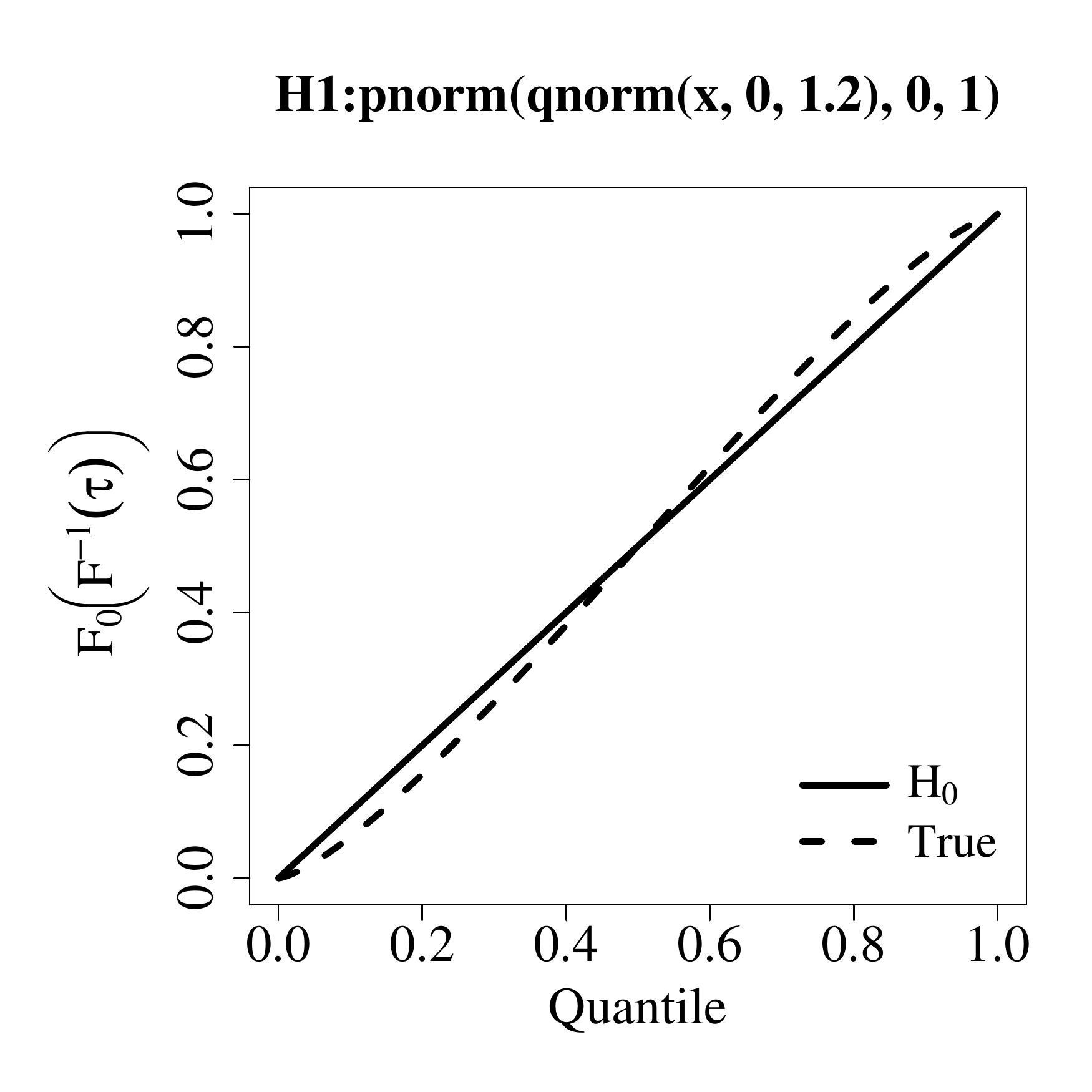}%
\hfill%
\includegraphics[width=0.44\textwidth,clip=true,trim=0 0 0 80]{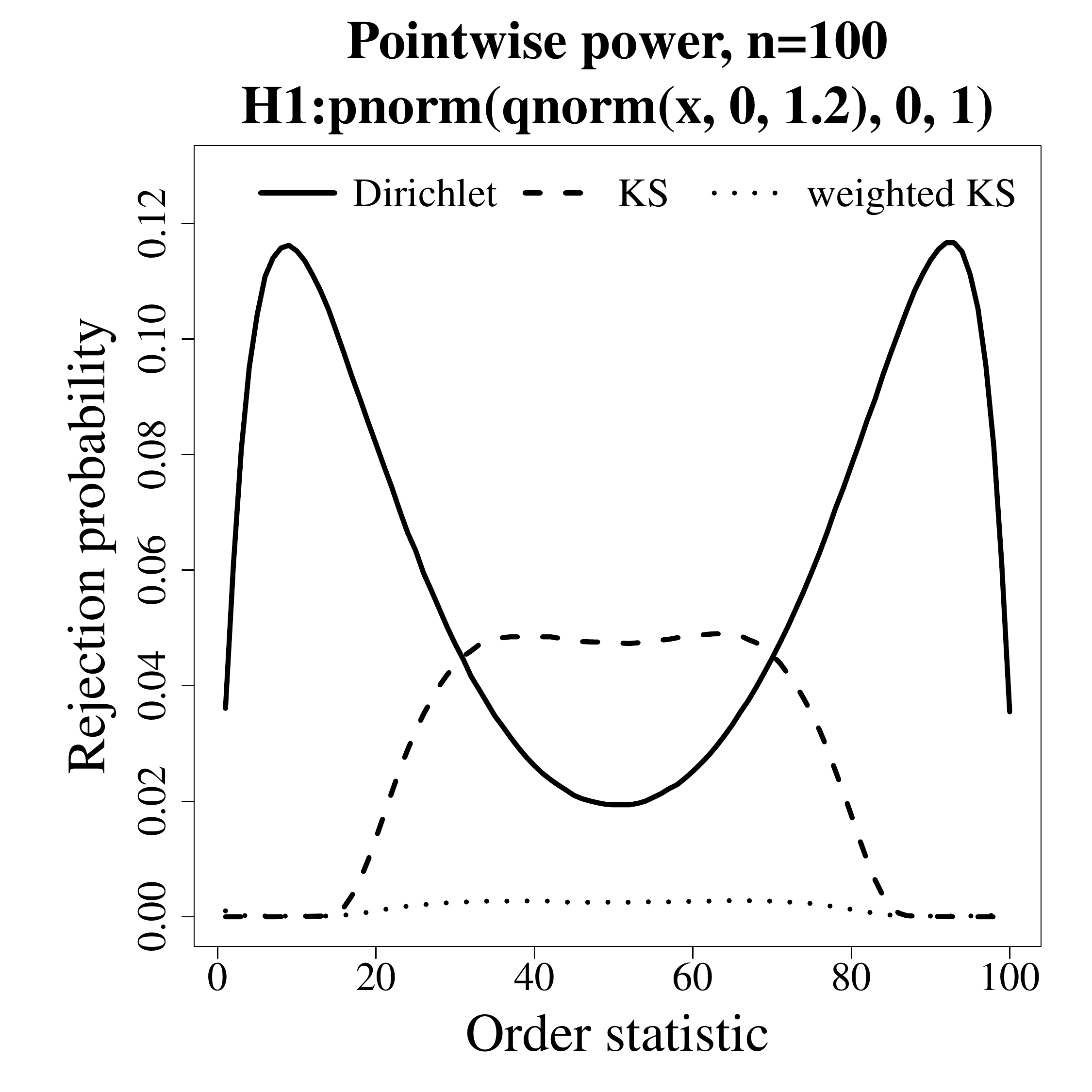} %
\hfill\null
\\
\null%
\hfill
\includegraphics[width=0.46\textwidth,clip=true,trim=20 25 10 80]{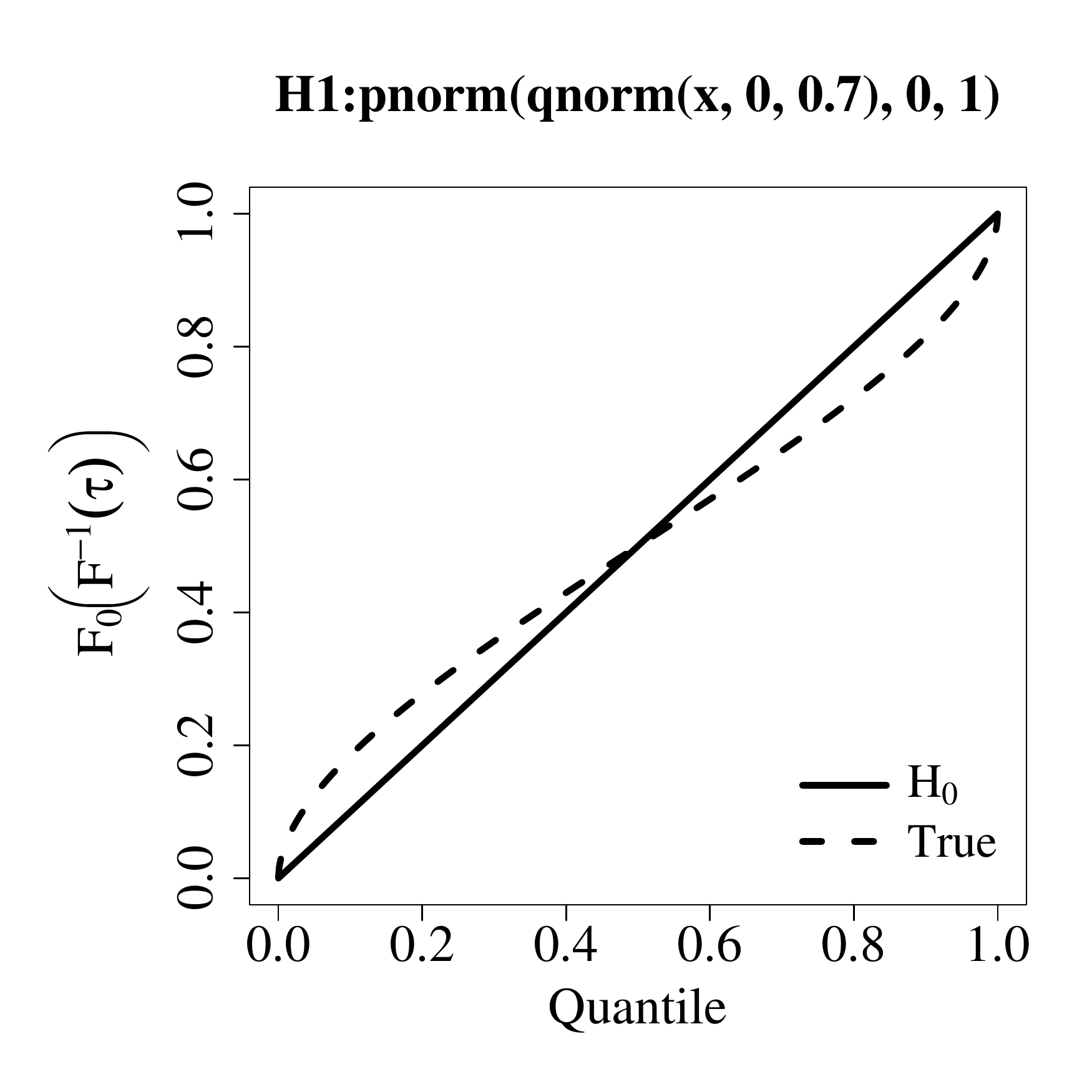}%
\hfill%
\includegraphics[width=0.44\textwidth,clip=true,trim=0 0 0 80]{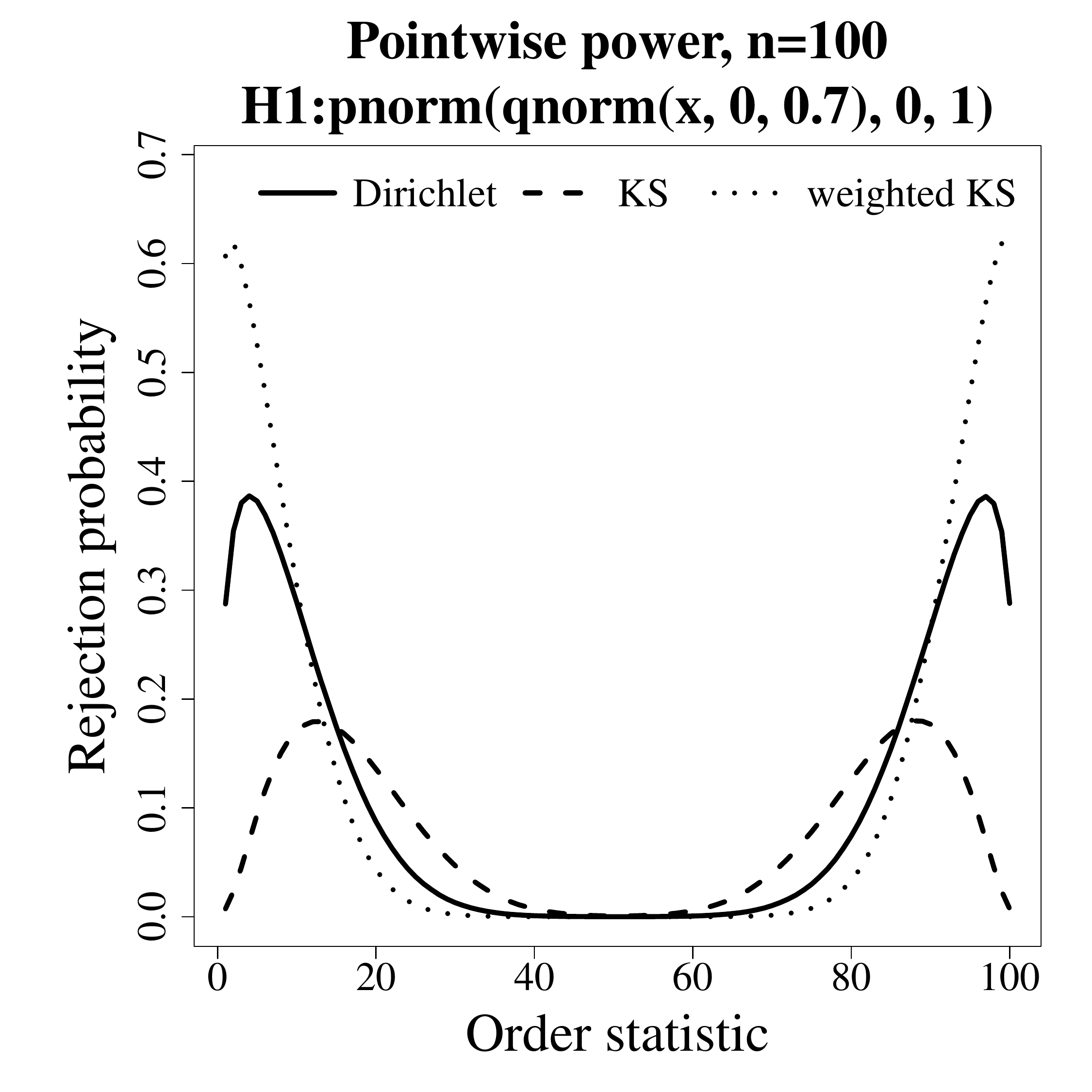}%
\hfill\null
\caption{\label{fig:sim-1s-power-spread}Simulated one-sample, two-sided RPs by order statistic when all $H_{0\tau}$ are false (except $\tau=0.5$), 
$F_0=\Normalp{0}{1}$, FWER $\alpha=0.1$, $n=100$, $10^6$ replications; $X_i\stackrel{iid}{\sim}\Normalp{0}{\sigma^2}$ 
with $\sigma=1.2$ (top) or $\sigma=0.7$ (bottom).}
\end{figure}

\Cref{tab:sim-1s-power} shows global power for one-sample, two-sided GOF tests of $H_0 \colon F(\cdot) = F_0(\cdot)$ with $F_0=\Normalp{0}{1}$ and $X_i\stackrel{iid}{\sim}\Normalp{\mu}{\sigma^2}$. 
For the Dirichlet, KS, and weighted KS tests alike, this is equivalent to testing $H_0 \colon F_0(X_i) \stackrel{iid}{\sim} \UnifDist(0,1)$, or $H_0 \colon F_0\bigl(F^{-1}(\tau)\bigr) = \tau$.%
\footnote{When the population CDF is $F(\cdot)$, then $F_0(X_i)=F_0(F^{-1}(F(X_i)))=F_0(F^{-1}(U_i))$, $U_i\stackrel{iid}{\sim}\UnifDist(0,1)$.} 
For pure location shifts with $\mu\ne0$ and $\sigma=1$, the deviations (of $F_0\bigl(F^{-1}(\tau)\bigr)$ from $\tau$) are largest near the middle of the distribution, where KS has the largest pointwise power.  
The weighted KS is not very sensitive to such deviations, so it has the worst power by far.  
The Dirichlet power is below KS, but only by a couple percentage points.  
With $\mu=0$ and $\sigma=0.7$, the largest vertical deviations of $F_0\bigl(F^{-1}(\tau)\bigr)$ are in the tails (i.e., near zero and one). 
Consequently, the weighted KS has the best power.  
The KS test has significantly lower power, but the Dirichlet is close to the weighted KS. 
With $\mu=0$ and $\sigma=0.8$, Dirichlet power is again between weighted KS (best) and KS (worst). 
When $\mu=0$ and $\sigma=1.2$, the deviations of $F_0\bigl(F^{-1}(\tau)\bigr)$ are no longer largest at the extremes. 
This poses a problem for the weighted KS, and its power is even lower than its size.  
Even though there is zero deviation at $\tau=0.5$, KS has better power than weighted KS in this case because it has better pointwise power around the upper and lower quartiles. 
The Dirichlet pointwise power is even higher in those regions, so its global power is far above either KS or weighted KS.

\begin{table}[htbp]
\centering
\caption{\label{tab:sim-1s-power}Simulated global power, one-sample, two-sided, $\alpha=0.1$, $n=100$.}
\begin{threeparttable}
\begin{tabular}{ccccc}
\toprule
$\mu$ & $\sigma$ & Dirichlet & KS & weighted KS \\
\midrule
0.3 & 1.0 & 80.5 & 82.4 & 62.4 \\
0.2 & 1.0 & 49.4 & 52.2 & 33.9 \\
0.0 & 0.7 & 92.0 & 65.6 & 98.6 \\
0.0 & 0.8 & 50.1 & 26.7 & 76.6 \\
0.0 & 1.2 & 64.2 & 25.5 & \phantom{9}2.8 \\
\bottomrule
\end{tabular}
\begin{tablenotes}
\item \textbf{Note:} $H_0 \colon F(\cdot) = F_0(\cdot)$, $F_0=\Normalp{0}{1}$, $X_i\stackrel{iid}{\sim}\Normalp{\mu}{\sigma^2}$, $10^6$ replications.  
RPs are shown as percentages. 
All methods have exact size. 
\end{tablenotes}
\end{threeparttable}
\end{table}

Additionally, Table 1 and Figure 8 in \citet{AldorNoimanEtAl2013} show a power advantage of the Dirichlet GOF test over the KS and Anderson--Darling (i.e., weighted Cram\'er--von Mises) tests for a variety of distributions.  

% ...


\begin{thebibliography}{52}
\expandafter\ifx\csname natexlab\endcsname\relax\def\natexlab#1{#1}\fi
\expandafter\ifx\csname url\endcsname\relax
  \def\url#1{\texttt{#1}}\fi
\expandafter\ifx\csname urlprefix\endcsname\relax\def\urlprefix{URL }\fi

\bibitem[{Aldor-Noiman et~al.(2013)Aldor-Noiman, Brown, Buja, Rolke, and
  Stine}]{AldorNoimanEtAl2013}
Aldor-Noiman, S., Brown, L.~D., Buja, A., Rolke, W., Stine, R.~A., 2013. The
  power to see: A new graphical test of normality. The American Statistician
  67~(4), 249--260.
\newline\urlprefix\url{https://doi.org/10.1080/00031305.2013.847865}

\bibitem[{Anderson and Darling(1952)}]{AndersonDarling1952}
Anderson, T.~W., Darling, D.~A., 1952. Asymptotic theory of certain ``goodness
  of fit'' criteria based on stochastic processes. Annals of Mathematical
  Statistics 23~(2), 193--212.
\newline\urlprefix\url{http://www.jstor.org/stable/2236446}

\bibitem[{Athey and Imbens(2006)}]{AtheyImbens2006}
Athey, S., Imbens, G.~W., 2006. Identification and inference in nonlinear
  difference-in-differences models. Econometrica 74~(2), 431--497.
\newline\urlprefix\url{http://www.jstor.org/stable/3598807}

\bibitem[{Banks(1988)}]{Banks1988}
Banks, D.~L., 1988. Histospline smoothing the {Bayesian} bootstrap. Biometrika
  75~(4), 673--684.
\newline\urlprefix\url{http://www.jstor.org/stable/2336308}

\bibitem[{Benjamini and Hochberg(1995)}]{BenjaminiHochberg1995}
Benjamini, Y., Hochberg, Y., 1995. Controlling the false discovery rate: a
  practical and powerful approach to multiple testing. Journal of the Royal
  Statistical Society: Series B (Statistical Methodology) 57~(1), 289--300.
\newline\urlprefix\url{http://www.jstor.org/stable/2346101}

\bibitem[{Beran and Hall(1993)}]{BeranHall1993}
Beran, R., Hall, P., 1993. Interpolated nonparametric prediction intervals and
  confidence intervals. Journal of the Royal Statistical Society: Series B
  (Statistical Methodology) 55~(3), 643--652.
\newline\urlprefix\url{http://www.jstor.org/stable/2345876}

\bibitem[{Berk and Jones(1979)}]{BerkJones1979}
Berk, R.~H., Jones, D.~H., 1979. Goodness-of-fit test statistics that dominate
  the {Kolmogorov} statistics. Zeitschrift f{\"u}r {Wahrscheinlichkeitstheorie}
  und verwandte {Gebiete} 47~(1), 47--59.
\newline\urlprefix\url{https://doi.org/10.1007/BF00533250}

\bibitem[{Bitler et~al.(2006)Bitler, Gelbach, and Hoynes}]{BitlerEtAl2006}
Bitler, M.~P., Gelbach, J.~B., Hoynes, H.~W., 2006. What mean impacts miss:
  Distributional effects of welfare reform experiments. American Economic
  Review 96~(4), 988--1012.
\newline\urlprefix\url{http://www.jstor.org/stable/30034327}

\bibitem[{Bitler et~al.(2008)Bitler, Gelbach, and Hoynes}]{BitlerEtAl2008}
Bitler, M.~P., Gelbach, J.~B., Hoynes, H.~W., 2008. Distributional impacts of
  the {Self-Sufficiency Project}. Journal of Public Economics 92~(3--4),
  748--765.
\newline\urlprefix\url{https://doi.org/10.1016/j.jpubeco.2007.07.001}

\bibitem[{Buja and Rolke(2006)}]{BujaRolke2006}
Buja, A., Rolke, W., 2006. Calibration for simultaneity: (re)sampling methods
  for simultaneous inference with applications to function estimation and
  functional data, working paper, available at
  \url{http://stat.wharton.upenn.edu/~buja/PAPERS/paper-sim.pdf}.

\bibitem[{Cameron et~al.(2008)Cameron, Gelbach, and Miller}]{CameronEtAl2008}
Cameron, A.~C., Gelbach, J.~B., Miller, D.~L., 2008. Bootstrap-based
  improvements for inference with clustered errors. Review of Economics and
  Statistics 90~(3), 414--427.
\newline\urlprefix\url{https://doi.org/10.1162/rest.90.3.414}

\bibitem[{Canay and Kamat(2017)}]{CanayKamat2017}
Canay, I.~A., Kamat, V., 2017. Approximate permutation tests and induced order
  statistics in the regression discontinuity design. Tech. Rep. CWP21/17,
  Centre for Microdata Methods and Practice (CeMMAP).
\newline\urlprefix\url{http://faculty.wcas.northwestern.edu/~iac879/wp/RDDPermutations.pdf}

\bibitem[{Cattaneo et~al.(2015)Cattaneo, Frandsen, and
  Titiunik}]{CattaneoEtAl2015}
Cattaneo, M.~D., Frandsen, B.~R., Titiunik, R., 2015. Randomization inference
  in the regression discontinuity design: An application to party advantages in
  the {U.S. Senate}. Journal of Causal Inference 3~(1), 1--24.
\newline\urlprefix\url{https://doi.org/10.1515/jci-2013-0010}

\bibitem[{Chicheportiche and Bouchaud(2012)}]{ChicheporticheBouchaud2012}
Chicheportiche, R., Bouchaud, J.-P., 2012. Weighted {Kolmogorov--Smirnov} test:
  accounting for the tails. Physical Review E 86~(4), 041115.
\newline\urlprefix\url{https://doi.org/10.1103/PhysRevE.86.041115}

\bibitem[{David and Nagaraja(2003)}]{DavidNagaraja2003}
David, H.~A., Nagaraja, H.~N., 2003. Order Statistics, 3rd Edition. Wiley, New
  York.
\newline\urlprefix\url{https://doi.org/10.1002/0471722162}

\bibitem[{Davidson and Duclos(2013)}]{DavidsonDuclos2013}
Davidson, R., Duclos, J.-Y., 2013. Testing for restricted stochastic dominance.
  Econometric Reviews 32~(1), 84--125.
\newline\urlprefix\url{https://doi.org/10.1080/07474938.2012.690332}

\bibitem[{Djebbari and Smith(2008)}]{DjebbariSmith2008}
Djebbari, H., Smith, J., 2008. Heterogeneous impacts in {PROGRESA}. Journal of
  Econometrics 145~(1), 64--80.
\newline\urlprefix\url{https://doi.org/10.1016/j.jeconom.2008.05.012}

\bibitem[{Donald and Hsu(2016)}]{DonaldHsu2016}
Donald, S.~G., Hsu, Y.-C., 2016. Improving the power of tests of stochastic
  dominance. Econometric Reviews 35~(4), 553--585.
\newline\urlprefix\url{https://doi.org/10.1080/07474938.2013.833813}

\bibitem[{Eicker(1979)}]{Eicker1979}
Eicker, F., 1979. The asymptotic distribution of the suprema of the
  standardized empirical processes. Annals of Statistics 7~(1), 116--138.
\newline\urlprefix\url{http://www.jstor.org/stable/2958837}

\bibitem[{Firpo and Galvao(2015)}]{FirpoGalvao2015}
Firpo, S., Galvao, A.~F., 2015. Uniform inference on functionals of quantiles
  of potential outcomes, working paper.

\bibitem[{Fisher(1932)}]{Fisher1932}
Fisher, R.~A., 1932. Statistical Methods for Research Workers, 4th Edition.
  Oliver and Boyd, Edinburgh.

\bibitem[{Gneezy and List(2006)}]{GneezyList2006}
Gneezy, U., List, J.~A., 2006. Putting behavioral economics to work: Testing
  for gift exchange in labor markets using field experiments. Econometrica
  74~(5), 1365--1384.
\newline\urlprefix\url{https://doi.org/10.1111/j.1468-0262.2006.00707.x}

\bibitem[{Goldman and Kaplan(2017{\natexlab{a}})}]{GoldmanKaplan2017a}
Goldman, M., Kaplan, D.~M., 2017{\natexlab{a}}. Fractional order statistic
  approximation for nonparametric conditional quantile inference. Journal of
  Econometrics 196~(2), 331--346.
\newline\urlprefix\url{https://doi.org/10.1016/j.jeconom.2016.09.015}

\bibitem[{Goldman and Kaplan(2017{\natexlab{b}})}]{GoldmanKaplan2017b}
Goldman, M., Kaplan, D.~M., 2017{\natexlab{b}}. Nonparametric inference on
  conditional quantile differences and linear combinations, using
  {$L$-statistics}. Econometrics Journal XXX~(XX), XXX--XXX.
\newline\urlprefix\url{https://doi.org/10.1111/ectj.12095}

\bibitem[{Holm(1979)}]{Holm1979}
Holm, S., 1979. A simple sequentially rejective multiple test procedure.
  Scandinavian Journal of Statistics 6~(2), 65--70.
\newline\urlprefix\url{http://www.jstor.org/stable/4615733}

\bibitem[{Jackson and Page(2013)}]{JacksonPage2013}
Jackson, E., Page, M.~E., 2013. Estimating the distributional effects of
  education reforms: A look at {Project STAR}. Economics of Education Review
  32, 92--103.
\newline\urlprefix\url{https://doi.org/10.1016/j.econedurev.2012.07.017}

\bibitem[{Jaeschke(1979)}]{Jaeschke1979}
Jaeschke, D., 1979. The asymptotic distribution of the supremum of the
  standardized empirical distribution function on subintervals. Annals of
  Statistics 7~(1), 108--115.
\newline\urlprefix\url{http://www.jstor.org/stable/2958836}

\bibitem[{Kaplan and Zhuo(2017)}]{KaplanZhuo2017a}
Kaplan, D.~M., Zhuo, L., 2017. Bayesian and frequentist nonlinear inequality
  tests, working paper, available at
  \url{https://faculty.missouri.edu/~kaplandm}.

\bibitem[{Kolmogorov(1933)}]{Kolmogorov1933}
Kolmogorov, A.~N., 1933. Sulla determinazione empirica di una legge di
  distribuzione. Giornale dell'Istituto Italiano degli Attuari 4~(1), 83--91.

\bibitem[{Lehmann and Romano(2005{\natexlab{a}})}]{LehmannRomano2005fwer}
Lehmann, E.~L., Romano, J.~P., 2005{\natexlab{a}}. Generalizations of the
  familywise error rate. Annals of Statistics 33~(3), 1138--1154.
\newline\urlprefix\url{https://doi.org/10.1214/009053605000000084}

\bibitem[{Lehmann and Romano(2005{\natexlab{b}})}]{LehmannRomano2005text}
Lehmann, E.~L., Romano, J.~P., 2005{\natexlab{b}}. Testing Statistical
  Hypotheses, 3rd Edition. Springer Texts in Statistics. Springer.
\newline\urlprefix\url{http://books.google.com/books?id=Y7vSVW3ebSwC}

\bibitem[{Linton et~al.(2010)Linton, Song, and Whang}]{LintonEtAl2010}
Linton, O., Song, K., Whang, Y.-J., 2010. An improved bootstrap test of
  stochastic dominance. Journal of Econometrics 154~(2), 186--202.
\newline\urlprefix\url{https://doi.org/10.1016/j.jeconom.2009.08.002}

\bibitem[{Lo(1993)}]{Lo1993}
Lo, A.~Y., 1993. A {Bayesian} method for weighted sampling. Annals of
  Statistics 21~(4), 2138--2148.
\newline\urlprefix\url{http://www.jstor.org/stable/2242333}

\bibitem[{Lockhart(1991)}]{Lockhart1991}
Lockhart, R.~A., 1991. Overweight tails are inefficient. Annals of Statistics
  19~(4), 2254--2258.
\newline\urlprefix\url{http://www.jstor.org/stable/2241930}

\bibitem[{MaCurdy et~al.(2011)MaCurdy, Chen, and Hong}]{MaCurdyEtAl2011}
MaCurdy, T., Chen, X., Hong, H., 2011. Flexible estimation of treatment effect
  parameters. American Economic Review (Papers and Proceedings) 101~(3),
  544--551.
\newline\urlprefix\url{http://www.jstor.org/stable/29783804}

\bibitem[{Moscovich and Nadler(2017)}]{MoscovichNadler2017}
Moscovich, A., Nadler, B., 2017. Fast calculation of boundary crossing
  probabilities for {Poisson} processes. Statistics \& Probability Letters 123,
  177--182.
\newline\urlprefix\url{https://doi.org/10.1016/j.spl.2016.11.027}

\bibitem[{Moscovich et~al.(2016)Moscovich, Nadler, and
  Spiegelman}]{MoscovichEtAl2016}
Moscovich, A., Nadler, B., Spiegelman, C., 2016. On the exact {Berk--Jones}
  statistics and their $p$-value calculation. Electronic Journal of Statistics
  10~(2), 2329--2354.
\newline\urlprefix\url{http://projecteuclid.org/euclid.ejs/1472829397}

\bibitem[{Neyman(1937)}]{Neyman1937}
Neyman, J., 1937. {\guillemotright}{Smooth} test{\guillemotright} for goodness
  of fit. Skandinavisk Aktuarietidskrift 20~(3--4), 149--199.
\newline\urlprefix\url{https://doi.org/10.1080/03461238.1937.10404821}

\bibitem[{Owen(1995)}]{Owen1995}
Owen, A.~B., 1995. Nonparametric likelihood confidence bands for a distribution
  function. Journal of the American Statistical Association 90~(430), 516--521.
\newline\urlprefix\url{http://www.jstor.org/stable/2291062}

\bibitem[{Pearson(1933)}]{Pearson1933}
Pearson, K., 1933. On a method of determining whether a sample of size $n$
  supposed to have been drawn from a parent population having a known
  probability integral has probably been drawn at random. Biometrika 25,
  379--410.
\newline\urlprefix\url{https://doi.org/10.1093/biomet/25.3-4.379}

\bibitem[{Qu and Yoon(2015)}]{QuYoon2015}
Qu, Z., Yoon, J., 2015. Nonparametric estimation and inference on conditional
  quantile processes. Journal of Econometrics 185~(1), 1--19.
\newline\urlprefix\url{https://doi.org/10.1016/j.jeconom.2014.10.008}

\bibitem[{{R Core Team}(2017)}]{R.core}
{R Core Team}, 2017. R: A Language and Environment for Statistical Computing. R
  Foundation for Statistical Computing, Vienna, Austria.
\newline\urlprefix\url{https://www.R-project.org/}

\bibitem[{Romano et~al.(2010)Romano, Shaikh, and Wolf}]{RomanoEtAl2010}
Romano, J.~P., Shaikh, A.~M., Wolf, M., 2010. Multiple testing. In: Durlauf,
  S.~N., Blume, L.~E. (Eds.), The New Palgrave Dictionary of Economics, online
  Edition. Palgrave Macmillan.
\newline\urlprefix\url{https://doi.org/10.1057/9780230226203.3826}

\bibitem[{Scheff{\'e} and Tukey(1945)}]{ScheffeTukey1945}
Scheff{\'e}, H., Tukey, J.~W., 1945. Non-parametric estimation. {I}. validation
  of order statistics. Annals of Mathematical Statistics 16~(2), 187--192.
\newline\urlprefix\url{https://projecteuclid.org/euclid.aoms/1177731119}

\bibitem[{Sedgewick and Wayne(2011)}]{SedgewickWayne2011}
Sedgewick, R., Wayne, K., 2011. Algorithms, 4th Edition. Addison-Wesley
  Professional.
\newline\urlprefix\url{https://books.google.com/books?id=MTpsAQAAQBAJ}

\bibitem[{Shen and Zhang(2016)}]{ShenZhang2016}
Shen, S., Zhang, X., 2016. Distributional tests for regression discontinuity:
  Theory and empirical examples. Review of Economics and Statistics 98~(4),
  685--700.
\newline\urlprefix\url{https://doi.org/10.1162/REST_a_00595}

\bibitem[{Smirnov(1948)}]{Smirnov1948}
Smirnov, N., 1948. Table for estimating the goodness of fit of empirical
  distributions. Annals of Mathematical Statistics 19~(2), 279--281.
\newline\urlprefix\url{http://www.jstor.org/stable/2236278}

\bibitem[{Smirnov(1939)}]{Smirnov1939}
Smirnov, N.~V., 1939. On the estimation of the discrepancy between empirical
  curves of distribution for two independent samples. Bulletin Math{\'e}matique
  de l'Universit{\'e} de Moscou 2~(2), 3--16.

\bibitem[{Stigler(1977)}]{Stigler1977}
Stigler, S.~M., 1977. Fractional order statistics, with applications. Journal
  of the American Statistical Association 72~(359), 544--550.
\newline\urlprefix\url{http://www.jstor.org/stable/2286215}

\bibitem[{van~der Vaart(1998)}]{vanderVaart1998}
van~der Vaart, A.~W., 1998. Asymptotic Statistics. Cambridge University Press,
  Cambridge.
\newline\urlprefix\url{https://books.google.com/books?id=UEuQEM5RjWgC}

\bibitem[{Wilks(1962)}]{Wilks1962}
Wilks, S.~S., 1962. Mathematical Statistics. Wiley, New York.

\bibitem[{Zhuo(2017)}]{Zhuo2017}
Zhuo, L., 2017. Nonparametric {Bayesian} inference on stochastic dominance,
  working paper.

\end{thebibliography}

\end{document}